\RequirePackage{fix-cm}
\documentclass{svjour3}
\usepackage{graphics,graphicx,epsf,subfigure,epstopdf}
\usepackage{amsfonts,amsmath,amsxtra,amscd,amssymb,bm, amsthm}
\usepackage{mathtools}
\usepackage[numbers]{natbib}

\usepackage{lmodern}
\usepackage{algorithm}
\usepackage{algpseudocode}
\usepackage[top=2.5cm,bottom=2.5cm,right=2.5cm,left=2.5cm]{geometry}
\usepackage[mathscr]{eucal}
\usepackage{color}
\numberwithin{equation}{section}
\usepackage{hyperref}

\DeclareMathOperator*{\argmin}{arg\,min}

\newcommand{\beq}{\begin{equation}}
\newcommand{\eeq}{\end{equation}}
\newcommand{\beqa}{\begin{eqnarray}}
\newcommand{\eeqa}{\end{eqnarray}}
\newcommand{\beqas}{\begin{eqnarray*}}
\newcommand{\eeqas}{\end{eqnarray*}}

\newtheorem{assumption}{Assumption}

\def\vgap{\vspace*{.1in}}

\usepackage{xcolor}

\usepackage[dvipsnames]{xcolor}

\title{
Stochastic Auto-conditioned Fast Gradient Methods with Optimal Rates}
\author{Yao Ji \and Guanghui Lan           \thanks{GL and YJ were partially supported by Air Force Office of Scientific Research grant FA9550-22-1-0447 and American Heart Association grant 23CSA1052735.}}
\institute{
    Y. Ji \and G. Lan \at
    H. Milton Stewart School of Industrial and Systems Engineering,
    Georgia Institute of Technology,
    225 North Ave, Atlanta, GA 30332, USA \\
    \email{yaoji@gatech.edu, george.lan@isye.gatech.edu}
}
\begin{document}
\maketitle
\begin{abstract}
Achieving optimal rates for stochastic composite convex optimization without prior knowledge of problem parameters remains a central challenge. In the deterministic setting, the auto-conditioned fast gradient method has recently been proposed to attain optimal accelerated rates without line-search procedures or prior knowledge of the Lipschitz smoothness constant, providing a natural prototype for parameter-free acceleration. However, extending this approach to the stochastic setting has proven technically challenging and remains open. Existing parameter-free stochastic methods either fail to achieve accelerated rates or rely on restrictive assumptions, such as bounded domains, bounded gradients, prior knowledge of the iteration limit, or strictly sub-Gaussian noise.
To address these limitations, we propose a stochastic variant of the auto-conditioned fast gradient method, referred to as stochastic AC-FGM. The proposed method is fully adaptive to the Lipschitz constant, the iteration limit, and the noise level, enabling both adaptive stepsize selection and adaptive mini-batch sizing without line-search procedures. Under standard bounded conditional variance assumptions, we show that stochastic AC-FGM achieves the optimal iteration complexity of $\mathcal{O}(1/\sqrt{\varepsilon})$ and the optimal sample complexity of $\mathcal{O}(1/\varepsilon^2)$.
\end{abstract}
\section{Introduction}
In this paper, we study a class of stochastic optimization problems of the form
\begin{align}\label{eqn:main-p}
\Psi^*\coloneqq\min\limits_{x\in X}\{\,\Psi (x)\coloneqq f(x)+h(x)\},
\end{align}
where $X\subseteq \mathbb{R}^n$ is a closed convex set, and $f: X\rightarrow \mathbb{R}$ is a closed convex and differentiable function with Lipschitz continuous gradients satisfying
\begin{align*}
    \|\nabla f(x)-\nabla f(y)\|\leq L\|x-y\|,\quad\forall \,\,x,\, y \in X.
\end{align*}
The function $h: X\rightarrow \mathbb{R}$ is a closed convex function, and its proximal mapping is efficiently computable.

Nesterov, in his celebrated work \cite{nesterov1983method}, introduced the accelerated gradient descent (AGD) method for solving \eqref{eqn:main-p} and showed that the number of iterations required by AGD to achieve $\Psi(\hat{x})-\Psi^* \leq \varepsilon$ is bounded by $\mathcal{O}(1/\sqrt{\varepsilon})$. {See also the classical developments by Nemirovski and Yudin in \cite{nemirovskij1983problem,nemirovskii1983information} and \cite{nemirovskii1985optimal} for optimal methods for Hölder smooth problems.}
When only noisy first-order information about $\Psi$ is available through subsequent calls to a stochastic oracle,
Lan \cite{lan2012optimal} proposed the stochastic accelerated gradient method (AC-SA) for solving \eqref{eqn:main-p} over bounded domains and established that, to find a point $\hat{x}$ satisfying $\mathbb{E}[\Psi(\hat{x})-\Psi^*] \leq \varepsilon$, {the sample complexity, i.e., the total number of calls to the stochastic oracle, is bounded by}
\begin{align}\label{eqn:L-known-case-optimal-sample}
  \mathcal{O}\left(\sqrt{\frac{LD_X^2}{\varepsilon}}+\frac{\sigma^2D_X^2}{\varepsilon^{2}}\right),
\end{align}
where $D_X$ is the diameter of the domain.
Later, in \cite{ghadimi2016accelerated}, they tackled the unbounded domain setting and incorporated mini-batches into AC-SA, while achieving the optimal iteration complexity $\mathcal{O}(1/\sqrt{\varepsilon})$ and optimal sample complexity
\begin{align}\label{eqn:L-known-case-optimal-sample-unbounded}
  \mathcal{O}\left(\sqrt{\frac{LD_0^2}{\varepsilon}}+\frac{\sigma^2D_0^2}{\varepsilon^{2}}\cdot\frac{D_0^2}{  {\tilde{D}^2}}\right),
\end{align}
where $D_0$ is an upper bound on the initial optimality gap and $\tilde{D}>0$ is a user defined quantity used in the mini-batch sizes. These guarantees are optimal up to constant factors in view of the classical lower bounds for smooth convex optimization due to Nemirovski and Yudin \cite{nemirovskij1983problem} and the lower bounds for stochastic convex optimization in \cite{agarwal2009information}.

However, achieving these optimal convergence rates requires knowledge of the smoothness parameter $L$, which is often difficult to estimate accurately; moreover, conservative estimates can dramatically slow down the algorithm.
Line-search procedures have long served as a classical mechanism for handling unknown problem parameters in first-order methods. In the deterministic setting, Nesterov \cite{nesterov1983method} incorporated a backtracking line-search procedure \cite{armijo1966minimization} into accelerated gradient methods for smooth optimization, and Beck and Teboulle extended this idea to the composite setting through FISTA \cite{beck2009fast}. Lan \cite{lan2015bundle} introduced the framework of \emph{uniformly optimal} methods. Building on the classical bundle-level method \cite{lemarechal1995new}, which does not require problem parameters in nonsmooth optimization, he developed several accelerated bundle-level methods that are uniformly optimal for convex optimization across smooth, weakly smooth, and nonsmooth settings \cite{lan2015bundle}. However, these bundle-level methods require solving a more complicated subproblem than AGD at each iteration, and their analysis also requires the feasible region $X$ to be bounded. To address these limitations, Nesterov \cite{nesterov2015universal} introduced a universal fast gradient method (FGM) by incorporating a novel line-search procedure and a smooth approximation scheme into AGD. He showed that this method attains uniformly optimal convergence rates for smooth, weakly smooth, and nonsmooth convex optimization problems, with the target accuracy as the only input. Although each iteration of FGM may be more expensive than that of AGD because of the line-search procedure, the total number of first-order oracle calls remains of the same order as that of AGD, up to an additive constant factor. In the stochastic setting, there is a vast literature on stochastic backtracking line-search methods. For representative works on stochastic line-search methods, see, for example, \cite{paquette2020stochastic,jin2024high,wang2025parameter,jiang2023adaptive,vaswani2025armijo} and the references therein for details. To the best of our knowledge, however, existing stochastic line-search methods do not establish the optimal sample complexity as shown in \eqref{eqn:L-known-case-optimal-sample-unbounded}. Moreover, at each iteration, line-search introduces an additional subroutine that requires extra evaluations of stochastic gradients, stochastic function values, or both until a termination condition is met,
thereby increasing the per-iteration cost.

The widespread use of first-order methods in data science and machine learning has sparked growing interest in easily implementable, parameter-free first-order methods with fast convergence guarantees. A notable line of research seeks to eliminate line-search procedures from first-order methods to reduce the per-iteration cost, and require only rough knowledge of the problem parameters to achieve fast convergence rates.

In the deterministic setting, many works have established convergence guarantees for gradient methods using auto-conditioned methods for smooth objectives \cite{malitsky2019adaptive,li2019convergence,orabona2023normalized,khaled2023dowg,malitsky2024adaptive,latafat2025adaptive}. However, it remained a longstanding open problem whether there exists an optimal first-order method for smooth optimization with an unknown Lipschitz constant $L$ that satisfies all of the following criteria: it does not assume a bounded feasible region and does not require line-search procedures \cite{orabona2023normalized,malitsky2019adaptive}. This problem was recently resolved by \cite{li2025simple}, where they proposed a novel first-order algorithm, called the auto-conditioned fast gradient method (AC-FGM), that achieves the optimal convergence rate $\mathcal{O}(1/\sqrt{\varepsilon})$ for smooth convex optimization \eqref{eqn:main-p} without knowing any problem parameters or resorting to line-search procedures. Later, \cite{suh2025adaptive} showed that the same stepsize policy as AC-FGM achieves the optimal convergence rate for adaptive AGD in the unconstrained case.

In the stochastic setting, there is a vast literature on auto-conditioned methods \cite{gupta2017unified, levy2017online, cutkosky2018black, carmon2022making, ivgi2023dog, khaled2023dowg, levy2017online,lan2024projected}. However, all the aforementioned auto-conditioned methods match at best the convergence rate of non-accelerated (sub)gradient descent in the worst case and therefore fail to achieve the optimal iteration complexity $\mathcal{O}(1/\sqrt{\varepsilon})$ and sample complexity \eqref{eqn:L-known-case-optimal-sample-unbounded}. Some works do achieve accelerated convergence rates; see, for example, \cite{cutkosky2019anytime,kavis2019unixgrad}. However, they either require the feasible domain $X$ to be bounded or assume a bounded gradient norm over an unbounded domain. This assumption may limit the applicability of the result, since even quadratic functions do not satisfy it. To the best of our knowledge, \cite{kreisler2024accelerated} is the only work that tackles the unbounded-domain setting with accelerated convergence rates. However, several caveats remain. Their analysis is limited to the sub-Gaussian case with high-probability guarantees and relies heavily on light-tail assumptions on the noise, and thus appears unable to handle the general bounded-variance setting in stochastic optimization. Since sub-Gaussian parameters are notoriously difficult to estimate in practice, whereas variance is often much easier to estimate, it is important to develop guarantees under the classical bounded-variance assumption. Furthermore, even in the sub-Gaussian case, the convergence rate and sample complexity are still not optimal. In particular, the iteration complexity is not of order $\mathcal{O}(1/\sqrt{\varepsilon})$. For the sample complexity, there is an additional error term $\max_{x\in \mathbb{B}(x^*,2d_0)}\sigma_x d_0/\varepsilon$, where $\sigma_x$ is the sub-Gaussian parameter at point $x$. Since this term takes a supremum over the entire ball rather than over the finitely many iterates actually visited by the algorithm, it can be much larger and dominate the final error.
{Moreover, though their analysis allows the use of mini-batches, it does not guarantee variance reduction and requires a bounded gradient norm assumption over an unbounded domain, and the stepsize depends on this uniform gradient norm bound and therefore may not be fully adaptive.}
Lastly, a proper choice of the stepsize requires fixing the number of iterations of the method in advance, which may be difficult to specify in the stochastic setting.

Therefore, in the stochastic setting, accelerated methods with optimal complexity over unbounded domains remain an open problem. In this paper, we develop the \emph{Stochastic Auto-conditioned Fast Gradient Method} (stochastic AC-FGM), an optimal parameter-free method that is adaptive to the Lipschitz smoothness constant, the iteration limit \(N\), and the underlying variance. The method permits both adaptive stepsize selection and adaptive mini-batch sizing, while achieving optimal iteration and sample complexity for \eqref{eqn:main-p} without assuming either a bounded domain or bounded gradients, and without resorting to stochastic line-search procedures.

Our contributions can be summarized as follows.
First, under the bounded conditional variance assumption, we show that, to obtain an $\varepsilon$-solution satisfying $\mathbb{E}[\Psi(x_N)-\Psi(x^*)]\leq \varepsilon$, stochastic AC-FGM requires
\begin{equation*}
   \mathcal{O}\left( \sqrt{\frac{\mathcal{L}D_0^2}{\varepsilon}\cdot\max\left\{\frac{v_{\max}}{v_0}, 1\right\}}\right)
\end{equation*}
iterations, where $\mathcal{L}$ is the largest sample Lipschitz smoothness parameter, $v_{\max}$ is the largest local conditional variance of the sample Lipschitz smoothness, $v_0$ is its initial value, and $D_0$ is the initial optimality gap. Furthermore, {the sample complexity is bounded by}
\begin{align*}
    \mathcal{O}\left(
\sqrt{\frac{\mathcal{L}D_0^2}{\varepsilon}\cdot\frac{v_{\max}}{v_0}}
+\frac{R_N^2\mathcal{L}^{2}D_0^2}{\varepsilon^{2}}\cdot\frac{D_0^2}{\tilde{D}^2}\cdot\frac{v^2_{\max}}{v_0^2}
\right),
\end{align*}
where $R_N$ characterizes the average variance-to-smoothness ratio over the iteration horizon $N$ and depends only on the trajectory, that is, on the finitely many points visited by the algorithm, and $\tilde{D}>0$ is an arbitrary user-defined quantity used in the mini-batch sizes. These bounds are optimal in their dependence on $\varepsilon$ for both the iteration complexity and the sample complexity \cite{nemirovskij1983problem,agarwal2009information}.
Second, by adding an anchored regularization term at each iteration, we remove the need to know the iteration limit $N$ while retaining the same iteration and sample complexity as in the case where $N$ is given in advance.
Third, we enlarge the underlying filtration to allow adaptive mini-batch sizes and incorporate variance estimation. It accommodates different variance estimators and remains adaptive to the Lipschitz constant, the total number of iterations $N$, and the local variances.

Lastly, under the additional light-tail assumption, stochastic AC-FGM requires
\begin{align*}
     \mathcal{O}\left( \sqrt{\frac{\hat{L}_{N}D_0^2}{\varepsilon}\cdot\max\left\{\frac{v^{\max}_{N+1}}{a_0}, 1\right\}}\right)
\end{align*}
iterations, where $\hat{L}_{N}$ is the largest Lipschitz smoothness parameter along the trajectory and $v^{\max}_{N+1}$ is the largest local sub-Gaussian parameter along the trajectory. Furthermore, the sample complexity is bounded by
\begin{align*}
\mathcal{O}\left(
\sqrt{\frac{\hat{L}_ND_0^2}{\varepsilon}\cdot\frac{v_{N+1}^{\max}}{v_0}}
+\frac{R_N^2\hat{L}_N^{2}D_0^2}{\varepsilon^{2}}\cdot\frac{D_0^2}{  {\tilde{D}^2}}\cdot\left(\frac{v_{N+1}^{\max}}{v_0}\right)^2
\right).
\end{align*}
Both the iteration complexity and the sample complexity match the in-expectation result in their dependence on $\varepsilon$ and are therefore optimal. Moreover, they depend only on the trajectory-dependent quantities $\hat{L}_{N}$ and $v^{\max}_{N+1}$, which are much smaller than the global quantities appearing in the existing literature.

The rest of this paper is organized as follows.
In \autoref{sec:Algorithm}, we present the stochastic AC-FGM method.
In \autoref{sec:main-results}, we present the main convergence results. In \autoref{proof-main}, we present the proofs for the main results.

\subsection{Notation and terminology}
We use $\|\cdot\|$ to denote the Euclidean norm in $\mathbb{R}^n,$ which is associated with the inner product $\langle \cdot, \cdot\rangle.$
For any real number $s,$ $\lceil s\rceil$ and $\lfloor s\rfloor$ denote the nearest integers to $s$ from above and below, respectively. Let $[m]\triangleq\{1,\dots,m\}$, with $m\in \mathbb{N}_{+}.$ We use the convention that $0/0=0.$
{Let $\xi_1,\dots,\xi_k$ be independent and identically distributed random variables on $(\Omega,\mathscr{B}, \mu)$. Set $\Omega_{[k]}\coloneqq \prod_{i=1}^k\Omega$ and define
\[
\mathscr{B}_{[k]}\coloneqq \sigma\!\left(\left\{A_1\times\cdots\times A_k:\; A_i\in\mathscr{B},\ i=1,\dots,k\right\}\right).
\]
We denote by $\mathbb{P}\coloneqq \prod_{i=1}^k \mu$ the corresponding product measure on $(\Omega_{[k]},\mathscr{B}_{[k]})$. For any sub-$\sigma$-algebra $\mathcal{G}\subseteq \mathscr{B}_{[k]}$, we write $\mathbb{E}_{\xi_i}[\cdot\mid\mathcal{G}]$ for the conditional expectation with respect to $\xi_i$, given $\mathcal{G}$.}
\section{Algorithm}\label{sec:Algorithm}

Consider the stochastic auto-conditioned fast gradient method (stochastic AC-FGM) in \autoref{alg:main-single-objective-stochastic}. We defer the detailed parameter choices to the theorems in the next section and only highlight their dependence on the main quantities here to clarify the algorithmic structure. For simplicity, we assume for now that the conditional variance quantities are known when querying a point $x_k$, to illustrate the main idea.

At each iteration $k$, given a batch size $m_k$ and a stepsize $\eta_k$, we run the algorithm as follows. Conditioned on the information available at the start of the iteration, we generate several fresh batches of independent and identically distributed random variables. These batches are independent of the past and serve distinct purposes in the algorithm.
\begin{algorithm}[!th]
\caption{Stochastic AC-FGM}\label{alg:main-single-objective-stochastic}
\begin{algorithmic}[1]
    \Require Initial point $x_0\in\mathbb{R}^n,$ and $y_0=x_0,$ and $\eta_1>0$,
    $\{\beta_k, \tau_k, \gamma_k\}_{k\geq 1}.$
    \State Compute the mini-batch size $m_{1}$ for the iterate update according to \eqref{eqn:m1}.
    \For{$k=1,\dots, $}
        \State Call the stochastic oracle $m_k$ times to obtain $G(x_{k-1}, \xi_{k, i}),$ $i\in[m_{k}],$ and set
        \begin{align}\label{eqn:stochastic-gradient}
            G_{k}=\frac{1}{m_{k}}\sum_{i=1}^{m_k}G(x_{k-1}, \xi_{k, i}).
        \end{align}
        \State Compute
        \begin{align}
            z_{k}&=\argmin\limits_{z\in X}\left\{\langle G_{k}, z\rangle+h(z) +\frac{1}{2\eta_k}\|y_{k-1}-z\|^2+\frac{\gamma_k}{2\eta_k}\|y_0-z\|^2\right\},\label{eqn:gradient-step-stochastic}\\
            x_k&=\frac{1}{1+\tau_k} z_k+\frac{\tau_k}{1+\tau_k} x_{k-1},\label{eqn:output-stochastic}\\
            y_{k}&=(1-\beta_k){y}_{k-1}+\beta_k z_k\label{output-center},
        \end{align}
        \State Compute the mini-batch size $n_{k}$ for $\bar{L}_k$ update
        according to \eqref{eqn:m2}.
        \State Compute the stepsize $\eta_{k+1}$ for the next iteration according to \eqref{final-eta}.
        \State Compute the mini-batch size $m_{k+1}$ for the iterate update according to \eqref{eqn:m1}.
    \EndFor
    \State \Return $x_{N}.$
\end{algorithmic}
\end{algorithm}

{\noindent \bf Iterate $z_k, x_k, y_k$ updates:}
Given $m_k$, we introduce the first type of batch, $\{\xi_{k,i}\}_{i=1}^{m_k}$, which is used to construct the stochastic gradient estimator $G_k$ in \eqref{eqn:stochastic-gradient}. {We adopt the convention that a group of observations is treated as one batch. The batch $\{\xi_{k,i}\}_{i=1}^{m_k}$ plays the same role as a standard mini-batch in classical stochastic optimization when the Lipschitz constant $L$ is known.}
Once $G_k$ is computed, together with the previously determined step size $\eta_k$ for the iteration, we update the search point $z_k$ via \eqref{eqn:gradient-step-stochastic}. {For simplicity, we assume that the anchored regularization satisfies $\gamma_k=0$ in order to illustrate the choices of the stepsize and mini-batch sizes, and we refer the reader to \autoref{sec:Adaptivity to the iteration limit} for the case in which $\gamma_k\neq 0$ is chosen appropriately to allow the algorithm to be adaptive to the iteration limit.}
Then, with the predefined weights $\tau_k = k/2$, we compute the output iterate $x_k$ through \eqref{eqn:output-stochastic}, where $x_k$ is a convex combination of the previous iterate $x_{k-1}$ and the search point $z_k$.

Furthermore, with the predefined parameter choice $\beta_k \equiv \beta \in (0,1)$ for all $k \geq 2$ and $\beta_1 = 0$, we update the moving-average center $y_k$ according to \eqref{output-center}. This point is a convex combination of the previous center $y_{k-1}$ and the search point $z_k$.
\medskip

{\noindent \bf Stepsize $\eta_{k+1}$
update:}
To move to the next iteration, we need to specify both the stepsize $\eta_{k+1}$ and the batch size $m_{k+1}$.
For the stepsize, we are inspired by the recent AC-FGM approach to stepsize design for adaptive accelerated optimization in the deterministic setting \cite{li2025simple}, where the local cocoercivity-based smoothness
\begin{equation}\label{eqn:determinstic}
    \bar{L}_k\coloneqq \frac{\|\nabla f(x_{k-1})-\nabla f(x_k)\|^2}{f(x_{k-1})-f(x_k)-\langle \nabla f(x_k), x_{k-1}-x_k\rangle}
\end{equation}
is defined. The method uses $\eta_{k+1}\propto k\bar{L}_k^{-1}$. This choice has been shown to yield accelerated convergence, first for AC-FGM in \cite{li2025simple}, and later for AGD in \cite{suh2025adaptive}.

In the stochastic setting, since the gradient is unknown, we need to define a stochastic counterpart of \eqref{eqn:determinstic}. This is highly nontrivial, as the stepsize is now a random variable, and one must construct an adaptive random estimator of the local cocoercivity-based smoothness while achieving the accelerated convergence rate. {Directly reusing the first batch $\{\xi_k\}_{m_k}$, or introducing one fresh batch and using it to construct both stochastic estimators of the numerator and denominator in \eqref{eqn:determinstic}, creates strong dependence between the iterate update and the smoothness estimator.}
This dependence propagates through the error analysis and prevents us from establishing the optimal convergence rate.

To address this issue, we introduce a second type of batches, $\{\bar{\xi}_{k,i}\}_{i=1}^{n_k}$ and $\{\hat{\xi}_{k,i}\}_{i=1}^{n_k}$, to estimate the denominator and numerator of the local smoothness quantity separately in a prescribed order and determine $\eta_{k+1}$. Specifically, we first reveal the batch $\{\bar{\xi}_{k,i}\}_{i=1}^{n_k}$, which is used to compute the stochastic gradient difference between the two consecutive iterates $x_{k-1}$ and $x_k$ as follows:
\begin{align}
    \Delta G(x_{k}, \bar{\xi}_{k})\coloneqq \frac{1}{n_{k}}\sum_{i=1}^{n_{k}}\left[G(x_{k}, \bar{\xi}_{k,i})-G(x_{k-1},\bar{\xi}_{k,i})\right].
\end{align}
Then, we reveal the second batch of this type, $\{\hat{\xi}_{k,i}\}_{i=1}^{n_{k}}$, which is used to compute the empirical first-order Taylor remainder
$T(x_{k}, \hat{\xi}_{k})$, defined by
\begin{equation}\label{eqn:empirical first-order taylor remainder}
 \begin{aligned}
T(x_{k}, \hat{\xi}_{k})
\coloneqq
\frac{1}{n_{k}}\sum_{i=1}^{n_{k}}\left[F(x_{k-1}, \hat{\xi}_{k,i})-F(x_{k}, \hat{\xi}_{k,i})-\langle G(x_{k}, \hat{\xi}_{k,i}), x_{k-1}-x_{k} \rangle\right].
    \end{aligned}
\end{equation}
We consider the case where the finite-sum function associated with $\hat{\xi}_k$ is convex, and hence $T(x_k, \hat{\xi}_k)\geq 0.$
Using these two quantities, we define the empirical local cocoercivity-based smoothness estimator as
\begin{align}\label{eqn:sample-wise-Lk-mini-batch}
   \bar{L}_{k}
\coloneqq \frac{\|\Delta G(x_{k}, \bar{\xi}_{k})\|^2}{2T(x_{k}, \hat{\xi}_{k})}.
\end{align}
Using $\bar{L}_{k}$, we define the stepsize recursively as follows. Let $\eta_1>0$ and define
\begin{align}\label{final-eta}
 \eta_2=\min\left\{\frac{1}{16\beta\bar{L}_1},\, 2(1-\beta)\eta_1\right\},\quad
 \eta_{k+1}=\min\left\{\frac{k}{16\bar{L}_{k}},\,\frac{(k+1)\eta_{k}}{k}\right\},
 \quad \forall\, k\ge 2.
\end{align}
Intuitively, \eqref{final-eta} shows how the local smoothness estimate $\bar L_k$ governs the adaptive stepsize choice. A large value of $\bar L_k$ corresponds to a highly curved local region around $x_k$, leading to a smaller stepsize $\eta_{k+1}$ and hence a more conservative update. In contrast, when $\bar L_k$ is small, the stepsize can grow, potentially at the rate $\mathcal{O}(k)$, thereby enabling acceleration. Thus, the method automatically adapts to the local geometry along the trajectory without requiring any knowledge of the global smoothness constant $L$.

\medskip
{\noindent \bf mini-batch $m_{k+1}, n_{k+1}$ updates:} {
With $\eta_{k+1}$ prepared, it remains to determine the mini-batch size $m_{k+1}$ used in \eqref{eqn:stochastic-gradient}, and the mini-batch size $n_{k+1}$ used to estimate $\bar{L}_{k+1}$ in \eqref{eqn:sample-wise-Lk-mini-batch} for the next stepsize update. We first consider the case in which the variance quantities for the in-expectation bounds (resp. the light-tail parameters for the high-probability bounds) are known, and defer the unknown-variance case to the next section. In particular, we assume that the conditional variance parameter $\sigma_k^2$ for the stochastic gradient estimator and the quantity $v_k^{\max}$ associated with the local smoothness estimator are available; see \autoref{assump:Bounded local variance} and \autoref{ass:subgaussian} for details. Since these quantities are known, they can be used directly to construct $m_{k+1}$, and hence there is no need to introduce a third type of batch to estimate the variance. We now specify the update rules for the mini-batch sizes $m_{k+1}$ and $n_{k+1}$.}

The batch size for the main update satisfies
\begin{align}\label{eqn:m1}
m_{k+1}=
\mathcal{O}\left(
\left\lceil
\frac{N(k+1)^2}{\bar{L}_{k}^2}\cdot\frac{\sigma_{k}^2}{\tilde{D}^2}
\right\rceil
\right),
\end{align}
where $\tilde{D}>0$ is arbitrary.
Observe that $m_{k+1}$  resembles the mini-batch structure in the parameter-known setting of AC-SA \cite{ghadimi2016accelerated}. In that setting, the global Lipschitz constant $L$ replaces the local quantity $\bar{L}_k$, and the global variance $\sigma^2$ replaces the local variance $\sigma_k^2$. Since $\bar{L}_k$ is a random variable, the mini-batch size $m_{k+1}$ is also random, unlike in the parameter-known case. Nevertheless, the rule remains fully adaptive: $m_{k+1}$ is determined entirely by previously observed quantities, namely $\bar{L}_k$ and $\sigma_k^2$.

After $m_{k+1}$ is determined, the updates in \eqref{eqn:gradient-step-stochastic}, \eqref{eqn:output-stochastic}, and \eqref{output-center} uniquely determine $z_{k+1}$, $x_{k+1}$, and $y_{k+1}$. We can then evaluate the local conditional variance quantity $\delta_{k+1}$ at $x_{k+1}$ and choose the next stepsize-update batch size $n_{k+1}$ as
\begin{align}\label{eqn:m2}
n_{k+1}
=
\mathcal{O}\left(\left\lceil\frac{N(k+1)^2}{\bar{L}_{k}^2}
\cdot\max\left\{v^{\max}_{k}, \frac{\delta_{k+1}^2+\sigma_{k}^2}{\tilde{D}^2}\right\}\right\rceil\right).
\end{align}
Observe that $n_{k+1}$ depends not only on the gradient noise levels at $x_{k+1}$ and the previous point, measured by $\delta_{k+1}^2$ and $\sigma_k^2$, but also on the variability of the local smoothness estimator from the previous iteration, measured by $v_k^{\max}$. This additional dependence is intrinsic to the parameter-free setting. {In the parameter-known case, the current stepsize $\eta_{k+1}$ is deterministic and depends only on the global Lipschitz constant $L$, whereas here it depends on the random local Lipschitz constant $\bar{L}_k$. Consequently, the algorithm must control not only the stochastic gradient noise, but also the additional variability induced by the random stepsize. Thus, the mini-batch size $n_{k+1}$ used to determine $\bar{L}_{k+1}$ must be sufficiently large to control this additional source of variability.}
We refer the reader to \autoref{remark:mini-batch} for a detailed discussion of this choice. Although no batch is needed to compute the stepsize in the parameter-known case, this rule still resembles the mini-batch structure of AC-SA \cite{ghadimi2016accelerated}; the differences reflect the new challenges of the parameter-unknown setting. Moreover, this rule remains fully adaptive, since $n_{k+1}$ is determined entirely by previously observed quantities: $v_k^{\max}$, $\bar{L}_k$, $\delta_{k+1}^2$, and $\sigma_k^2$.

When the conditional variance quantities $\sigma_x^2$, $\delta_x^2$, and $v_x$ at a given point $x$ are unknown, we introduce Type III fresh batches to determine the mini-batch sizes $m_{k+1}$ for the main update and $n_{k+1}$ for the next stepsize update. We refer the reader to \autoref{sec:Adaptivity to the local variances} for the construction of the third type of batches used to estimate these quantities.
\medskip

Up to this point, the algorithm has been fully introduced. Note that our analysis is carried out under the filtration induced by these observations. For now, we define the natural filtration $\{\mathcal{F}_k\}_{k\ge 0}$ as follows. Later, when we introduce the third type of batches for variance estimation, we will slightly abuse notation and use the same symbol for the augmented filtration. See \autoref{fig:filtration} for an illustration of the filtration.
\begin{equation}\label{def:filtration}
\begin{aligned}
\mathcal{F}_{0} &\coloneqq \sigma(\emptyset),\\
\mathcal{F}_{k-\frac{2}{3}} &\coloneqq \sigma\!\left(\mathcal{F}_{k-1}, \{\xi_{k,i}\}_{i=1}^{m_k}\right),\\
\mathcal{F}_{k-\frac{1}{3}} &\coloneqq \sigma\!\left(\mathcal{F}_{k-\frac{2}{3}}, \{\bar{\xi}_{k,i}\}_{i=1}^{n_k}\right),\\
\mathcal{F}_{k} &\coloneqq \sigma\!\left(\mathcal{F}_{k-\frac{1}{3}}, \{\hat{\xi}_{k,i}\}_{i=1}^{n_k}\right),
\end{aligned}
\qquad k\ge 1.
\end{equation}
\begin{figure}[!ht]
 \centering
\includegraphics[width=0.99\linewidth]{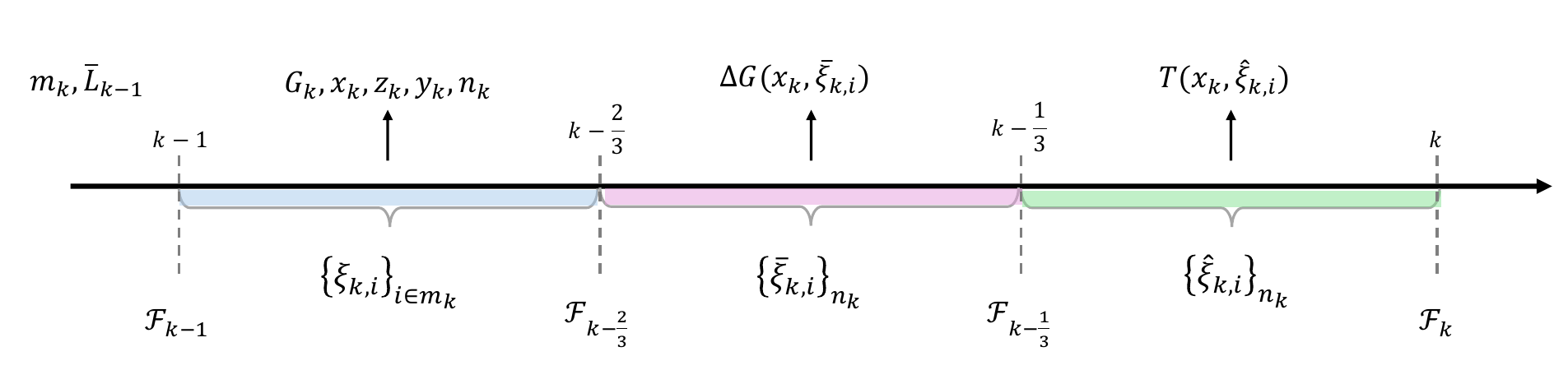}
 \caption{Illustration of the filtration $\{\mathcal{F}_{k}\}_{k\in \mathbb{N}_{+}}$ and the intermediate $\sigma$-algebras $\mathcal{F}_{k-\frac{1}{3}}$ and $\mathcal{F}_{k-\frac{2}{3}}$.}
 \label{fig:filtration}
 \vspace{-0.3cm}
\end{figure}

Moreover, we define the filtration generated by the iterates as
\begin{align}\label{eqn:G-k}
    \mathcal{G}_k \coloneqq \sigma(x_0,\dots,x_{k}).
\end{align}
By construction, $\mathcal{G}_{k} \subseteq \mathcal{F}_{k-\frac{2}{3}}$ for all $k \ge 1$. From the definition of the filtration in \eqref{def:filtration} and the definitions of $x_k$, $z_k$, and $y_k$ in \autoref{alg:main-single-objective-stochastic}, it is straightforward to verify that $x_k, z_k \in \mathcal{F}_{k-\frac{2}{3}}$. Moreover, the mini-batch sizes $m_k$ and $n_k$ may be chosen as random variables such that $m_k \in \mathcal{F}_{k-1}$ and $n_k \in \mathcal{F}_{k-\frac{2}{3}}$. Furthermore, we have $\|\Delta G(x_k,\bar{\xi}_k)\|^2 \in \mathcal{F}_{k-\frac{1}{3}}$, $T(x_{k}, \hat{\xi}_{k})\in \mathcal{F}_k$, and hence $\bar{L}_k, \eta_{k+1} \in \mathcal{F}_k$. Therefore, both the stepsize and the mini-batch sizes are random while remaining fully adaptive.

\section{Main Results}\label{sec:main-results}
{We consider the stochastic optimization problem \eqref{eqn:main-p}, where only noisy zeroth- and first-order information about $\Psi$ is available through subsequent calls to a stochastic oracle $(\mathcal{SO})$.}
We assume that, for any $x\in X$, a call to $\mathcal{SO}$ returns unbiased estimators of the function value and gradient, denoted by $F(x,\xi)$ and $G(x,\xi)$, respectively.
\begin{assumption}[Conditional unbiased estimator]
\label{a:unbiasness}
Given the current iterate $x_k$, the following hold.
\noindent
For a main update observation $\xi_{k+1}$, we have
\begin{align}\label{eqn: usual-unbiasedness}
\mathbb{E}_{\xi_{k+1}}\!\left[G(x_k,\xi_{k+1}) \mid \mathcal{F}_k\right]
= \nabla f(x_k).
\end{align}
\noindent
For stepsize selection observations $\bar{\xi}_k$ and $\hat{\xi}_k$, we have
\begin{align}
\mathbb{E}_{\bar{\xi}_k}\!\left[G(x_k,\bar{\xi}_k)\mid \mathcal{F}_{k-\frac{2}{3}}\right]
&= \nabla f(x_k), \,\, {\mathbb{E}_{\bar{\xi}_k}\!\left[G(x_{k-1},\bar{\xi}_k)\mid \mathcal{F}_{k-\frac{2}{3}}\right]
= \nabla f(x_{k-1})}, \label{unbiased-conditional}\\
\mathbb{E}_{\hat{\xi}_k}\!\left[G(x_k,\hat{\xi}_k)\mid \mathcal{F}_{k-\frac{1}{3}}\right]
&= \nabla f(x_k),\,\,\mathbb{E}_{\hat{\xi}_k}\!\left[F(x_k,\hat{\xi}_k)\mid \mathcal{F}_{k-\frac{1}{3}}\right]
= f(x_k), \,\,
\mathbb{E}_{\hat{\xi}_k}\!\left[F(x_{k-1},\hat{\xi}_k)\mid \mathcal{F}_{k-\frac{1}{3}}\right]
= f(x_{k-1}).\label{unbiased-conditional-new}
\end{align}
\end{assumption}
This assumption is natural because all stochastic quantities appearing in the algorithm are built from fresh observations and are used to estimate the true gradient or function value at the current iterate. \eqref{eqn: usual-unbiasedness} is the usual unbiasedness assumption for the stochastic gradient in the main update. The remaining conditions simply require that the fresh observations used for adaptive stepsize selection also provide conditionally unbiased estimators of the corresponding gradient and function values.

We emphasize that the algorithm itself only requires the local cocoercivity parameter $\bar{L}_{k-1}$ for adaptive stepsize selection. This makes it a natural choice in the stochastic setting, since it is directly computable from sample-wise first-order information and admits a clear interpretation. However, its analysis is nontrivial. Indeed, even with sample splitting, $\bar{L}_{k-1}^{-1}$ is generally not an unbiased estimator of its {deterministic counterpart defined in \eqref{eqn:determinstic}}. Nevertheless, as shown in the following lemma, with the aid of the filtration design $\{\mathcal{F}_k\}$, the induced error can still be controlled through the fluctuation of the sample Taylor remainder around its conditional mean. This motivates the introduction of the population local smoothness $L_{k-1}$ and its sample-wise counterpart $\ell_{k-1}(\hat{\xi})$. While $L_{k-1}$ is not directly computable, $\ell_{k-1}(\hat{\xi})$ is sample-based and serves as an unbiased estimator of $L_{k-1}$, constructed from a fresh sample $\hat{\xi}$:
\begin{equation}\label{eqn:L-k}
\begin{aligned}
L_{k-1}
&\coloneqq
\frac{
2[f(x_{k-2})-f(x_{k-1})
-\langle \nabla f(x_{k-1}),\,x_{k-2}-x_{k-1}\rangle]
}{\|x_{k-1}-x_{k-2}\|^2},\\
\ell_{k-1}(\hat{\xi})
&\coloneqq
\frac{2[
F(x_{k-2},\hat{\xi})-F(x_{k-1},\hat{\xi})
-\langle G(x_{k-1},\hat{\xi}),\,x_{k-2}-x_{k-1}\rangle]
}{\|x_{k-1}-x_{k-2}\|^2}.
\end{aligned}
\end{equation}
Similarly, when $x_{k-1}=x_{k-2},$ {we define $\ell_{k-1}=L_{k-1}=0/0=0$.
Furthermore, define
\begin{align}\label{eqn:bar-L-k}
    \tilde{L}_{k-1}=\tfrac{1}{n_{k-1}}\textstyle\sum_{i=1}^{n_{k-1}} \ell_{k-1}(\hat{\xi}_{k-1, i}),
\end{align}
where $\ell_{k-1}(\hat{\xi}_{k-1, i})$ is defined in \eqref{eqn:L-k}.}
\begin{lemma}\label{lem:bias-local cocoercivity-based smoothnes}
    Suppose \autoref{a:unbiasness} holds, $n_{k-1}\in \mathcal{F}_{k-\frac{5}{3}}$, and $\bar{L}_{k-1}\neq 0$. Then, it holds that
    \begin{align*}
\bar{L}_{k-1}^{-1}
-\mathbb{E}_{\hat{\xi}_{k-1}}\!\left[\bar{L}_{k-1}^{-1}\mid\mathcal{F}_{k-\frac{4}{3}}\right]
=
\frac{(\tilde{L}_{k-1}-L_{k-1})
    \|x_{k-1}-x_{k-2}\|^2}{\left\|\Delta G(x_{k-1}, \bar{\xi}_{k-1})\right\|^2}.
\end{align*}
\end{lemma}
\begin{proof}
Observe that $\Delta G(x_{k-1}, \bar{\xi}_{k-1})\in\mathcal{F}_{k-\frac{4}{3}}$, since
$n_{k-1}\in\mathcal{F}_{k-\frac{5}{3}}\subseteq \mathcal{F}_{k-\frac{4}{3}}$,
$x_{k-1}$ and $x_{k-2}$ are $\mathcal{F}_{k-\frac{5}{3}}$-measurable, and
$\{\bar{\xi}_{k-1,i}\}_{i\in[n_{k-1}]}$ is $\mathcal{F}_{k-\frac{4}{3}}$-measurable.

By the definitions of $\bar{L}_{k-1}$ in \eqref{eqn:sample-wise-Lk-mini-batch} and $\mathcal{F}_{k-\frac{4}{3}}$ in \eqref{def:filtration}, we have
\begin{align}\label{eqn:unbiased}
\mathbb{E}_{\hat{\xi}_{k-1}}[\bar{L}_{k-1}^{-1}\,|\,\mathcal{F}_{k-\frac{4}{3}}]
&\overset{\text{(i)}}{=}\frac{2[f(x_{k-2})- f(x_{k-1})-\langle \nabla f(x_{k-1}), x_{k-2}-x_{k-1}\rangle]}{\left\|\Delta G(x_{k-1}, \bar{\xi}_{k-1})\right\|^2},
\end{align}
where in \textnormal{(i)}, we used the fact that
$\Delta G(x_{k-1}, \bar{\xi}_{k-1})\in\mathcal{F}_{k-\frac{4}{3}}$ together with the unbiased estimator
assumptions \eqref{unbiased-conditional} and \eqref{unbiased-conditional-new}.
Thus, we have
\begin{align*}
\bar{L}_{k-1}^{-1}-\mathbb{E}_{\hat{\xi}_{k-1}}[\bar{L}_{k-1}^{-1}\,|\,\mathcal{F}_{k-\frac{4}{3}}]
&\overset{\textnormal{(ii)}}{=}\frac{(\tilde{L}_{k-1}-L_{k-1})
\|x_{k-1}-x_{k-2}\|^2}{\left\|\Delta G(x_{k-1}, \bar{\xi}_{k-1})\right\|^2},
\end{align*}
where in \textnormal{(ii)}, we used \eqref{eqn:unbiased} and the definitions of $L_{k-1}$ and $\tilde{L}_{k-1}$.
\end{proof}
Since $n_k$ is the random mini-batch size used to form the averaged stochastic objective for estimating the local smoothness in \eqref{eqn:sample-wise-Lk-mini-batch}, it is natural to require the relevant regularity only for this mini-batch objective at the query pair $(x_{k-1},x_k)$. This is captured by the following finite-sample cocoercivity--smoothness inequality.
\begin{assumption}[Finite-sample cocoercivity--smoothness condition]\label{ass:finite sample convexity}
For a query pair $(x_{k-1},x_k)$, there exist positive random variables
$\mathcal{L}(\bar{\xi}_{k},\hat{\xi}_{k})$ and
$\mathcal{L}(\hat{\xi}_{k})$ such that
\begin{equation}\label{eqn:sample-convexity-smoothness}
\begin{aligned}
\frac{\left\| \Delta G(x_k, \bar{\xi}_k)\right\|^2}
{2\mathcal{L}(\bar{\xi}_{k},\hat{\xi}_{k})}
\leq
T(x_k, \hat{\xi}_k)
\leq
\frac{\mathcal{L}(\hat{\xi}_{k})}{2}\left\|x_k-x_{k-1}\right\|^2,
\quad \text{a.s.}
\end{aligned}
\end{equation}
where $T(x_{k}, \hat{\xi}_{k})$ is the empirical first-order Taylor remainder defined in \eqref{eqn:empirical first-order taylor remainder}.
\end{assumption}
Observe that \autoref{ass:finite sample convexity} can be viewed as a finite-sample analogue of the standard upper and lower bounds for the first-order Taylor remainder of a smooth convex function. In particular, \eqref{eqn:sample-convexity-smoothness} requires the empirical first-order Taylor remainder to be bounded below by the squared empirical gradient difference and above by a term proportional to $\|x_k-x_{k-1}\|^2$. These correspond to the classical cocoercivity lower bound and smoothness type upper bound for deterministic smooth convex functions.

Let $\mathcal{L}$ be a deterministic upper bound on $\mathcal{L}(\bar{\xi}_{k},\hat{\xi}_{k})$. Then
$
\bar{L}_k \le \mathcal{L}.
$
Moreover, $\bar{L}_k$ in \eqref{eqn:sample-wise-Lk-mini-batch} is well defined: if
$
T(x_k,\hat{\xi}_k)=0,
$
then \eqref{eqn:sample-convexity-smoothness} implies
$$
\|\Delta G(x_k,\bar{\xi}_k)\|=0,
$$
so we define the resulting $0/0$ ratio as $0$.
Denote by $\mathcal{N}^k\subseteq \Omega^k$ the null set on which \eqref{eqn:sample-convexity-smoothness} fails to hold.

In the following subsections, we present the main convergence results. We delegate the proofs to \autoref{proof-main}.

\subsection{In-expectation guarantee}
In this subsection, we establish convergence results under the bounded conditional variance assumption.

At each iterate, the stochastic gradients associated with the different sampling streams are assumed to have bounded conditional variance with respect to the available information. Moreover, the random local smoothness estimator is assumed to have bounded conditional variance.
\begin{assumption}[Locally bounded conditional variances]\label{assump:Bounded local variance}
Given the current iterate $x_k$, there exists $\sigma_k \ge 0$ such that, for a  main-update sample $\xi_{k+1}$, and a stepsize-selection sample $\bar{\xi}_{k+1},$
\begin{align}
\mathbb{E}_{\xi_{k+1}}\!\left[\|G(x_k,\xi_{k+1})-\nabla f(x_k)\|^2 \,\middle|\, \mathcal{F}_k\right]
\le \sigma_k^2,\quad \mathbb{E}_{\bar{\xi}_{k+1}}\!\left[\|G(x_k,\bar{\xi}_{k+1})-\nabla f(x_k)\|^2 \,\middle|\, \mathcal{F}_k\right]
\le \sigma_k^2.
\label{eqn:local-var}
\end{align}
Similarly, there exists $\delta_k \ge 0$ such that, for a stepsize-selection sample $\bar{\xi}_k$,
\begin{align}
\mathbb{E}_{\bar{\xi}_k}\!\left[\|G(x_k,\bar{\xi}_k)-\nabla f(x_k)\|^2 \,\middle|\, \mathcal{G}_k\right]
\le \delta_k^2.
\label{eqn:local-var-2}
\end{align}
Furthermore, there exists $v_k \ge 0$ such that, for a stepsize-selection sample $\hat{\xi}_k$,
\begin{align}
\mathbb{E}_{\hat{\xi}_k}\!\left[|\ell_k(\hat{\xi}_k)-L_k|^2 \,|\, \mathcal{F}_{k-\frac{2}{3}}\right]
\le v_k,
\label{eqn:definition-v_k}
\end{align}
where $\ell_k(\cdot)$ and $L_k$ are defined in \eqref{eqn:L-k}.
\end{assumption}
Observe that, in classical stochastic optimization with convergence results stated in expectation, it is typically sufficient to assume only \eqref{eqn:local-var} as the bounded-variance condition. In the parameter-free setting, however, the stepsize is random. Consequently, beyond the classical gradient noise, one must also control the error induced by the random stepsize. This motivates the additional assumptions in \eqref{eqn:local-var-2} and \eqref{eqn:definition-v_k}, together with the adaptive choice of mini-batch sizes introduced later.

In \eqref{eqn:local-var}, the quantity $\sigma_k^2$ represents the classical stochastic error induced by the sample $\xi_{k+1}$ in the iterate update; see \eqref{eqn:gradient-step-stochastic}. More precisely, it is the conditional local gradient variance at the iterate $x_k$ given $\mathcal{F}_k$.

In \eqref{eqn:local-var-2}, the quantity $\delta_k^2$ characterizes the error associated with the stepsize-selection sample $\bar{\xi}_k$ in the update of the stepsize $\eta_{k+1}$. More precisely, it is the conditional variance at $x_k$ with respect to the filtration $\mathcal{G}_k$. Note that $\mathcal{G}_k \neq \mathcal{F}_k$. Indeed, since
\[
\mathcal{G}_k \subseteq \mathcal{F}_{k-\frac{2}{3}} \subseteq \mathcal{F}_k,
\]
it follows in general that $\sigma_k^2 \neq \delta_k^2$. Moreover, we emphasize that $\sigma_k^2$ is not available when choosing $n_k$, because $n_k \in \mathcal{F}_{k-\frac{2}{3}}$, whereas $\sigma_k^2$ is determined only at time $k$, that is, from information accumulated up to $\mathcal{F}_k$. By contrast, $\delta_k^2$ can be used to construct $n_k$, since
\[
\delta_k^2 \in \mathcal{G}_k \subseteq \mathcal{F}_{k-\frac{2}{3}}.
\]
Thus, choosing $n_k$ as a function of $\delta_k^2$ preserves its adaptivity. In particular, we will show that, in order to achieve accelerated convergence, $n_k$ should be chosen proportionally to $\delta_k^2$, so that the error induced by the randomness of the stepsize can be properly controlled.

In \eqref{eqn:definition-v_k}, the quantity $v_k$ characterizes the error associated with the stepsize-selection sample $\hat{\xi}_k$ in the update of $\eta_{k+1}$, which is a new feature of accelerated parameter-free optimization. As shown in \autoref{lem:bias-local cocoercivity-based smoothnes}, although the algorithm uses only the local cocoercivity quantity $\bar{L}_k$, the analysis must also control the fluctuation of the corresponding sample local smoothness quantity $\tilde{L}_k$ around its target value. This observation motivates introducing deterministic lower and upper bounds on the variance of $\tilde{L}_k$ along the iterate trajectory. Specifically, let $0 \leq v_{\min} \leq v_{\max}$ denote such nontrivial bounds, that is,
\begin{align}\label{eqn:variance-bounds-local-smoothness}
0 \leq v_{\min}
\;\le\;
v_k
\;\le\;
v_{\max}.
\end{align}

\autoref{assump:Bounded local variance} is mild and easy to satisfy in practice. In \autoref{sec:Adaptivity to the local variances}, we introduce a third type of fresh batch for variance estimation and modify the filtration accordingly, thereby enabling variance estimation while preserving the adaptive properties of the stepsize and the mini-batch sizes.

\subsubsection{Adaptivity to the Lipschitz constant}
We begin with a baseline setting in which, at each iteration $k$, the local conditional variances and the iteration limit $N$ are known, while the Lipschitz constant $L$ is unknown. In this regime, we show that stochastic AC-FGM can adaptively choose both the stepsizes and the mini-batch sizes while attaining the optimal accelerated convergence rate and a nearly optimal sample complexity comparable to those obtained under the assumption that the global smoothness constant $L$ is known; see, for example, AC-SA \cite{ghadimi2016accelerated}.

We first introduce several quantities that appear in the convergence rate. In particular, we choose an arbitrary nonnegative number $v_0>0$, and define the largest local conditional variance of the sample Lipschitz smoothness along the trajectory up to iteration $k-1$ by
\begin{align}\label{eqn:largest-var-lip}
 v^{\max}_{k-1} \coloneqq \max_{0\le i\le k-1} v_i,
\end{align}
where $\{v_i^2\}_{i\leq k-1}$ are the local conditional variance parameters from \autoref{ass:finite sample convexity}. We define the universal constants $c $ and $\tilde{c}$ by
\begin{align}\label{eqn:constant}
  c \coloneqq 73,\quad \tilde{c} \coloneqq 1728.
\end{align}
We define the initial optimality gap $D_0$ by
\begin{align}\label{eqn:chocie D}
  D^2_0\coloneqq 36\eta_1^2\left\| {\nabla f(x_0)+s_0}\right\|^2
+18\left(\min\limits_{x^*\in X^*}\|x^*-x_{0}\|^2+\tilde{D}^2\right),
\end{align}
where $s_0\in\partial h(x_0),$ and $\tilde{D}>0$ is arbitrary.
Recall that the gradient mapping $\mathcal{P}(u, y, c)$ and the corresponding reduced gradient $\mathcal{G}(x, y, c)$ are defined as \cite{nemirovskij1983problem}
\begin{equation}\label{eqn:reduced-gradient}
\begin{aligned}
\mathcal{P}(u, y, c)&\coloneqq\argmin\limits_{x\in X}\left\{\left\langle y, x\right\rangle+h(x)+\frac{1}{2c}\|x-u\|^2\right\},\\
\mathcal{G}(x, y, c)&\coloneqq \frac{1}{c}[x-\mathcal{P}(x, y, c)].
\end{aligned}
\end{equation}
        We now state the corresponding convergence guarantee.
\begin{theorem}\label{cor:main-fixed-T-const}
Suppose that \autoref{a:unbiasness}, \autoref{ass:finite sample convexity}, and \autoref{assump:Bounded local variance} hold. Assume further that $\gamma_k \equiv 0$ and $\tau_k = \frac{k}{2}$ for all $k \ge 1$. Let $\beta_1 = 0$ and $\beta_k \equiv \beta \in \left(0,\frac{1}{8}\right]$ for all $k \ge 2$. In addition, let $\eta_1 > 0$ and define
\begin{align}\label{eqn:eta-cor1-const}
\eta_2
=
\min\left\{
\frac{1}{16\bar{L}_1},\,
2(1-\beta)\eta_1,\,
\frac{2\eta_1}{\beta}
\right\},\quad
\eta_k
=
\min\left\{
\frac{k-1}{16\bar{L}_{k-1}},\,
\frac{k\eta_{k-1}}{k-1}
\right\},
\qquad \forall\, k \ge 3.
\end{align}
Furthermore, for all $k \ge 1$, suppose that the mini-batch sizes satisfy
\begin{align}
m_k
&=
\max\left\{
1,\,
\frac{(N+2)\eta_k^2}{\beta^2}\cdot \frac{c\sigma_{k-1}^2}{\tilde{D}^2}
\right\},
\label{eqn:batch-size-cor1-const}
\\
n_k
&=
\max\left\{
1,\,
\frac{\tilde{c}(N+2)\eta_k^2 v_{k-1}^{\max}}{\beta^3},\,
\frac{(N+2)\eta_k^2}{\beta^2}\cdot \frac{c(\sigma_{k-1}^2+\delta_k^2)}{\tilde{D}^2}
\right\},
\label{eqn:n-k}
\end{align}
for any $\tilde{D} > 0$. Then, for all $N \ge 1$, it holds that
\begin{equation*}
\begin{aligned}
\mathbb{E}\bigl[\Psi(x_N)-\Psi(x^*)\bigr]
&\le
\frac{32\mathcal{L}D_0^2}{\beta N^2}
\max\left\{\frac{v_{\max}}{v_0},1\right\},
\\[0.5ex]
\min_{x^*\in X^*}\mathbb{E}\bigl[\|y_{N+1}-x^*\|^2\bigr]
&\le
D_0^2
\max\left\{\frac{v_{\max}}{v_0},1\right\},
\\[0.5ex]
\min\limits_{2\leq k\leq N}\mathbb{E}\bigl[\|\mathcal{G}(y_k,\nabla f(x_k),\eta_{k+1})\|^2\bigr]
&\le
\frac{4096\mathcal{L}^2D_0^2}{\beta^2N^3}
\max\left\{\frac{v_{\max}}{v_0},1\right\},
\end{aligned}
\end{equation*}
where $D_0$ is defined in \eqref{eqn:chocie D}.
\end{theorem}
We add a few observations about \autoref{cor:main-fixed-T-const}. First, in view of \eqref{eqn:eta-cor1-const}, $\eta_k$ depends only on the previous stepsize $\eta_{k-1}$ and on $\bar{L}_{k-1}$, both of which are $\mathcal{F}_{k-1}$-measurable. Hence, $\eta_k$ is fully adaptive.

Similarly, recall that the batch size $n_k$ must be chosen without using any future information beyond $\mathcal{F}_{k-\frac{2}{3}}$, since it is used to compute $\bar{L}_k$ in \eqref{eqn:sample-wise-Lk-mini-batch}, which in turn determines $\eta_{k+1}$.  This requirement is satisfied by \eqref{eqn:n-k}, since $n_k$ depends on $\eta_k$ and $v_{k-1}^{\max}$, both of which are $\mathcal{F}_{k-1}$-measurable: indeed, $\eta_k\in\mathcal{F}_{k-1}$ and $v_{k-1}^{\max}\in\mathcal{F}_{k-\frac{5}{3}}\subseteq \mathcal{F}_{k-1}$. Moreover, by \eqref{eqn:local-var-2} and the definition of the filtration $\mathcal{G}_k$ in \eqref{eqn:G-k}, it follows that
\begin{equation*}
\delta_k \in\mathcal{G}_k\subseteq \mathcal{F}_{k-\frac{2}{3}},\quad\text{and}\quad \sigma_{k-1}\in \mathcal{F}_{k-1}\subseteq \mathcal{F}_{k-\frac{2}{3}}.
\end{equation*}
Therefore, $n_k$ is fully adaptive. By a similar argument, $m_k\in\mathcal{F}_{k-1}$, and hence $m_k$ is also fully adaptive.
\begin{remark}\label{remark:mini-batch}
By substituting $\eta_k$ into $m_k$ and $n_k$, we obtain
\begin{align}\label{eqn:m-1}
m_k&=\mathcal{O}\left(\left\lceil\frac{Nk^2}{\bar{L}_{k-1}^2} \cdot\frac{\sigma_{k-1}^2}{  {\tilde{D}^2}}\right\rceil\right),\quad n_k=\mathcal{O}\left(\left\lceil\frac{Nk^2}{\bar{L}_{k-1}^2} \max\left\{v^{\max}_{k-1}, \frac{\delta_{k}^2+\sigma_{k-1}^2}{  {\tilde{D}^2}}\right\}\right\rceil\right).
\end{align}
This choice closely resembles the sample size used to obtain the optimal convergence rate when the Lipschitz constant $L$ is known; see AC-SA \citep[Corollary 5]{ghadimi2016accelerated}. In the parameter known case, AC-SA requires a batch size of
\begin{align}\label{eqn:m-2}
m_k=\mathcal{O}\left(\frac{Nk^2\sigma^2}{L^2  {\tilde{D}^2}}\right).
\end{align}
Compared with \eqref{eqn:m-2}, the main update batch $m_k$ in stochastic AC-FGM requires only the local cocoercivity-based smoothness estimator $\bar{L}_{k-1}$ and the local variance $\sigma_{k-1}^2$ (resp. $L$ and $\sigma^2$ for AC-SA), and is therefore random. Furthermore, the additional batch size $n_k$ used to compute the next stepsize is new, and it depends not only on the variance of the stochastic gradient, namely $\delta_k^2$ and $\sigma_{k-1}^2$, but also on the variability of the local smoothness estimator, captured by $v_{k-1}^{\max}$ (cf. \eqref{eqn:largest-var-lip}). This additional dependence is intrinsic to the parameter-free setting, where the stepsize is random. Consequently, the batch size must also control the variability induced by the random stepsize.
More precisely, $v_{k-1}^{\max}$ controls the bias of the estimator $\bar{L}_{k-1}^{-1}$; see \autoref{lem:bias-local cocoercivity-based smoothnes}. Thus, the appearance of $v_{k-1}^{\max}$ reflects a key feature of the parameter-free case relative to the known-$L$ setting.
\end{remark}

In view of \autoref{cor:main-fixed-T-const}, we can derive the iteration complexity of stochastic AC-FGM. To obtain an $\varepsilon$-solution satisfying $\mathbb{E}[\Psi({x}_N)-\Psi(x^*)]\leq \varepsilon$, stochastic AC-FGM requires at most
\begin{equation}\label{eqn:iteration complexity}
   \mathcal{O}\left( \sqrt{\frac{\mathcal{L}\left(\eta_1^2\left\| {\nabla f(x_0)+s_0}\right\|^2
+\min\limits_{x^*\in X^*}\|x^*-x_{0}\|^2+\tilde{D}^2\right)}{\varepsilon}}\cdot\max\left\{\sqrt{\frac{v_{\max}}{v_0}}, 1\right\}\right)
\end{equation}
iterations. Thus, it achieves the optimal iteration complexity in terms of $\varepsilon$, matching that of AC-SA \citep[Corollary 5]{lan2012optimal} and meeting the lower bound in \cite{nemirovskij1983problem}.
It is worth noting that the initial optimality gap depends on the squared initial gradient norm $\eta_1^2\|\nabla f(x_0)+s_0\|^2$, which does not appear in the parameter-known case. In the parameter-free case, $L$ is unknown, and the initial step needs to be chosen carefully, always by line-search to ensure convergence. For example,
in the deterministic case, for AC-FGM \cite{li2025simple}, one needs to perform an initial line-search step in order to derive an error bound that depends solely on the distance to the optimal solution $\|x^*-x_0\|$. Here, in the stochastic case, we eliminate the need for this initial line-search step, which allows for arbitrary positive initial stepsize $\eta_1$; consequently, the initial optimality gap depends on $\eta_1^2\|\nabla f(x_0)+s_0\|^2$.

Furthermore, at iteration $k$, the algorithm makes $m_k+2n_k$ calls to the $\mathcal{SO}$. This is because stochastic AC-FGM uses three independent mini-batches, namely $\{\xi_{k,i}\}$, $\{\bar{\xi}_{k,i}\}$, and $\{\hat{\xi}_{k,i}\}$, to compute the current iterate update \eqref{eqn:gradient-step-stochastic} and the empirical local cocoercivity-based smoothness estimator $\bar{L}_k$ in \eqref{eqn:sample-wise-Lk-mini-batch}, which in turn determines the stepsize for the next iteration. Hence, by substituting the bound $\eta_k\leq \frac{k-1}{\bar{L}_{k-1}}$ into the mini-batch size choices \eqref{eqn:batch-size-cor1-const} and \eqref{eqn:n-k}, the total number of calls to the $\mathcal{SO}$ satisfies
\begin{align}\label{eqn:R_N}
\sum_{k=1}^{N+1}m_k+2\sum_{k=1}^{N}n_{k}
=\mathcal{O}\left(N+\frac{R_N^2}{  {\tilde{D}^2}}N^4\right),\quad\text{where}\quad R_N^2
\coloneqq \frac{1}{N}\sum_{k=1}^N\frac{v_{k-1}^{\max}  {\tilde{D}^2}+\delta_{k}^2+\sigma_{k-1}^2}{\bar{L}_{k-1}^2}.
\end{align}
The quantity $R_N$ characterizes the average variance-to-smoothness ratio over the iteration limit $N$ and depends only on the trajectory, that is, on the finitely many points visited by the algorithm.

In the stochastic case, since $v_{\max}\geq v_0>0$, we have $\max\left\{\frac{v_{\max}}{v_0}, 1\right\}=\frac{v_{\max}}{v_0}$, and it simplifies to
\begin{align}\label{eqn:final-sample complexity}
\mathcal{O}\left(
\sqrt{\frac{\mathcal{L}D_0^2}{\varepsilon}\cdot\frac{v_{\max}}{v_0}}
+\frac{R_N^2\mathcal{L}^{2}D_0^2}{\varepsilon^{2}}\cdot\frac{D_0^2}{  {\tilde{D}^2}}\cdot\frac{v^2_{\max}}{v_0^2}
\right)
\end{align}
calls to the $\mathcal{SO}$,
where $v_{\max}$ is a deterministic upper bound in \eqref{eqn:variance-bounds-local-smoothness} on the conditional variance of the local Lipschitz smoothness, $\mathcal{L}\geq \mathcal{L}(\bar{\xi}, \hat{\xi})$ is deterministic upper bound on the smoothness parameter, and $D_0$ is defined in \eqref{eqn:chocie D}.
 Recall the AC-SA sample complexity in the parameter known setting \citep[Corollary 5]{ghadimi2016accelerated}:
\begin{align}\label{sample-complexity-classical}
\mathcal{O}\left(\sqrt{\frac{L(\min\limits_{x^*\in X^*}\|x^*-x_{0}\|^2+\tilde{D}^2)}{\varepsilon}}+\frac{\sigma^2(\min\limits_{x^*\in X^*}\|x^*-x_{0}\|^2+\tilde{D}^2)}{\varepsilon^{2}}\cdot\frac{\min\limits_{x^*\in X^*}\|x^*-x_{0}\|^2+\tilde{D}^2}{  {\tilde{D}^2}}\right).
\end{align}
Therefore, compared with the parameter known case \eqref{sample-complexity-classical}, the sample complexity \eqref{eqn:final-sample complexity} in the parameter-free setting remains optimal in its dependence on $\varepsilon$. Observe that \eqref{eqn:final-sample complexity} also depends on the local quantity $v_{k-1}^{\max}$ because of the random stepsize, which controls the bias of the estimator $\bar{L}_{k-1}^{-1}$; see \autoref{lem:bias-local cocoercivity-based smoothnes}. Moreover, by the definition of $R_N^2$ in \eqref{eqn:R_N}, $R_N^2$ in \eqref{eqn:final-sample complexity} plays the role of ${\sigma^2}/{L^2}$ in the parameter known case \eqref{sample-complexity-classical}. In addition, \eqref{eqn:final-sample complexity} involves the global deterministic bounds $\mathcal{L}$ and $v_{\max}$. This dependence arises because the guarantees here hold in expectation, while both the stepsize $\eta_{N+1}$ and the conditional variance $v_{N+1}^{\max}$ are random quantities. In particular, in the convergence analysis, to lower bound
\begin{align*}
    \mathbb{E}\left[\frac{\eta_{N+1}\beta_{N+1}(\tau_{N}+1)[\Psi(x_{N})-\Psi(x^*)]}{v_{N+1}^{\max}}\right],
\end{align*}
the analysis must account for the worst-case dependence on $\mathcal{L}$ and $v_{\max}$.

If instead we consider the weaker guarantee of obtaining an $\varepsilon$-solution satisfying $\mathbb{E}^{2}\left[\sqrt{\Psi(x_N)-\Psi(x^*)}\right]\leq \varepsilon$, then the dependence on the global deterministic bounds $\mathcal{L}$ and $v_{\max}$ can be sharpened to expectations of local quantities, as shown below. In the deterministic case, this guarantee coincides exactly with that of AC-FGM.
\begin{corollary}\label{remark:weak-guarantee}
    Suppose the conditions in \autoref{cor:main-fixed-T-const} hold. Then
   \begin{equation*}
\begin{aligned}
\mathbb{E}^2\left[\sqrt{\Psi({x}_N)-\Psi(x^*)}\right]
&\le
\frac{32D_0^2}{\beta N^2} \max\left\{\frac{\mathbb{E}\left[\hat{L}_Nv_{N+1}^{\max}\right]}{v_0}, \mathbb{E}\left[\hat{L}_N\right]\right\},
\\[0.5ex]
\min_{x^*\in X^*}\mathbb{E}^2\bigl[\|y_{N+1}-x^*\|\bigr]
&\le
D_0^2
\max\left\{\frac{\mathbb{E}\left[v^{\max}_{N}\right]}{v_0},1\right\},
\\[0.5ex]
 \min\limits_{2\leq k\leq N}\mathbb{E}^2\bigl[\|\mathcal{G}(y_k,\nabla f(x_k),\eta_{k+1})\|\bigr]
&\le
\frac{4096D_0^2}{\beta^2N^3}
\max\left\{\frac{\mathbb{E}\left[\hat{L}_N^2v_{N+1}^{\max}\right]}{v_0}, \mathbb{E}\left[\hat{L}_N^2\right]\right\},
\end{aligned}
\end{equation*}
where $D_0$ is defined in \eqref{eqn:chocie D}, $v^{\max}_{N}\coloneqq\max_{0\le i\le N}\left\{v_i^2\right\}$, and $\hat{L}_{N}\coloneqq \max\left\{\frac{1}{32(1-\beta)\eta_1},\, \bar{L}_1,\, \bar{L}_2,\, \dots,\, \bar{L}_{N}\right\}$.
It is worth noting that in the deterministic setting, when $v_{\max}=0$, and hence $v^{\max}_{N}=0$, we have
$$
\max\left\{\frac{\mathbb{E}\left[v^{\max}_{N}\right]}{v_0},1\right\}=1,\quad
\max\left\{\frac{\mathbb{E}\left[\hat{L}_Nv_{N+1}^{\max}\right]}{v_0}, \mathbb{E}\left[\hat{L}_N\right]\right\}=\hat{L}_{N}, \,\, \max\left\{\frac{\mathbb{E}\left[\hat{L}_N^2v_{N+1}^{\max}\right]}{v_0}, \mathbb{E}\left[\hat{L}_N^2\right]\right\} =\hat{L}_N^2,
$$
and the resulting complexity bound recovers the deterministic result of AC-FGM \cite{li2025simple}, where $\hat{L}_N\leq L$ depends only on the finitely many points visited by the algorithm. Note also that AC-FGM naturally yields a reduced gradient-norm bound of $\mathcal{O}(L^2/N^3)$.
\end{corollary}

\autoref{remark:weak-guarantee} provides one way to remove the dependence on the global quantities $\mathcal{L}$ and $v_{\max}$. However, in general, if we want to guarantee $\mathbb{E}[\Psi({x}_N)-\Psi(x^*)]\leq \varepsilon$, it is unclear how to remove this global dependence due to the random stepsize. In the next section, we show that, under standard light-tail assumptions, the global dependencies on $\mathcal{L}$ and $v_{\max}$ in \eqref{eqn:final-sample complexity} disappear when obtaining a solution satisfying $\Psi({x}_N)-\Psi(x^*)\leq \varepsilon$ with high probability. More specifically, we sharpen them to the local quantities $L_N$ and $v_N^{\max}$, which depend only on the finitely many points visited by the algorithm.

Observe that the method still depends on the typically unknown initial optimality gap
$\min_{x^*\in X^*}\|x^*-x_{0}\|^2$.
If the user-chosen $\tilde D^2$ in the mini-batch sizes \eqref{eqn:batch-size-cor1-const} and \eqref{eqn:n-k} satisfies
$\tilde D^2 \leq \min_{x^*\in X^*}\|x^*-x_{0}\|^2$,
then we obtain the desirable dependence on the initial optimality gap, since in this case, the iteration complexity in \eqref{eqn:iteration complexity} simplifies to
$$
  \mathcal{O}\left( \frac{\mathcal{L}^{\frac{1}{2}}\left(\eta_1^2\left\| {\nabla f(x_0)+s_0}\right\|^2
+\min\limits_{x^*\in X^*}\|x^*-x_{0}\|^2\right)^{\frac{1}{2}}}{\varepsilon^{\frac{1}{2}}}\cdot\max\left\{\frac{v_{\max}^{\frac{1}{2}}}{v_0^{\frac{1}{2}}}, 1\right\}\right).
$$
In general, if $\min_{x^*\in X^*}\|x^*-x_{0}\|^2$ is unknown, we can only derive the iteration complexity in \eqref{eqn:iteration complexity} and the sample complexity in \eqref{eqn:final-sample complexity}. Such dependence on the initial optimality gap remains an interesting problem for stochastic parameter-free methods. For example, as shown in \cite{kreisler2024accelerated}, unlike stochastic AC-FGM, which can converge regardless of the choice of $\tilde{D}$, the method U-DOG \cite{kreisler2024accelerated} requires a lower bound on the initial optimality gap for the algorithm to run and converge. Other works impose even more stringent conditions, such as requiring the diameter of a bounded domain or upper bounds on the gradient norm over an unbounded domain, for the algorithm to converge.

To ultimately remove this dependence, one may leverage the idea of \emph{accumulative regularization} \cite{lan2026optimal,ji2025high}. Combined with a standard \emph{guess-and-check} procedure, this dependence on $\min_{x^*\in X^*}\|x^*-x_{0}\|^2$ can in principle be eliminated. We leave a complete development of this approach to future work.

Observe that the algorithm is now fully agnostic to the global smoothness parameter $L$. However, several limitations remain. At iteration $k$, \autoref{cor:main-fixed-T-const} still requires knowledge of the total iteration budget $N$, as well as the local variances $\sigma_{k-1}^2$ and $v_{k-1}$ for choosing $m_k$, and $\delta_k^2$ and $v_{k-1}$ for choosing $n_k$. Moreover, the current complexity bound depends on the conservative global quantities $v_{\max}/v_0$ and $\mathcal{L}$. In the sequel, we relax these requirements step by step: we remove the need to know the iteration limit $N$, the local variances $\sigma_{k-1}^2$, $v_{k-1}$, and $\delta_k^2$, and further improve the dependence on $v_{\max}/v_0$ and $\mathcal{L}$ in the high-probability convergence guarantee.

\subsubsection{Adaptivity to the iteration limit}\label{sec:Adaptivity to the iteration limit}
In this subsection, we remove the dependence on the iteration limit $N$ by introducing the nontrivial anchored regularizer $\frac{\gamma_k}{2\eta_k}\|z-y_0\|^2$ in \autoref{alg:main-single-objective-stochastic} with $\gamma_k\neq 0$, which induces curvature around the fixed reference point $y_0$. By choosing the regularization parameter $\gamma_k$ appropriately, we obtain the same order of convergence and sample complexity as in the setting where one assumes $N$ is known in advance.

We adopt the notation $v^{\max}_{k-1}$ in \eqref{eqn:largest-var-lip} for the largest local conditional variance of the sample Lipschitz smoothness along the trajectory up to iteration $k-1$, and define the universal constants
\begin{align}\label{eqn:constant-exp-p}
c \coloneqq 8,\quad
\tilde{c}\coloneqq 745.
\end{align}
We also define the initial optimality gap measure $D_0$ by
\begin{align}\label{eqn:chocie D-case2}
D_0^2
\coloneqq \frac{9\eta_1^2\| {\nabla f(x_0)+s_0}\|^2}{2}
+30\left(\min\limits_{x^*\in X^*}\|x^*-x_{0}\|^2+\tilde{D}^2\right),
\end{align}
where $s_0\in\partial h(x_0),$ and $\tilde{D}>0$ is arbitrary. We are now ready to state the corresponding convergence guarantee.
\begin{theorem}\label{cor:main}
Suppose that \autoref{a:unbiasness}, \autoref{ass:finite sample convexity}, and \autoref{assump:Bounded local variance} hold.
Assume further that $\gamma_k=\frac{1}{k}$ and $\tau_k=\frac{k+2-\beta}{2}$ for all $k\geq 1$. Let $\beta_1=0$ and $\beta_k\equiv\beta\in\left(0,\frac{1}{8}\right)$ for all $k\geq 2$. In addition, let $\eta_1>0$ and define
\begin{align}\label{eqn:stepsize-5}
    \eta_2= \min\left\{\frac{1}{16\bar{L}_{1}},\frac{2(1-\beta)}{{3-\beta}}\eta_{1}\right\},\quad  \eta_k&=\min\left\{\frac{k-1}{16\bar{L}_{k-1}},\frac{(k-1)(k+2-\beta)}{k^2}\eta_{k-1}\right\},
 \quad \forall\, k\ge 3.
\end{align}
Furthermore, suppose that the mini-batch sizes satisfy
\begin{align}
    &m_k= \max\left\{1, \,\frac{(k+2)\eta_{k}^2}{\beta^2}\cdot\frac{c {\sigma}_{k-1}^2}{\tilde{D}^2}\right\},\label{eqn:mini-batch-N-free}\\
        &n_k= \max\left\{1,\,\frac{\tilde{c}(k+2)\eta_{k}^2v^{\max}_{k-1}}{\beta^4}, \,\frac{(k+2)\eta_{k}^2}{\beta^2}\cdot\frac{c (\sigma_{k-1}^2+\delta_k^2)}{\tilde{D}^2}
         \right\},
         \label{eqn:mini-batch-n-N-free}
\end{align}
for any $\tilde{D}>0$. Then, for all $N\geq 1,$ it holds that
\begin{align*}
&\mathbb{E}[\Psi(x_N)-\Psi(x^*)]\leq
\frac{20\mathcal{L}D_0^2}{\beta N^2}\cdot\max\left\{\frac{v_{\max}}{v_0}, 1\right\},\quad
\min_{x^*\in X^*}\mathbb{E}[\|y_{N+1}-x^*\|^2]\leq D_0^2\cdot\max\left\{\frac{v_{\max}}{v_0}, 1\right\},
\end{align*}
where $D_0$ is defined in \eqref{eqn:chocie D-case2}.
\end{theorem}
In this case, to obtain an $\varepsilon$-solution satisfying $\mathbb{E}[\Psi(\bar{x}_N)-\Psi(x^*)]\leq \varepsilon$, the stochastic mini-batch AC-FGM \autoref{alg:main-single-objective-stochastic} requires the same order of iterations as in the case where $N$ is known a priori, namely,
\begin{equation*}
   \mathcal{O}\left( \sqrt{\frac{\mathcal{L}D_0^2}{\varepsilon}\cdot\max\left\{\frac{v_{\max}}{v_0}, 1\right\}}\right)
\end{equation*}
iterations, where $D_0$ is defined in \eqref{eqn:chocie D-case2}. The total number of calls to the $\mathcal{SO}$ is
\begin{align*}
&\sum_{k=1}^{N+1}m_k+2\sum_{k=1}^{N}n_{k}=\mathcal{O}\left(N+\frac{R_N^2}{  {\tilde{D}^2}}N^4\right),
\end{align*}
where $R_N$ is defined in \eqref{eqn:R_N}. Notice that the resulting sample complexity matches \eqref{eqn:final-sample complexity} in \autoref{cor:main-fixed-T-const}, even though we no longer assume that the total number of iterations is known in advance.

 \subsubsection{Adaptivity to local variances}\label{sec:Adaptivity to the local variances}

Up to this point, \autoref{cor:main-fixed-T-const} and \autoref{cor:main} assume that, at each iteration $k$, the local conditional variances of the stochastic gradient, $\sigma_{k-1}^2$ and $\delta_k^2$, as well as the local conditional variance associated with Lipschitz smoothness, $v_{k-1}$, are available for computing the mini-batch sizes. In fact, exact knowledge of these quantities is not necessary: it suffices to use suitable variance proxies that overestimate them in order to ensure convergence.

In this subsection, we show that stochastic AC-FGM allows for variance estimation by replacing $(\sigma_{k-1}^2, v_{k-1}, \delta_k^2)$ with their estimators $(\hat{\sigma}_{k-1}^2, \hat{v}_{k-1}^2, \hat{\delta}_k^2)$ and enlarging the underlying filtration through a third type of batches $(\xi_k^{\,b}, \bar{\xi}_k^{\,b}, \hat{\xi}_k^{\,b})$, in addition to $\xi_k$ for the main update and $\bar{\xi}_k, \hat{\xi}_k$ for the stepsize update. This framework accommodates different variance estimators constructed from the third type of batch, for example, pairwise sample variance estimators. The algorithm remains adaptive to the Lipschitz constant, the total number of iterations $N$, and, through the third type of batch, the local variances. Moreover, it still achieves the optimal convergence rate with the same iteration and sample complexity guarantees as before, except that the convergence guarantee now holds on a high-probability event determined by the quality of the input variance estimators.

Specifically, when the conditional variance quantities $\sigma_x^2$, $\delta_x^2$, and $v_x$ at a given point $x$ are unknown, we introduce a third type of fresh batch to determine the mini-batch sizes $m_{k+1}$ for the main update and $n_{k+1}$ for the next stepsize update. In particular, instead of \eqref{eqn:mini-batch-N-free} and \eqref{eqn:mini-batch-n-N-free}, we consider
\begin{align}\label{eqn:m1'}
  m_{k+1}= \max\left\{1, \,\frac{(k+3)\eta_{k+1}^2}{\beta^2}\cdot\frac{c \hat{\sigma}_{k}^2}{\tilde{D}^2}\right\},
\end{align}
where $\hat{\sigma}_{k}^2$ is constructed using the fresh batch $\{\hat{\xi}_{k,i}^{\,b}\}_{i=1}^{r_k}$. Furthermore,
\begin{align}\label{eqn:m2'}
n_{k+1}= \max\left\{1,\,\frac{\tilde{c}(k+3)\eta_{k+1}^2\hat{v}^{\max}_{k}}{\beta^4}, \,\frac{(k+3)\eta_{k+1}^2}{\beta^2}\cdot\frac{c (\sigma_{k}^2+\hat{\delta}_{k+1}^2)}{\tilde{D}^2}
         \right\},
\end{align}
where $\hat{v}^{\max}_{k}\coloneqq \max_{0\le j\le k} \hat{v}_j$, $\hat{v}_j$ is constructed using the fresh batch $\{\bar{\xi}_{j,i}^{\,b}\}_{i=1}^{r_j}$, and $\hat{\delta}_{k+1}^2$ is constructed using the fresh batch $\{\xi_{k+1,i}^{\,b}\}_{i=1}^{r_{k+1}}$. This third type of batches determines the data-dependent batch sizes $m_{k+1}$ and $n_{k+1}$, thereby making the method fully adaptive.
The choice of the third type of auxiliary batch size $r_k$ depends on the specific application. In general, $r_k$ is chosen to guarantee a reliable upper bound on the local variance quantities $\sigma_k^2$ and $\delta_k^2$, typically with high probability.

To incorporate the convergence analysis from the previous sections, we enlarge the filtration as follows. This enlarged filtration preserves the properties of the original filtration \eqref{def:filtration} while incorporating the variance-estimation batches. Specifically, we define the natural filtration $\{\mathcal{F}_k\}_{k\ge 0}$ recursively according to the order in which the batches are revealed:
\begin{equation}\label{def:filtration-new}
\begin{aligned}
\mathcal{F}_{0} &\coloneqq \sigma(\emptyset),\\
\mathcal{F}_{k-\frac{2}{3}} &\coloneqq \sigma\!\left(\mathcal{F}_{k-1}, \{\xi_{k,i}\}_{i=1}^{m_k}, \{\xi_{k,i}^{\,b}\}_{i=1}^{r_k}\right),\\
\mathcal{F}_{k-\frac{1}{3}} &\coloneqq \sigma\!\left(\mathcal{F}_{k-\frac{2}{3}}, \{\bar{\xi}_{k,i}\}_{i=1}^{n_k},\{\bar{\xi}_{k,i}^{\,b}\}_{i=1}^{r_k}\right),\\
\mathcal{F}_{k} &\coloneqq \sigma\!\left(\mathcal{F}_{k-\frac{1}{3}}, \{\hat{\xi}_{k,i}\}_{i=1}^{n_k},\{\hat{\xi}_{k,i}^{\,b}\}_{i=1}^{r_k}\right),
\end{aligned}
\qquad k\ge 1.
\end{equation}
By the constructions in \eqref{eqn:m1'} and \eqref{eqn:m2'}, it is immediate to see that $m_k\in\mathcal{F}_{k-1}$ and $n_k\in\mathcal{F}_{k-\frac{2}{3}}$.

It is natural to assume that under the enlarged filtration \eqref{def:filtration-new}, the conditional unbiased estimator property in \autoref{a:unbiasness} and the conditional bounded variance property in \autoref{ass:finite sample convexity} still hold, since the filtration is only slightly enlarged. We continue to denote by $\mathcal{G}_k$ the filtration generated by the iterates, as defined in \eqref{eqn:G-k}. The key properties of the original filtration $\mathcal{F}_k$ in \eqref{def:filtration} needed for the analysis are the following: for all $k \ge 1$, $\mathcal{G}_{k} \subseteq \mathcal{F}_{k-\frac{2}{3}}$, $x_k, z_k \in \mathcal{F}_{k-\frac{2}{3}}$, $m_k \in \mathcal{F}_{k-1}$, $n_k \in \mathcal{F}_{k-\frac{2}{3}}$, $\|\Delta G(x_k,\bar{\xi}_k)\|^2 \in \mathcal{F}_{k-\frac{1}{3}}$, and $T(x_{k}, \hat{\xi}_{k})\in \mathcal{F}_k$, and hence $\bar{L}_k, \eta_{k+1} \in \mathcal{F}_k$. Therefore, both the stepsize and the mini-batch sizes are random while remaining fully adaptive. All these properties continue to hold under the enlarged filtration \eqref{def:filtration-new}; with a slight abuse of notation, we still denote it by $\mathcal{F}_k$. In fact, all the analysis from the previous section is carried out under this enlarged filtration. When the variance is known, we may simply regard the two filtrations \eqref{def:filtration} and \eqref{def:filtration-new} as coinciding. See \autoref{fig:filtration}  and \autoref{fig:2}  for comparison.
  \begin{figure}
      \centering
      \includegraphics[width=0.99\linewidth]{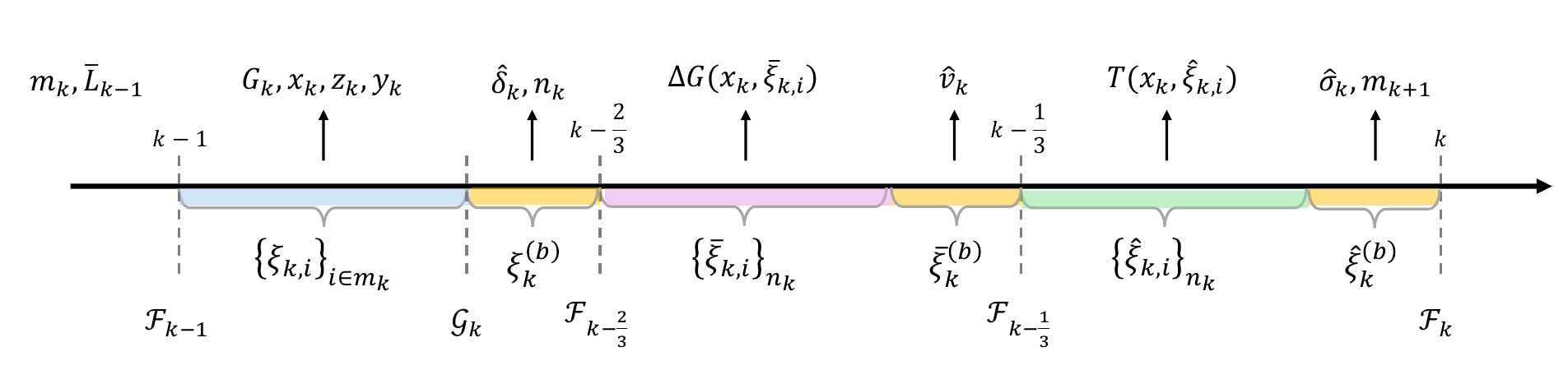}
      \caption{Illustration of the enlarged filtration $\{\mathcal{F}_{k}\}_{k\in \mathbb{N}{+}}$ and the intermediate $\sigma$-algebra $\mathcal{F}_{k-\frac{1}{3}}$ and $\mathcal{F}_{k-\frac{2}{3}}$.}
      \label{fig:2}
  \end{figure}

We now state the corresponding convergence guarantee.
\begin{theorem}\label{cor:remove variance}
Suppose the same conditions as in \autoref{cor:main} hold, with the modifications that $m_k$ and $n_k$ satisfy \eqref{eqn:m1'} and \eqref{eqn:m2'}, respectively. Suppose
\begin{align}\label{eqn:high-p-event}
   A_N\coloneqq \{\forall\, k\in[N], \hat{v}_{k-1}\geq {v}_{k-1}, \hat{\sigma}_{k-1}^2\geq {\sigma}_{k-1}^2, \hat{\delta}_{k}^2\geq {\delta}_k^2\},\quad  \mathbb{P}(A_N)\geq 1-p_N.
\end{align}
Then, conditional on the event $A_N$, the conclusions of \autoref{cor:main} hold. In particular,
\begin{align*}
\mathbb{E}\!\left[\Psi(x_N)-\Psi(x^*)\,\middle|\,A_N\right]
&\le
\frac{20\mathcal{L}D_0^2}{\beta N^2}\cdot\max\left\{\frac{v_{\max}}{v_0}, 1\right\},\\
\min_{x^*\in X^*}\mathbb{E}\!\left[\|y_{N+1}-x^*\|^2\,\middle|\,A_N\right]
&\le
D_0^2\cdot\max\left\{\frac{v_{\max}}{v_0}, 1\right\}.
\end{align*}
where $D_0$ is defined in \eqref{eqn:chocie D-case2}.
\end{theorem}

In this case, to obtain an $\varepsilon$-solution satisfying $\mathbb{E}\!\left[\Psi(x_N)-\Psi(x^*)\,\middle|\,A_N\right]\leq \varepsilon$, the stochastic mini-batch AC-FGM \autoref{alg:main-single-objective-stochastic} requires the same order of iterations as in the setting where the iteration limit $N$ and the previous conditional variances are known a priori, namely,
\begin{equation*}
   \mathcal{O}\left( \sqrt{\frac{\mathcal{L}D_0^2}{\varepsilon}\cdot\max\left\{\frac{v_{\max}}{v_0}, 1\right\}}\right)
\end{equation*}
iterations. The total number of calls to the $\mathcal{SO}$ reads
\begin{align*}
\sum_{k=1}^{N+1}(m_k+2n_{k}+6 r_k)
=\mathcal{O}\left(N+\frac{\hat{R}_N^2}{  {\tilde{D}^2}}N^4+\sum_{k=1}^{N+1} r_k\right),\quad\text{where}\quad \hat{R}_N^2
\coloneqq \frac{1}{N}\sum_{k=1}^N\frac{\hat{v}_{k-1}^{\max}  {\tilde{D}^2}+\hat{\delta}_{k}^2+\hat{\sigma}_{k-1}^2}{\bar{L}_{k-1}^2}.
\end{align*}
Compared with ${R}_N$ in \eqref{eqn:R_N}, the ratio $\hat{R}_N$ characterizes the average ratio of the {\it sample} variance to the smoothness estimator $\bar{L}_k$ over the horizon $N$ and depends only on the trajectory, that is, on the finitely many points visited by the algorithm. Moreover, $6\sum_{k=1}^{N+1} r_k$ quantifies the number of observations required by the input variance estimator to ensure $\mathbb{P}(A_N)\geq 1-p_N$.

In the stochastic case, since $v_{\max}\geq v_0>0$, the total number of calls to the $\mathcal{SO}$ reads
\begin{align}\label{eqn:final-sample complexity-var}
\mathcal{O}\left(
\sqrt{\frac{\mathcal{L}D_0^2}{\varepsilon}\cdot\frac{v_{\max}}{v_0}}
+\frac{\hat{R}_N^2\mathcal{L}^{2}D_0^2}{\varepsilon^{2}}\cdot\frac{D_0^2}{  {\tilde{D}^2}}\cdot\frac{v^2_{\max}}{v_0^2}+\sum_{k=1}^{N+1} r_k
\right).
\end{align}
Notice that this sample complexity matches \eqref{eqn:final-sample complexity} in \autoref{cor:main-fixed-T-const}, with the dependence on the true variances ${\sigma}^2_{k-1}$, $v_{k-1}$, and $\delta_k^2$ replaced by the dependence on the empirical variances $\hat{\sigma}^2_{k-1}$, $\hat{v}_{k-1}^2$, and $\hat{\delta}_k^2$. Furthermore, $r_k$ depends only on the confidence level $p_N$ and can be very small in many cases. Thus, it does not affect the overall iteration complexity or sample complexity. We next present one example in which overestimates of the variances can be derived with probability at least $1-p_N$, where $r_k$ can be chosen on the order of $\log(N/p_N)$.

Consider the pairwise estimators defined as follows. For each $k\in[N]$, let
\begin{equation}\label{eqn:pair-wise-sample}
    \begin{aligned}
      \hat{\sigma}_{k}^2
&\coloneqq
\frac{1}{2r_k}\sum_{i=1}^{r_k}
\left\|
G(x_k,\hat{\xi}^{\,b}_{k,2i-1})
-
G(x_k,\hat{\xi}^{\,b}_{k,2i})
\right\|^2, \\
\hat{v}_{k}^{2}
&\coloneqq
\frac{1}{2r_k}\sum_{i=1}^{r_k}
\left[
\tilde{L}_{k}(\bar{\xi}^{\,b}_{k,2i-1})
-
\tilde{L}_{k}(\bar{\xi}^{\,b}_{k,2i})
\right]^2,\\
       \hat{\delta}_{k+1}^2
&\coloneqq
\frac{1}{2r_k}\sum_{i=1}^{r_{k+1}}
\left\|
G(x_{k+1},{\xi}^{\,b}_{k+1,2i-1})
-
G(x_{k+1},{\xi}^{\,b}_{k+1,2i})
\right\|^2,
    \end{aligned}
\end{equation}
where $r_{k+1}$ and $r_{k}$ denote the numbers of pairs used to estimate the variances, and ${\xi}^{\,b}_{k+1,i}$, $\bar{\xi}^{\,b}_{k,i}$, and $\hat{\xi}^{\,b}_{k,i}$ are fresh observations used for variance estimation. In particular, observe that $m_{k+1}\in\mathcal{F}_k$ and $n_{k+1}\in\mathcal{F}_{k+\frac{1}{3}}$; thus, the batch sizes are adaptive.

To obtain the uniform high-probability event
\[
A_N\coloneqq \{\forall\, k\in[N],\ \hat{v}_{k-1}\ge v_{k-1},\ \hat{\sigma}_{k-1}^2\ge \sigma_{k-1}^2,\ \hat{\delta}_{k}^2\ge \delta_k^2\},
\]
one may replace the raw pairwise variance averages in \eqref{eqn:pair-wise-sample} with inflated robust mean estimators applied to the corresponding nonnegative pairwise observations. Standard choices include the median-of-means estimator, Catoni's estimator, and the geometric median-of-means estimator. These estimators admit high-probability deviation guarantees under weak moment assumptions and are therefore suitable for constructing variance overestimates with high probability; see, for example, \cite{lugosi2019mean,catoni2012challenging,minsker2013geometric}. In particular, under a bounded fourth-moment assumption, for all $k\in[N]$
it suffices to take
\[
r_k=\mathcal{O}\!\left(\log\frac{N}{p_N}\right)
\]
auxiliary pairs per iteration to guarantee $\mathbb{P}(A_N)\ge 1-p_N$. Therefore, the overall sample complexity remains of the same order as in the variance-known case; however, the guarantee is now conditional on the high-probability event $A_N$.
\subsection{High probability guarantees with sharper rates}

The in-expectation complexity bounds from the previous subsections depend on the conservative upper bounds $v_{\max}/v_0$ and $\mathcal{L}$. In this subsection, we show that in the high-probability analysis, these quantities can be replaced by local ones, leading to sharper convergence guarantees and improved sample complexity bounds. This yields the optimal convergence rate and sample complexity and, to the best of our knowledge, also achieves the tightest currently known dependence on the Lipschitz smoothness constant and the noise level in the stochastic parameter-free optimization literature. Following the standard convention in the literature on sub-Gaussian noise assumptions, we treat the relevant sub-Gaussian parameters as known at the current iterate. Whenever they can be estimated, our filtration design and the corresponding theory still preserve full adaptivity to the Lipschitz smoothness, the iteration limit, and the mini-batch sizes.

Specifically, if \autoref{assump:Bounded local variance} is replaced by the following \emph{light-tail} assumption, then the convergence guarantees can be strengthened from in-expectation bounds to high-probability bounds.

\begin{assumption}[Sub-Gaussian noise]\label{ass:subgaussian}
Given the current iterate $x_{k}$, there exists $\sigma_{k}\geq 0$ such that for a fresh main update batch $\{\xi_{k+1,i}\}_{i=1}^{m_{k+1}}$,
\begin{align}
\mathbb{E}_{\xi_{k+1}}\!\left[\exp\left\{\frac{\|\sum_{i=1}^{m_{k+1}}[G(x_{k},\xi_{k+1, i})-\nabla f(x_{k})]\|^{2}}{m_{k+1}\sigma_{k}^{2}}\right\}\,\bigg|\,\mathcal{F}_{k}\right]\leq \exp\{1\}.\label{eqn:local-var-sub}
\end{align}
There exists $\delta_k\geq 0$ such that for a fresh stepsize selection batch $\{\bar{\xi}_{k,i}\}_{i=1}^{n_k}$,
\begin{align}
\mathbb{E}_{\bar{\xi}_k}\!\left[\exp\left\{\frac{\|\sum_{i=1}^{n_k}[G(x_{k},\bar{\xi}_{k,i})-\nabla f(x_{k})]\|^{2}}{n_k\delta_{k}^{2}}\right\}\,\bigg|\,\mathcal{F}_{k-\frac{2}{3}}\right]\leq \exp\{1\}.\label{eqn:local-var-sub-2}
\end{align}
Furthermore, there exists $v_{k}>0$ such that for a fresh stepsize selection batch $\{\hat{\xi}_{k,i}\}_{i=1}^{n_k}$,
\begin{align}\label{eqn:local-var-sub2}
\mathbb{E}_{\hat{\xi}_k}\!\left[\exp\left\{\frac{|\sum_{i=1}^{n_{k}}[\ell_k(\hat{\xi}_{k,i})-L_{k}]|^{2}}{n_{k}v_{k}}\right\}\,\bigg|\,\mathcal{F}_{k-\frac{2}{3}}\right]\leq \exp\{1\},
\end{align}
where $\ell_k(\cdot)$ is defined in \eqref{eqn:L-k}.
\end{assumption}

We next introduce several quantities in the convergence rate. We choose an arbitrary nonnegative number $v_0>0$ and define the largest local sub-Gaussian parameter along the trajectory up to iteration $k-1$ by
\begin{align}\label{eqn:a_k-high-p}
   v^{\max}_{k-1}\coloneqq\max_{0\le i\le k-1}\left\{v_i^2\right\},
\end{align}
and define the universal constants $c_{\Lambda}$ and $\tilde{c}_{\Lambda}$, which depend on the confidence level $\Lambda$, as follows:
\begin{align}\label{eqn:constant-high-p}
c_{\Lambda}\coloneqq 9(1+\Lambda) +729\Lambda^2,\quad \tilde{c}_{\Lambda}\coloneqq 988(1+\Lambda).
\end{align}
Moreover, we define the largest Lipschitz smoothness parameter along the trajectory by
\begin{align}\label{eqn:conclusion-stepsize-N-unknown}
    \hat{L}_{N}\coloneqq \max\left\{\frac{1}{64(1-\beta)\eta_1},\, \bar{L}_1,\, \bar{L}_2,\, \dots,\, \bar{L}_{N}\right\}.
\end{align}
We now state the corresponding convergence guarantee.
\begin{theorem}\label{eqn:high-main-2}
Suppose \autoref{a:unbiasness}, \autoref{ass:finite sample convexity}, and \autoref{ass:subgaussian} hold. Suppose $\gamma_k, \tau_k, \beta_k, \eta_k$ are chosen as in \autoref{cor:main}. Furthermore, for all $k\geq 1$, suppose that the mini-batch sizes satisfy
\begin{equation}\label{eqn:batch-high-gamma-neq-0'}
    \begin{aligned}
         &m_k= \max\left\{1,\,\frac{(k+2)\eta_{k}^2}{\beta^2}\cdot\frac{c_{\Lambda}{\sigma}_{k-1}^2}{\tilde{D}^2}\right\},\\
          &n_k= \max\left\{1,\,\frac{\tilde{c}_{\Lambda}(k+2)\eta_{k}^2 v_{k-1}^{\max}}{  {\beta^4}},\, \frac{(k+2)\eta_{k}^2}{\beta^2}\cdot\frac{c_{\Lambda}(\sigma_{k-1}^2+\delta_k^2)}{\tilde{D}^2}
         \right\},
 \quad \forall\, k\geq 1.
    \end{aligned}
\end{equation}
Then, with probability at least $1-(N+1)\exp\left\{-\frac{\Lambda^2}{3}\right\}-4(N+1)\exp\{-\Lambda\}$, it holds that
\begin{align*}
\Psi(x_{N})-\Psi(x^*)\leq \frac{20\hat{L}_{N}D_0^2}{\beta N^2}\cdot\max\left\{\frac{v^{\max}_{N+1}}{v_0}, 1\right\}, \quad \|y_{N+1}-x^*\|^2\leq D_0^2\cdot\max\left\{\frac{v^{\max}_{N}}{v_0}, 1\right\},
\end{align*}
where $\hat{L}_N$ is defined in \eqref{eqn:conclusion-stepsize-N-unknown}, $v^{\max}_{N+1}$ is defined in \eqref{eqn:a_k-high-p}, and $D_0$ is defined in \eqref{eqn:chocie D-case2}.
\end{theorem}
In this case, to reach an $\varepsilon$-solution such that $\Psi(x_N)-\Psi(x^*)\leq \varepsilon,$ with probability at least $1- (N+1)\exp\left\{-\frac{\Lambda^2}{3}\right\}-4(N+1)\exp\{-\Lambda\},$
the stochastic mini-batch AC-FGM \autoref{alg:main-single-objective-stochastic} requires
\begin{equation}\label{eqn:iteration-high-p}
   \mathcal{O}\left( \sqrt{\frac{\hat{L}_{N}D_0^2}{\varepsilon}\cdot\max\left\{\frac{v^{\max}_{N+1}}{a_0}, 1\right\}}\right)
\end{equation}
iterations. The total number of calls to $\mathcal{SO}$ is bounded by
\begin{align}\label{eqn:sample-high-p}
\sum_{k=1}^{N+1}m_k+2\sum_{k=2}^{N+1}n_{k-1}
=\mathcal{O}\left(N+\frac{R_N^2}{  {\tilde{D}^2}}N^4\right),\quad\text{where}\quad R_N^2
\coloneqq \frac{1}{N}\sum_{k=1}^N\frac{v_{k-1}^{\max}  {\tilde{D}^2}+\delta_{k}^2+\sigma_{k-1}^2}{\bar{L}_{k-1}^2}.
\end{align}
The ratio $R_N$ characterizes the average ratio of the sub-Gaussian parameter to the smoothness estimator $\bar{L}_k$ over the horizon $N$. In the stochastic case, it holds that $v^{\max}_{N+1}\geq v_0>0,$ and the total number of calls to $\mathcal{SO}$ is
\begin{align*}
\mathcal{O}\left(
\sqrt{\frac{\hat{L}_ND_0^2}{\varepsilon}\cdot\frac{v_{N+1}^{\max}}{v_0}}
+\frac{R_N^2\hat{L}_N^{2}D_0^2}{\varepsilon^{2}}\cdot\frac{D_0^2}{  {\tilde{D}^2}}\cdot\left(\frac{v_{N+1}^{\max}}{v_0}\right)^2
\right).
\end{align*}
Notice that both the iteration complexity and the sample complexity match the in-expectation result in their dependence on $\varepsilon$ (cf. \eqref{eqn:final-sample complexity}) and are therefore optimal. Moreover, they depend only on the trajectory-dependent quantities $\hat{L}_{N}$ and $v^{\max}_{N+1}$, rather than on global bounds, since these are determined solely by the iterates actually visited by the algorithm. Finally, $\Lambda$ governs the confidence level: a larger $\Lambda$ yields larger constants $c_{\beta,\Lambda}$ and $\tilde{c}_{\beta,\Lambda}$, and hence requires more observations, as expected.

Observe that both the iteration complexity and the sample complexity of stochastic AC-FGM are much smaller than those of U-DOG in \cite{kreisler2024accelerated}, whose iteration and sample complexities are
\[
\mathcal{\tilde{O}}\left(\sqrt{\frac{L d_0^{\,2}}{\varepsilon}}+\frac{d_0^{\,2} V_N}{\varepsilon^2}+\frac{d_0
}{\varepsilon}\cdot\max\limits_{x\in \mathbb{B}(x^*,2d_0)}\sigma_x \right),
\]
where $d_0$ is the initial optimality gap $\|x_0-x^*\|$, $V_N$ is the average variance along the trajectory over the iteration limit $N$, and $\sigma_x$ denotes the sub-Gaussian parameter at point $x$.
$\mathcal{\tilde{O}}$ contains polylogarithmic dependence on $\varepsilon.$
Notice that the iteration complexity is not optimal as a function of $\varepsilon$, since it is not of order $\mathcal{O}(1/\sqrt{\varepsilon})$. Furthermore, for the sample complexity, since the third term takes a supremum over the entire ball rather than over the finitely many iterates actually visited by the algorithm, it can be much larger and dominate the final error. By contrast, for stochastic AC-FGM, $\hat{L}_N$ and $v^{\max}_{N+1}$ in the iteration complexity \eqref{eqn:iteration-high-p} and sample complexity \eqref{eqn:sample-high-p} depend only on the algorithm trajectory. However, we emphasize that they do not require the finite-sample cocoercivity--smoothness condition in \autoref{ass:finite sample convexity}. It would be interesting to generalize stochastic AC-FGM beyond \autoref{ass:finite sample convexity}.

Finally, although the literature typically assumes known sub-Gaussian parameters, these parameters are notoriously difficult to estimate, much more so than a variance proxy, since they depend on the global tail behavior of the noise rather than only on its second moment. While variance-type quantities can often be estimated directly from auxiliary observations, reliable estimation of a sub-Gaussian parameter typically requires additional structural assumptions on the underlying distribution, which may be infeasible in practice.

Therefore, an alternative way to derive high-probability bounds for stochastic AC-FGM is through a median-of-means (MOM) type analysis, where one constructs estimators for the stochastic error terms and derives high-probability bounds under only a fourth-moment assumption. One caveat is that such an approach can boost the confidence level but not the convergence rate. Thus, the final bound still depends on the quantities appearing in the in-expectation bound, namely $\mathcal{L}$ and $v^{\max}$.

Unlike MOM-type arguments, under sub-Gaussian assumptions one can derive sharp dependence on the smoothness parameter and the variance. In some limited cases, such as bounded noise, these sub-Gaussian parameters can be estimated from the auxiliary sampling streams in \eqref{def:filtration-new}. Specifically, $\delta_k$ can be estimated from $\{\xi^{\,b}_{k,i}\}_{i=1}^{r_k}$, $\sigma_k$ from $\{\hat{\xi}^{\,b}_{k,i}\}_{i=1}^{r_k}$, and $v_k$ from $\{\bar{\xi}^{\,b}_{k,i}\}_{i=1}^{r_k}$.
\section{Convergence Analysis}\label{proof-main}

The goal of this section is to establish our main results. Specifically, Theorems~\ref{cor:main-fixed-T-const}--\ref{eqn:high-main-2} are derived from Proposition~\ref{thm:main-expectation}, which provides a trajectory-wise convergence guarantee for stochastic AC-FGM in \autoref{alg:main-single-objective-stochastic}. We begin by proving several technical lemmas needed for the proof of this proposition.

\begin{lemma}\label{lem:descent-one-step-equality}
Suppose \autoref{a:unbiasness} and \autoref{ass:finite sample convexity},
and let $m_k\in\mathcal{F}_{k-1}$ and $n_{k-1}\in \mathcal{F}_{k-\frac{5}{3}}.$
Furthermore, suppose $\tau_k \ge 0$ for all $k \ge 1$. Then, for
\autoref{alg:main-single-objective-stochastic}, for all $k \ge 2$, the following
holds almost surely:
\begin{align*}
& \langle \nabla f(x_{k-1}), z_{k-1}-z\rangle-\tau_{k-1}[f(x_{k-1})-f(x_{k-2})]-
\langle\nabla f(x_{k-1}), x_{k-1}-z\rangle\\
&=\tfrac{\tau_{k-1} }{2n_{k-1}^2}\left\|\textstyle\sum_{i=1}^{n_{k-1}}[G(x_{k-1},\bar{\xi}_{k-1,i})-G(x_{k-2},\bar{\xi}_{k-1,i})]\right\|^2
\mathbb{E}_{\hat{\xi}_{k-1}}[\bar{L}_{k-1}^{-1}\,|\,\mathcal{F}_{k-\frac{4}{3}}].
\end{align*}
\end{lemma}
\begin{proof}

{\noindent\bf i)}
Suppose that for all $\omega\in \Omega^k/ \mathcal{N}^k,$ there holds
        \begin{align*}
            {\tfrac{1}{n_{k-1}}\textstyle\sum_{i=1}^{n_{k-1}}[F(x_{k-2}, \hat{\xi}_{k-1,i})-F(x_{k-1}, \hat{\xi}_{k-1,i})-\langle G(x_{k-1}, \hat{\xi}_{k-1, i}),   x_{k-2}-x_{k-1}\rangle]}> 0,
        \end{align*}
        where we recall from \autoref{ass:finite sample convexity},  $\mathcal{N}^k$
 denotes the null set on which \eqref{eqn:sample-convexity-smoothness}  fail to hold.
     Then, by the definition of $\bar{L}_{k-1}$ in \eqref{eqn:sample-wise-Lk-mini-batch}, for all $k\geq 2,$
      there holds
        \begin{align}\label{eqn:sample-L-transform}
          &{\tfrac{1}{n_{k-1}}\textstyle\sum_{i=1}^{n_{k-1}}[F(x_{k-2}, \hat{\xi}_{k-1,i})-F(x_{k-1}, \hat{\xi}_{k-1,i})-\langle G(x_{k-1}, \hat{\xi}_{k-1, i}),   x_{k-2}-x_{k-1}\rangle]}\notag\\
          &=\tfrac{1}{2\bar{L}_{k-1}}\left\|\tfrac{1}{n_{k-1}}\textstyle\sum_{i=1}^{n_{k-1}}\left[G(x_{k-1},\bar{\xi}_{k-1, i})-G(x_{k-2},\bar{\xi}_{k-1, i})\right]\right\|^2,\quad\text{a.s.}
        \end{align}
       Moreover, notice that $x_{k-1}, x_{k-2}, n_{k-1}\in\mathcal{F}_{k-\frac{5}{3}},$ and
       $\bar{L}_{k-1},\hat{\xi}_{k-1, i}\in\mathcal{F}_{k-1},$ for all $i\in[n_{k-1}].$
       Therefore, we have
\begin{equation}\label{eqn:step3}
            \begin{aligned}
                 &f(x_{k-2})-f(x_{k-1})-\langle \nabla f(x_{k-1}),   x_{k-2}-x_{k-1}\rangle\\
             &\overset{\text{(i)} }{=}\mathbb{E}_{\hat{\xi}_{k-1}}\left[\tfrac{1}{n_{k-1}}\textstyle\sum_{i=1}^{n_{k-1}}F(x_{k-2}, \hat{\xi}_{k-1, i})-\tfrac{1}{n_{k-1}}\textstyle\sum_{i=1}^{n_{k-1}}F(x_{k-1}, \hat{\xi}_{k-1, i})\,\bigg|\,\mathcal{F}_{k-\frac{4}{3}}\right]\\
             &\quad -\tfrac{1}{n_{k-1}}\textstyle\sum_{i=1}^{n_{k-1}}\left\langle \mathbb{E}_{\hat{\xi}_{k-1}}\left[G(x_{k-1}, \hat{\xi}_{k-1, i})\,\bigg|\,\mathcal{F}_{k-\frac{4}{3}}\right],   x_{k-2}-x_{k-1}\right\rangle\\
             &\overset{\text{(ii)} }{=}\mathbb{E}_{\hat{\xi}_{k-1}}\left[\tfrac{1}{n_{k-1}}\textstyle\sum_{i=1}^{n_{k-1}}F(x_{k-2}, \hat{\xi}_{k-1, i})-\tfrac{1}{n_{k-1}}\textstyle\sum_{i=1}^{n_{k-1}}F(x_{k-1}, \hat{\xi}_{k-1, i})\,\bigg|\,\mathcal{F}_{k-\frac{4}{3}}\right]\\
             &\quad -\mathbb{E}_{\hat{\xi}_{k-1}}\left[\tfrac{1}{n_{k-1}}\textstyle\sum_{i=1}^{n_{k-1}}\left\langle G(x_{k-1}, \hat{\xi}_{k-1, i}),   x_{k-2}-x_{k-1}\right\rangle\,\bigg|\,\mathcal{F}_{k-\frac{4}{3}}\right]\\
            &\overset{\eqref{eqn:sample-L-transform}}{=}\mathbb{E}_{\hat{\xi}_{k-1}}\left[\tfrac{\left\|\tfrac{1}{n_{k-1}}\textstyle\sum_{i=1}^{n_{k-1}}[G(x_{k-1},\bar{\xi}_{k-1, i})-G(x_{k-2},\bar{\xi}_{k-1, i})]\right\|^2}{2 \bar{L}_{k-1}}\,\bigg|\,\mathcal{F}_{k-\frac{4}{3}}\right]\\
            &\overset{\text{(iii)}}{=}\tfrac{1}{2}\left\|\tfrac{1}{n_{k-1}}\textstyle\sum_{i=1}^{n_{k-1}}[G(x_{k-1},\bar{\xi}_{k-1, i})-G(x_{k-2},\bar{\xi}_{k-1, i})]\right\|^2\mathbb{E}_{\hat{\xi}_{k-1}}[\bar{L}_{k-1}^{-1}\,|\,\mathcal{F}_{k-\frac{4}{3}}],
            \end{aligned}
        \end{equation}
        where $\mathbb{E}_{\hat{\xi}_{k-1}}$ denotes taking expectation with respect to ${\hat{\xi}_{k-1, i}}, i\in[n_{k-1}];$ in (i), we used \autoref{a:unbiasness}, and in (ii), we used $x_{k-1}, x_{k-2}, n_{k-1}\in\mathcal{F}_{k-\frac{5}{3}};$
        in (iii), we used $n_{k-1}\in\mathcal{F}_{k-\frac{5}{3}}$ and $$\left\|\textstyle\sum_{i=1}^{n_{k-1}}\left[G(x_{k-1},\bar{\xi}_{k-1,i})-G(x_{k-2},\bar{\xi}_{k-1,i})\right]\right\|^2\in\mathcal{F}_{k-\frac{4}{3}},$$
        due to \eqref{def:filtration}.
Furthermore, by the definition of $x_{k-1}$ in \eqref{eqn:output-stochastic}, for all $k\geq 2,$
there holds
\begin{equation}\label{eqn:output-convex}
    \begin{aligned}
        &\langle\nabla f(x_{k-1}), z_{k-1}-z\rangle+\tau_{k-1}[f(x_{k-1})+ \langle \nabla f(x_{k-1}),   x_{k-2}-x_{k-1}\rangle]=\tau_{k-1}f(x_{k-1})+\langle\nabla f(x_{k-1}), x_{k-1}-z\rangle.
    \end{aligned}
\end{equation}
 Combining \eqref{eqn:step3} with \eqref{eqn:output-convex}, for all $k\geq 2,$ there holds
 \begin{equation}\label{eqn:equality-rel-output-sol}
     \begin{aligned}
         &\tau_{k-1}[f(x_{k-1})-f(x_{k-2})]+\langle\nabla f(x_{k-1}), x_{k-1}-z\rangle\\
          &\quad+\tfrac{\tau_{k-1} }{2}\left\|\tfrac{1}{n_{k-1}}\textstyle\sum_{i=1}^{n_{k-1}}[G(x_{k-1},\bar{\xi}_{k-1, i})-G(x_{k-2},\bar{\xi}_{k-1, i})]\right\|^2\mathbb{E}_{\hat{\xi}_{k-1}}[\bar{L}_{k-1}^{-1}\,|\,\mathcal{F}_{k-\frac{4}{3}}]\\
          &\overset{\eqref{eqn:step3}}{=}\tau_{k-1}f(x_{k-1})+\langle\nabla f(x_{k-1}), x_{k-1}-z\rangle-\tau_{k-1}[f(x_{k-1})+ \langle \nabla f(x_{k-1}),   x_{k-2}-x_{k-1}\rangle]\\
                  &\overset{\eqref{eqn:output-convex}}{=}\langle \nabla f(x_{k-1}), z_{k-1}-z\rangle.
     \end{aligned}
 \end{equation}
    {\noindent\bf ii)}      Notice that if there exists some $\omega\in \Omega^k/ \mathcal{N}^k$ such that
        \begin{align*}
            {\tfrac{1}{n_{k-1}}\textstyle\sum_{i=1}^{n_{k-1}}[F(x_{k-2}, \hat{\xi}_{k-1,i})-F(x_{k-1}, \hat{\xi}_{k-1,i})-\langle G(x_{k-1}, \hat{\xi}_{k-1, i}),   x_{k-2}-x_{k-1}\rangle]}=0,
        \end{align*}
             then, by \autoref{ass:finite sample convexity},
              \begin{align*}
      \left\|\tfrac{1}{n_{k-1}}\textstyle\sum_{i=1}^{n_{k-1}}\left[G(x_{k-1},\bar{\xi}_{k-1, i})-G(x_{k-2},\bar{\xi}_{k-1, i})\right]\right\|^2=0,
              \end{align*}
           and thus by \eqref{eqn:sample-wise-Lk-mini-batch}, on the event where both numerator and denominator vanish, we define $\bar{L}_{k-1} =\tfrac{0}{0}=0.$ By \eqref{eqn:step3} and \eqref{eqn:output-convex}, \eqref{eqn:equality-rel-output-sol} also holds.
           This concludes the proof.
\end{proof}
The following lemma extends the main convergence result of the deterministic AC-FGM \citep[Proposition 1]{li2025simple} to the stochastic setting; compared with the deterministic case, it features an additional regularization term $\gamma_k$ and stochastic error terms.
\begin{lemma}\label{lem:optimality-condition}
Suppose that $\beta_1=0$, $\beta_k\in(0,1)$ for all $k\ge 2$, and $\tau_k>0$ for all $k\ge 1$.
Furthermore, suppose the stepsizes $\{\gamma_k\}_{k\ge 1}$ and $\{\eta_k\}_{k\ge 1}$ satisfy $\gamma_k\ge 0$, $\eta_1>0$, and
\begin{equation}\label{eqn:stepsize-2'}
0<\eta_k \le 2(1-\beta_{k-1})(1-\beta_k)\eta_{k-1},\quad \forall\,\, k\ge 2.
\end{equation}
Then, for all $k\geq 2$ and any $z\in X$, it holds almost surely that
\begin{align*}
          &\eta_k\langle G_{k}, z_{k-1}-z\rangle+\eta_k[h(z_{k-1})-h(z)]+\tfrac{1+\gamma_k(1-\beta_k)}{2}\|z_k-y_{k-1}\|^2\notag\\
          &\leq\tfrac{1+\gamma_k(1-\beta_k)}{2\beta_k}{}\|y_{k-1}-z\|^2-\tfrac{1+\gamma_k}{2\beta_k}{}\|y_k-z\|^2+\eta_k\langle G_{k}- G_{k-1}, z_{k-1}-z_k\rangle\notag\\
          &\quad +{\tfrac{\gamma_k}{2}\|y_0-z\|^2}-{\left(\tfrac{\gamma_k}{2}-\tfrac{\eta_k\gamma_{k-1}}{4\eta_{k-1}}\right)\|z_k-y_0\|^2}.
      \end{align*}
\end{lemma}
\begin{proof}
By the optimality conditions of \eqref{eqn:gradient-step-stochastic} at $z_k$ and $z_{k-1}$, and the convexity of $h,$
for all $z\in X,$ there holds
\begin{align}
    &\langle G_k+\tfrac{z_k-y_{k-1}}{\eta_k}+\tfrac{\gamma_k(z_k-y_0)}{\eta_k}, z-z_k\rangle \geq h(z_k)-h(z),\label{eqn:optimality-k}\\
    &\langle G_{k-1}+\tfrac{z_{k-1}-y_{k-2}}{\eta_{k-1}}+\tfrac{\gamma_{k-1}(z_{k-1}-y_0)}{\eta_{k-1}}, z-z_{k-1}\rangle \geq h(z_{k-1})-h(z)\label{eqn:optimality-k-1}.
\end{align}
Choosing $z=z_{k}$ in \eqref{eqn:optimality-k-1} and  combining it with \eqref{output-center}, we have
\begin{align}\label{eqn:inter-0}
    &\langle \eta_k G_{k-1}+\tfrac{\eta_k(z_{k-1}-y_{k-1})}{(1-\beta_{k-1})\eta_{k-1}}+\tfrac{\eta_k\gamma_{k-1}(z_{k-1}-y_0)}{\eta_{k-1}}, z_k-z_{k-1}\rangle \geq \eta_k[h(z_{k-1})-h(z_k)].
\end{align}
Combining \eqref{eqn:optimality-k} with \eqref{eqn:inter-0}, we have
\begin{equation}\label{eqn:inter-1}
\begin{aligned}
    &\eta_k\langle G_k-G_{k-1}, z_{k-1}-z_k\rangle+\eta_k\langle G_k, z-z_{k-1}\rangle+\langle z_k-y_{k-1}, z-z_k\rangle\\
    &\quad+\gamma_k\langle z_k-y_0, z-z_k\rangle+\tfrac{\eta_k\langle z_{k-1}-y_{k-1}, z_k-z_{k-1}\rangle}{(1-\beta_{k-1})\eta_{k-1}}+\tfrac{\eta_{k}\gamma_{k-1}\langle z_{k-1}-y_0, z_{k}-z_{k-1}\rangle}{\eta_{k-1}}\geq \eta_{k}[h(z_{k-1})-h(z)].
\end{aligned}
\end{equation}
By a standard Euclidean identity, for all $x, y,z \in \mathbb{R}^n,$
it holds that $2\langle y-x, z-y\rangle=\|x-z\|^2-\|x-y\|^2-\|y-z\|^2,$ thus, we have
    \begin{equation}\label{eqn:inter-2}
        \begin{aligned}
            &2\langle z_k-y_{k-1}, z-z_k\rangle=\|y_{k-1}-z\|^2-\|z_k-y_{k-1}\|^2-\|z-z_k\|^2,\\
        &2\langle z_k-y_0, z-z_k\rangle=\|y_0-z\|^2-\|z_k-y_0\|^2-\|z-z_k\|^2,\\
        &2\langle z_{k-1}-y_{k-1}, z_k-z_{k-1}\rangle=\|z_k-y_{k-1}\|^2-\|z_{k-1}-y_{k-1}\|^2-\| z_k-z_{k-1}\|^2,\\
        &2\langle  z_{k-1}-y_0, z_{k}-z_{k-1}\rangle=\|z_k-y_0\|^2-\|z_{k-1}-y_0\|^2-\| z_{k}-z_{k-1}\|^2.
        \end{aligned}
    \end{equation}
Substituting \eqref{eqn:inter-2} into \eqref{eqn:inter-1}, we obtain
\begin{equation}\label{eqn:inter-3}
    \begin{aligned}
        &\eta_{k}[h(z_{k-1})-h(z)]+\eta_k\langle G_k-G_{k-1}, z_{k}-z_{k-1}\rangle+\eta_k\langle G_k, z_{k-1}-z\rangle\\
    &\leq \tfrac{\|y_{k-1}-z\|^2-\|z_k-y_{k-1}\|^2-\|z-z_k\|^2}{2}+\tfrac{\gamma_k[\|y_0-z\|^2-\|z_k-y_0\|^2-\|z-z_k\|^2]}{2}\\
     &\quad+\tfrac{\eta_k[\|z_k-y_{k-1}\|^2-\|z_{k-1}-y_{k-1}\|^2-\| z_k-z_{k-1}\|^2]}{2(1-\beta_{k-1})\eta_{k-1}}+\tfrac{\eta_k\gamma_{k-1}[\|z_k-y_0\|^2-\|z_{k-1}-y_0\|^2-\| z_{k}-z_{k-1}\|^2]}{2\eta_{k-1}}\\
      &\overset{\text{(i)}}{\leq}\tfrac{1}{2}\|y_{k-1}-z\|^2-(\tfrac{1}{2}+\tfrac{\gamma_k}{2})\|z-z_k\|^2-\tfrac{1}{2}\left(1-\tfrac{\eta_k}{2(1-\beta_{k-1})\eta_{k-1}}\right)\|z_k-y_{k-1}\|^2\\
    &\quad+\tfrac{\gamma_k}{2}[\|y_0-z\|^2-\|z_k-y_0\|^2]+\tfrac{\eta_k\gamma_{k-1}}{4\eta_{k-1}}\|z_k-y_0\|^2,
    \end{aligned}
\end{equation}
where in (i), we used the basic inequality $\|a+b\|^2\leq 2\|a\|^2+2\|b\|^2$ for all $a, b\in\mathbb{R}^n.$
Specifically, letting $a=z_{k-1}-y_{k-1}$ and $b=z_k-z_{k-1}$, we have $a+b=z_k-y_{k-1}$, and hence
\begin{align*}
   \|z_{k-1}-y_{k-1}\|^2+\| z_k-z_{k-1}\|^2\geq\tfrac{1}{2} \|z_k-y_{k-1}\|^2,
\end{align*}
and similarly, it holds that
\begin{align*}
    \|z_{k-1}-y_0\|^2+\| z_{k}-z_{k-1}\|^2\geq \tfrac{1}{2}\|z_k-y_0\|^2.
\end{align*}
Furthermore, it follows that
\begin{equation*}
    \begin{aligned}
        \|z_k-z\|^2&\overset{\eqref{output-center}}{=} \|\tfrac{1}{\beta_k}y_k-\tfrac{1-\beta_k}{\beta_k}y_{k-1}-z\|^2
\overset{\textnormal{(ii)}}{=}\tfrac{1}{\beta_k} \|y_k-z\|^2+(1-\tfrac{1}{\beta_k})\|y_{k-1}-z\|^2+(1-\beta_k)\|z_k-y_{k-1}\|^2,
    \end{aligned}
\end{equation*}
where in (ii), we use the quadratic identity $\|\alpha a+(1-\alpha) b\|^2=\alpha \|a\|^2+(1-\alpha)\|b\|^2-\alpha(1-\alpha)\|a-b\|^2,$ for all $\alpha\in\mathbb{R}, a, b\in\mathbb{R}^n.$
Combining it with \eqref{eqn:inter-3}, we have
\begin{align*}
     &\eta_{k}[h(z_{k-1})-h(z)]+\eta_k\langle G_k-G_{k-1}, z_{k}-z_{k-1}\rangle+\eta_k\langle G_k, z_{k-1}-z\rangle\notag\\
      &\leq \tfrac{1}{2}\left[\tfrac{1}{\beta_k}+\gamma_k\left(\tfrac{1}{\beta_k}-1\right)\right]\|y_{k-1}-z\|^2-\tfrac{1}{\beta_k}\left(\tfrac{1}{2}+\tfrac{\gamma_k}{2}\right)\|y_{k}-z\|^2-\left(\tfrac{\gamma_k}{2}-\tfrac{\eta_k\gamma_{k-1}}{4\eta_{k-1}}\right)\|z_k-y_0\|^2\notag\\
      &\quad-\tfrac{1}{2}\left(1+(1+{\gamma_k)(1-\beta_k)}-\tfrac{\eta_k}{2(1-\beta_{k-1})\eta_{k-1}}\right)\|z_k-y_{k-1}\|^2+\tfrac{\gamma_k}{2}\|y_0-z\|^2.
\end{align*}
Substituting the stepsize condition \eqref{eqn:stepsize-2'} into it concludes the proof.
\end{proof}

\vgap

We next define several error terms and establish a one-step recursion for \autoref{alg:main-single-objective-stochastic}, which forms the foundation of the convergence analysis. This recursion also highlights the role of the local smoothness estimator $\bar{L}_k$ in ensuring convergence. For all $k \ge 1$, define the stochastic gradient errors as
\begin{equation}\label{eqn:diff-noise-def}
    \begin{aligned}
           \delta_{k,i}(x)&\coloneqq G(x, \xi_{k, i})-\nabla f(x),\quad
                   \bar{\delta}_{k,i}(x)\coloneqq{G(x,\bar{\xi}_{k, i})}-\nabla f(x),
    \end{aligned}
\end{equation}
              and     define
the error function related to the stochasticity of the gradient as
\begin{equation}\label{eqn:e-k}
    \begin{aligned}
        \|\Delta_k\|\coloneqq&\tfrac{\|\textstyle\sum_{i=1}^{m_k} \delta_{k,i}(x_{k-1})\|^2}{m_k^2}+\tfrac{\|\textstyle\sum_{i=1}^{m_{k-1}} \delta_{k-1,i}(x_{k-2})\|^2}{m_{k-1}^2}+\tfrac{\|\textstyle\sum_{i=1}^{n_{k-1}}\bar{\delta}_{k-1,i}(x_{k-1})\|^2}{n_{k-1}^2}+\tfrac{\|\textstyle\sum_{i=1}^{n_{k-1}}\bar{\delta}_{k-1,i}(x_{k-2})\|^2}{n_{k-1}^2}.
    \end{aligned}
\end{equation}
Furthermore, for all $k \ge 2$, recall that $T(x_{k-1}, \hat{\xi}_{k-1})$ is defined in \eqref{eqn:empirical first-order taylor remainder} by
              \begin{align*}
                  &T(x_{k-1}, \hat{\xi}_{k-1})=  \tfrac{1}{n_{k-1} }\sum_{i=1}^{n_{k-1} }[F(x_{k-2} , \hat{\xi}_{k-1,i} )-F(x_{k-1} , \hat{\xi}_{k-1,i} )-\langle G(x_{k-1} , \hat{\xi}_{k-1,i} ),   x_{k-2} -x_{k-1} \rangle].
              \end{align*}
\begin{lemma}[One-step recursion]\label{prop:main-recursion}
    Suppose the assumptions in \autoref{lem:descent-one-step-equality} and \autoref{lem:optimality-condition} hold. Furthermore, suppose $\zeta_k>0$ for all $k\geq 2$. Suppose $\eta_1>0$, and $\eta_k$ satisfies
\begin{equation*}
        \begin{aligned}
        \eta_k\bar{L}_{k-1}\leq \tfrac{\zeta_k\tau_{k-1}}{8},\quad\text{and}\quad
      {\eta_k\gamma_{k-1}}{}\leq 2\gamma_k\eta_{k-1},\quad
     \forall\quad k\geq 2.
    \end{aligned}
    \end{equation*}
   Then, for all $k\geq 2$ and any $z\in X$, it holds almost surely that
   \begin{equation}\label{eqnindetermdeia 4}
        \begin{aligned}
            &\eta_k\left\{(\tau_{k-1}+1)[\Psi(x_{k-1})-\Psi(z)]-\tau_{k-1}[\Psi(x_{k-2})-\Psi(z)]\right\}\\
            &\leq\,\tfrac{1+\gamma_k(1-\beta_k)}{2\beta_k}\|y_{k-1}-z\|^2-{}\tfrac{1+\gamma_k}{2\beta_k}\|y_k-z\|^2+\tfrac{\gamma_k}{2}{\|y_0-z\|^2} -\eta_k\langle G_{k}-\nabla f(x_{k-1}), z_{k-1}-z\rangle\\
               &\quad+\tfrac{\zeta_k(1-\beta_{k-1})^2}{4}\|z_{k-1}-y_{k-2}\|^2-\left[\tfrac{1}{2}-\tfrac{\zeta_k}{4}+\tfrac{\gamma_k(1-\beta_k)}{2}\right]\|z_k-y_{k-1}\|^2\\
&\quad+\tfrac{16{{\eta_k^2\|\Delta_k\|}{}}}{\zeta_k}+\eta_k \tau_{k-1} (\tilde{L}_{k-1}-L_{k-1})
    \|x_{k-1}-x_{k-2}\|^2.
        \end{aligned}
    \end{equation}
\end{lemma}
  \begin{proof}
By \autoref{lem:optimality-condition} and the stepsize condition ${\eta_k\gamma_{k-1}}{}\leq 2\gamma_k\eta_{k-1}$,
for all $k\geq 2,$
there holds
      \begin{align*}
          &\eta_k  \left[\langle \nabla f(x_{k-1}), z_{k-1}-z\rangle+\langle G_{k}-\nabla f(x_{k-1}), z_{k-1}-z\rangle\right]+{ }\eta_k[h(z_{k-1})-h(z)]{}\notag\\
          &\leq\tfrac{  1+\gamma_k(1-\beta_k)}{2\beta_k}\|y_{k-1}-z\|^2-\tfrac{  1+\gamma_k}{2\beta_k}\|y_k-z\|^2+\eta_k\langle G_{k}- G_{k-1}, z_{k-1}-z_k\rangle\-\tfrac{  1+\gamma_k(1-\beta_k)}{2}{}\|z_k-y_{k-1}\|^2+{{}\tfrac{ \gamma_k  }{2}\|y_0-z\|^2}.
      \end{align*}
     Combining it with \autoref{lem:descent-one-step-equality}, for all $k\geq 2,$ there holds
     \begin{equation}\label{eqn:intermedia1}
         \begin{aligned}
             &  \eta_k\left(\tau_{k-1}[f(x_{k-1})-f(x_{k-2})]+{ }\langle\nabla f(x_{k-1}), x_{k-1}-z\rangle+\langle G_{k}-\nabla f(x_{k-1}), z_{k-1}-z\rangle\right)+{ }  \eta_k[h(z_{k-1})-h(z)]\\
                &{\leq}{}\tfrac{1+\gamma_k(1-\beta_k)}{2\beta_k}\|y_{k-1}-z\|^2-{}\tfrac{1+\gamma_k}{2\beta_k}\|y_k-z\|^2+\tfrac{\gamma_k   }{2}{\|y_0-z\|^2}-\tfrac{  1+\gamma_k(1-\beta_k)}{2}{ }\|z_k-y_{k-1}\|^2\\
               &\quad\underbrace{+{ }  \eta_k\left\langle\tfrac{1}{m_{k}}\textstyle\sum_{i=1}^{m_k}G(x_{k-1}, \xi_{k, i})-\tfrac{1}{n_{k-1}}\textstyle\sum_{i=1}^{n_{k-1}}G(x_{k-1},\bar{\xi}_{k-1, i}), z_{k-1}-z_k\right\rangle}_{\texttt{Term I}}\\
               &\quad\underbrace{-{ }  \eta_k\left\langle\tfrac{1}{m_{k-1}}\textstyle\sum_{i=1}^{m_{k-1}}G(x_{k-2}, \xi_{k-1, i})-\tfrac{1}{n_{k-1}}\textstyle\sum_{i=1}^{n_{k-1}}G(x_{k-2},\bar{\xi}_{k-1, i}), z_{k-1}-z_k\right\rangle}_{\texttt{Term II}}\\
               &\quad\underbrace{+{ }  \eta_k\left\langle \tfrac{1}{n_{k-1}}\textstyle\sum_{i=1}^{n_{k-1}}[G(x_{k-1},\bar{\xi}_{k-1,i})-G(x_{k-2},\bar{\xi}_{k-1,i})], z_{k-1}-z_k\right\rangle}_{\texttt{Term III}}\\
              &\quad-\tfrac{{ }  \eta_k \tau_{k-1} }{2 n_{k-1}^2}\left\|\textstyle\sum_{i=1}^{n_{k-1}}[G(x_{k-1},\bar{\xi}_{k-1,i})-G(x_{k-2},\bar{\xi}_{k-1,i})]\right\|^2\mathbb{E}_{\hat{\xi}_{k-1}}[\bar{L}_{k-1}^{-1}\,|\,\mathcal{F}_{k-\frac{4}{3}}].
         \end{aligned}
     \end{equation}
              We proceed with bounding the three inner products in \eqref{eqn:intermedia1}. For \texttt{Term I}, it holds that
              \begin{align*}
                  \texttt{Term I}
              &\overset{\text{(i)}}{\leq} \tfrac{16{{  \eta_k^2}{}}}{\zeta_k m_k^2}\|\textstyle\sum_{i=1}^{m_k}[G(x_{k-1}, \xi_{k, i})-\nabla f(x_{k-1})]\|^2\\
              &\quad+\tfrac{16{{  \eta_k^2}{}}}{\zeta_k n_{k-1}^2}\|\textstyle\sum_{i=1}^{n_{k-1}}{[G(x_{k-1},\bar{\xi}_{k-1, i})-\nabla f(x_{k-1})]}\|^2+ \tfrac{\zeta_k  }{32}\|z_k-z_{k-1}\|^2,
              \end{align*}
              where in (i), we inserted $\eta_k\nabla f(x_{k-1}),$
              used the condition that $\zeta_k>0,$ a.s. and
            Young inequality.
              Similarly, by inserting $\eta_k\nabla f(x_{k-2})$ into \texttt{Term II}, we have
              \begin{align*}
                          \texttt{Term II}&\leq{{}{}}\tfrac{16  \eta_k^2}{\zeta_{k}m_{k-1}^2}\|\textstyle\sum_{i=1}^{m_{k-1}}[G(x_{k-2}, \xi_{k-1, i})-\nabla f(x_{k-2})]\|^2\\
                          &\quad+\tfrac{16  \eta_k^2}{\zeta_{k}n_{k-1}^2}\|\textstyle\sum_{i=1}^{n_{k-1}}[G(x_{k-2}, \bar{\xi}_{k-1, i})-\nabla f(x_{k-2})]\|^2+ \tfrac{  \zeta_k}{32}\|z_k-z_{k-1}\|^2.
              \end{align*}
              For \texttt{Term III}, by Young inequality, for all $\zeta_k>0$ a.s.,
              there holds
              \begin{align*}
                  \texttt{Term III}
                  & \leq\tfrac{4   \eta_k^2}{\zeta_kn_{k-1}^2}\left\|\textstyle\sum_{i=1}^{n_{k-1}}{[G(x_{k-1},\bar{\xi}_{k-1,i})-G(x_{k-2},\bar{\xi}_{k-1,i})]}\right\|^2+\tfrac{   \zeta_k}{16}\|z_k-z_{k-1}\|^2\\
              & {\overset{\textnormal{(ii)}}{\leq} \tfrac{\eta_k\tau_{k-1}}{2n_{k-1}^2 }\left\|\textstyle\sum_{i=1}^{n_{k-1}}{[G(x_{k-1},\bar{\xi}_{k-1,i})-G(x_{k-2},\bar{\xi}_{k-1,i})]}\right\|^2\bar{L}_{k-1}^{-1}+\tfrac{   \zeta_k}{16}\|z_k-z_{k-1}\|^2}
              \end{align*}
           where in (ii), we used $\eta_k\bar{L}_{k-1}\leq \tfrac{\zeta_k\tau_{k-1}}{8}$.
              Notice that if there exists some $\omega\in \Omega^k/ \mathcal{N}^k$ such that
        \begin{align*}
            {\tfrac{1}{n_{k-1}}\textstyle\sum_{i=1}^{n_{k-1}}[F(x_{k-2}, \hat{\xi}_{k-1,i})-F(x_{k-1}, \hat{\xi}_{k-1,i})-\langle G(x_{k-1}, \hat{\xi}_{k-1, i}),   x_{k-2}-x_{k-1}\rangle]}=0,
        \end{align*}
             then, by \autoref{ass:finite sample convexity},
              \begin{align*}
      \left\|\tfrac{1}{n_{k-1}}\textstyle\sum_{i=1}^{n_{k-1}}\left[G(x_{k-1},\bar{\xi}_{k-1, i})-G(x_{k-2},\bar{\xi}_{k-1, i})\right]\right\|^2=0,
              \end{align*}
           and thus
          \texttt{Term III} vanishes on those $\omega,$ and therefore does not contribute to the integral of the error.
Substituting the bounds of \texttt{Term I, II, III} into \eqref{eqn:intermedia1}, for all $k\geq 2,$ we have
\begin{equation}
 \begin{aligned}\label{eqnindetermdeia 3}
& \,\, \eta_k\left\{\tau_{k-1}(f(x_{k-1})-f(x_{k-2}))+[f(x_{k-1})-f(z)]+\langle G_{k}-\nabla f(x_{k-1}), z_{k-1}-z\rangle\right\}\\
&\overset{\text{(iii)}}{\leq}  \eta_k\left[\tau_{k-1}(f(x_{k-1})-f(x_{k-2}))+{ }\langle\nabla f(x_{k-1}), x_{k-1}-z\rangle+\langle G_{k}-\nabla f(x_{k-1}), z_{k-1}-z\rangle\right]\\
&\overset{\text{(iv)}}{\leq}\tfrac{1+\gamma_k(1-\beta_k)}{2\beta_k}\|y_{k-1}-z\|^2-{}\tfrac{1+\gamma_k}{2\beta_k}\|y_k-z\|^2+\tfrac{  \gamma_k }{2}{\|y_0-z\|^2}\\
&\quad-\tfrac{1+\gamma_k(1-\beta_k)}{2}\|z_k-y_{k-1}\|^2+\tfrac{\zeta_k  }{8}{ }\|z_k-z_{k-1}\|^2-  \eta_k[h(z_{k-1})-h(z)]+\tfrac{16  {{\eta_k^2}{}}}{\zeta_k}\|\Delta_k\|\\
              &\quad+\left(\bar{L}_{k-1}^{-1}-\mathbb{E}_{\hat{\xi}_{k-1}}[\bar{L}_{k-1}^{-1}\,|\,\mathcal{F}_{k-\frac{4}{3}}]\right)\tfrac{{ }  \eta_k \tau_{k-1} }{2 n_{k-1}^2}\left\|\textstyle\sum_{i=1}^{n_{k-1}}[G(x_{k-1},\bar{\xi}_{k-1,i})-G(x_{k-2},\bar{\xi}_{k-1,i})]\right\|^2\\
              &\overset{\text{(v)}}{=}\tfrac{1+\gamma_k(1-\beta_k)}{2\beta_k}\|y_{k-1}-z\|^2-{}\tfrac{1+\gamma_k}{2\beta_k}\|y_k-z\|^2+\tfrac{  \gamma_k }{2}{\|y_0-z\|^2}-\tfrac{1+\gamma_k(1-\beta_k)}{2}\|z_k-y_{k-1}\|^2\\
&\quad+\tfrac{\zeta_k  }{8}{ }\|z_k-z_{k-1}\|^2-  \eta_k[h(z_{k-1})-h(z)]+\tfrac{16  {{\eta_k^2}{}}}{\zeta_k}\|\Delta_k\|+\eta_k \tau_{k-1} (\tilde{L}_{k-1}-L_{k-1})
    \|x_{k-1}-x_{k-2}\|^2,
              \end{aligned}
              \end{equation}
             where in (iii), we used the convexity of $f$; in (iv), we substituted \texttt{Term I, II, III} into \eqref{eqn:intermedia1} and  {used the definition of $\|\Delta_k\|$ in \eqref{eqn:e-k}$;$}
in (v), we used \autoref{lem:bias-local cocoercivity-based smoothnes}.

It remains to bound $\|z_k-z_{k-1}\|^2.$
        By the basic inequality, there holds
        \begin{equation}\label{eqn:z_k-z_k-1}
            \begin{aligned}
                 \|z_k-z_{k-1}\|^2&\overset{\eqref{output-center}}
              =\|z_k-y_{k-1}-(1-\beta_{k-1})(z_{k-1}-y_{k-2})\|^2\\
              &\,\,\leq 2(1-\beta_{k-1})^2\|z_{k-1}-y_{k-2}\|^2+2\|z_k-y_{k-1}\|^2.
            \end{aligned}
        \end{equation}
            Furthermore, by the convexity of $h$ and \eqref{eqn:output-stochastic}, we have
          \begin{align}\label{eqn:h-bound}
              h(x_k)\leq \tfrac{\tau_k} {\tau_k+1}h(x_{k-1})+\tfrac{1}{{\tau_k+1}} h(z_{k}).
          \end{align}
        Combining \eqref{eqn:z_k-z_k-1} and  \eqref{eqn:h-bound} with \eqref{eqnindetermdeia 3}
       concludes the proof.
\end{proof}
\vgap

 Notice that step (iv) in \eqref{eqnindetermdeia 3}
 highlights that, although the local cocoercivity parameter $\bar{L}_{k-1}$ need not be an unbiased estimator of its deterministic counterpart defined in \eqref{eqn:determinstic}, the induced error can still be controlled through the fluctuation of the sample local smoothness estimator $\tilde{L}_{k-1}$ around its mean ${L}_{k-1}$. This also indicates that the variance of the local smoothness estimator $v_{k-1}$ will play an important role in the following analysis.

We next establish the following trajectory wise convergence guarantee for stochastic AC-FGM (\autoref{alg:main-single-objective-stochastic}), which serves as the foundation for Theorems~\ref{cor:main-fixed-T-const}--\ref{eqn:high-main-2}.

We define the following quantity to characterize the convergence rate.
For any $\gamma_k\in [0,1), \beta_k\in [0,1),$  define
\begin{equation}\label{ean:Gamma}
       \Gamma_k=\left\{
      \begin{array}{ll}
       1, &\quad\text{if }\quad  k=1, \medskip \\
         \Gamma_{k-1}\left(1-\tfrac{\beta_k\gamma_k}{1+\gamma_k}\right), &\quad\text{if }\quad k>1. \\
      \end{array}
      \right.
\end{equation}

\begin{proposition}\label{thm:main-expectation}
Suppose \autoref{a:unbiasness} and \autoref{ass:finite sample convexity} hold, and $m_k\in\mathcal{F}_{k-1}$ and $n_{k-1}\in \mathcal{F}_{k-\frac{5}{3}}$. Furthermore, suppose $\tau_k > 0$ for all $k \ge 1$. Suppose also that $\gamma_k\in[0,1]$ for all $k\geq 1$, $\beta_1=0$, and $\beta_k \equiv\beta\in(0,1)$ for all $k\geq 2$, with $\zeta_k$ chosen as
\begin{equation}\label{eqn:beta-k}
\begin{aligned}
\zeta_{k}\coloneqq \tfrac{1+\gamma_{k}(1-\beta)}{1+\gamma_{k-1}},\quad\forall \,k\geq 2.
\end{aligned}
\end{equation}
Finally, suppose $\eta_1>0$ and, for all $k\geq 2$, $\eta_{k}$ satisfies
\begin{equation}\label{eqn:stepsize-3}
    \begin{aligned}
        &\eta_k\bar{L}_{k-1}\leq \tfrac{\zeta_k\tau_{k-1}}{8},\tfrac{\eta_{k+1}\tau_{k}}{\Gamma_{k+1}(1+\gamma_{k+1})}\leq \tfrac{\eta_{k}(\tau_{k-1}+1)}{\Gamma_{k}(1+\gamma_k)},\\
       & \eta_k\leq 2(1-\beta)^2{}\eta_{k-1},\,\,\eta_k\leq\tfrac{2\gamma_k\eta_{k-1}}{\gamma_{k-1}}.
    \end{aligned}
\end{equation}
Then, for any sequence $\{a_k\}_{k\geq 0}$ satisfying $a_{k+1}\geq a_{k}>0,$ for all $k\geq 0$ and $a_{-1}=a_0$, for any $N\geq 1$ and all $x^*\in X$, it holds almost surely that
  \begin{equation}\label{eqn:final-iterate}
      \begin{aligned}
&\tfrac{  \eta_{N+1}\beta(\tau_{N}+1)[\Psi(x_{N})-\Psi(x^*)]}{a_{N+1}(1+\gamma_{N+1})\Gamma_{N+1}}+\tfrac{\|y_{N+1}-x^*\|^2}{2a_N\Gamma_{N+1}}+\textstyle\sum_{k=3}^{N+1}\tfrac{\beta^2\zeta_k\|z_{k}-y_{k-1}\|^2}{4a_{k-1}\Gamma_k(1+\gamma_k)}\\
      &\leq \tfrac{\beta\eta_2\tau_1[\Psi(x_{0})-\Psi(x^*)]}{  {a_0}(1+\gamma_2)\Gamma_2}+\tfrac{\|x^*-y_{0}\|^2}{2}\left[\tfrac{2}{a_0}+\textstyle\sum_{k=2}^{N+1}\tfrac{\beta\gamma_k}{a_{k-1}\Gamma_k(1+\gamma_k)}\right]+\tfrac{\eta_1^2\left\| {G_1+s_0}\right\|^2}{a_0(1+\gamma_1)^3}\\
      & \quad+\tfrac{\eta_1[\langle {G}_1, x^*-x_0\rangle+h(x^*)-    h(x_0)]}{a_0(1+\gamma_1)^2}+\textstyle\sum_{k=2}^{N+1}\tfrac{\beta\eta_k\langle G_{k}-\nabla f(x_{k-1}), x^*-z_{k-1}\rangle}{a_{k-1}\Gamma_k(1+\gamma_k)}+\textstyle\sum_{k=2}^{N+1}\tfrac{8\eta_{k-1}^2\|\Delta_k\|}{a_{k-1}\beta\Gamma_{k-1}}\\
&\quad+\textstyle\sum_{k=2}^{N+1}\tfrac{1}{\Gamma_k(1+\gamma_k)}\left(\tfrac{{9n_{k-1}}\beta^2\lambda {(\tilde{L}_{k-1}-L_{k-1})^2}}{{a_{k-1}}\tau_{k-1}^2}+\tfrac{\eta_k^2 a_{k-2}}{{36\lambda}n_{k-1}} \right)\left(\tfrac{\|z_{k-1}-x^*\|^2}{a_{k-2}}+\tfrac{\|x^*-x_{k-2}\|^2}{a_{k-3}}\right),
      \end{aligned}
  \end{equation}
 {where $G_1$ is defined in \eqref{eqn:stochastic-gradient},}
$\|\Delta_k\|$ is defined in \eqref{eqn:e-k},
and $\lambda>0$ is arbitrary.
\end{proposition}
\begin{proof}
 {It is immediate from \eqref{eqn:beta-k} that $\zeta_k>0$. Moreover, under \autoref{a:unbiasness}, \autoref{ass:finite sample convexity}, and the stated conditions on $\gamma_k$, $\tau_k$, $\beta_k$, $\eta_k$, $m_k$, and $n_{k-1}$, \autoref{prop:main-recursion} holds.} Hence, by taking $z=x^*$ in \eqref{eqnindetermdeia 4}, multiplying both sides by $\tfrac{2\beta}{a_{k-1}\Gamma_k(1+\gamma_k)}$, and using the definition of $\Gamma_k$ in \eqref{ean:Gamma}, we obtain, for all $k\geq 3$,
\begin{equation}\label{eqn:main-exp-0}
    \begin{aligned}
     &\,\,\,\tfrac{\|y_k-x^*\|^2}{a_{k-1}\Gamma_k}+\tfrac{2\beta\eta_k{(\tau_{k-1}+1)}[\Psi(x_{k-1})-\Psi(x^*)]}{a_{k}\Gamma_k(1+\gamma_k)}-\tfrac{2\beta\eta_k\tau_{k-1}[\Psi(x_{k-2})-\Psi(x^*)]}{a_{k-1}\Gamma_k(1+\gamma_k)}\\
        &\,\,\,\overset{\text{(i)}}{\leq}\tfrac{\|y_k-x^*\|^2}{a_{k-1}\Gamma_k}+\tfrac{2\beta\eta_k\{{(\tau_{k-1}+1)}[\Psi(x_{k-1})-\Psi(x^*)]-\tau_{k-1}[\Psi(x_{k-2})-\Psi(x^*)]\}}{a_{k-1}\Gamma_k(1+\gamma_k)}\\
&\overset{\eqref{eqnindetermdeia 4}}{\leq}\tfrac{\|y_{k-1}-x^*\|^2}{a_{k-1}\Gamma_{k-1}}-\tfrac{\beta[1-\zeta_k+2\gamma_k(1-\beta)]}{2a_{k-1}\Gamma_k(1+\gamma_k)}\|z_{k}-y_{k-1}\|^2+\tfrac{32\beta\eta_k^2\|\Delta_k\|}{a_{k-1}\zeta_k\Gamma_k(1+\gamma_k)}\\
              &\quad+\tfrac{\beta}{a_{k-1}\Gamma_k(1+\gamma_k)}\left[\tfrac{\zeta_k(1-\beta)^2}{2}\|z_{k-1}-y_{k-2}\|^2-\tfrac{1}{2}\|z_k-y_{k-1}\|^2\right]\\
&\quad+\tfrac{\beta\gamma_k\|y_0-x^*\|^2}{a_{k-1}\Gamma_k(1+\gamma_k)}+\tfrac{2\beta\eta_k\langle G_{k}-\nabla f(x_{k-1}), x^*-z_{k-1}\rangle}{a_{k-1}\Gamma_k(1+\gamma_k)}+\tfrac{2\beta\eta_k \tau_{k-1} (\tilde{L}_{k-1}-L_{k-1})
   \|x_{k-1}-x_{k-2}\|^2}{a_{k-1}\Gamma_k(1+\gamma_k)}\\
&\overset{\text{(ii)}}{\leq}\tfrac{\|y_{k-1}-x^*\|^2}{a_{k-2}\Gamma_{k-1}}-\tfrac{\beta[1-{\zeta_k}+2\gamma_k(1-\beta)]}{2a_{k-1}\Gamma_k(1+\gamma_k)}\|z_{k}-y_{k-1}\|^2-\tfrac{\beta^2\zeta_k}{2a_{k-1}\Gamma_k(1+\gamma_k)}\|z_{k-1}-y_{k-2}\|^2\\
              &\quad+\tfrac{\beta}{\Gamma_k(1+\gamma_k)}\left[\tfrac{\zeta_k}{2a_{k-1}}\|z_{k-1}-y_{k-2}\|^2-\tfrac{1}{2a_k}\|z_k-y_{k-1}\|^2\right]+\tfrac{32\beta\eta_k^2\|\Delta_k\|}{a_{k-1}\zeta_k\Gamma_k(1+\gamma_k)}\\
 &\quad+\tfrac{\beta\gamma_k\|y_0-x^*\|^2}{a_{k-1}\Gamma_k(1+\gamma_k)}+\tfrac{2\beta\eta_k\langle G_{k}-\nabla f(x_{k-1}),  {x^*}-z_{k-1}\rangle}{a_{k-1}\Gamma_k(1+\gamma_k)}+\tfrac{2\beta\eta_k \tau_{k-1} (\tilde{L}_{k-1}-L_{k-1})
   \|x_{k-1}-x_{k-2}\|^2}{a_{k-1}\Gamma_k(1+\gamma_k)},
    \end{aligned}
\end{equation}
where in (i) we used the monotonicity of $a_k$; in (ii), we used
$(1-\beta)^2\leq 1-\beta$ and $a_{k}\geq a_{k-1}$.
Similarly, when $k=2$, we have
\begin{equation}\label{eqn:main-exp-0-0}
    \begin{aligned}
     &\tfrac{\|y_2-x^*\|^2}{a_{1}\Gamma_k}+\tfrac{2\beta\eta_2{(\tau_{1}+1)}[\Psi(x_{1})-\Psi(x^*)]}{a_{2}\Gamma_2(1+\gamma_2)}-\tfrac{2\beta\eta_2\tau_{1}[\Psi(x_{0})-\Psi(x^*)]}{a_{1}\Gamma_2(1+\gamma_2)}\\
&{\leq}\tfrac{\|y_{1}-x^*\|^2}{a_{0}\Gamma_{1}}-\tfrac{\beta[1-{\zeta_2}+2\gamma_2(1-\beta)]}{2a_{1}\Gamma_2(1+\gamma_2)}\|z_{2}-y_{1}\|^2+\tfrac{\beta}{\Gamma_2(1+\gamma_2)}\left[\tfrac{\zeta_2}{2a_{1}}\|z_{1}-y_{0}\|^2-\tfrac{1}{2a_2}\|z_2-y_{1}\|^2\right]+\tfrac{32\beta\eta_2^2\|\Delta_2\|}{a_{1}\zeta_2\Gamma_2(1+\gamma_2)}\\
 &\quad+\tfrac{\beta\gamma_2\|y_0-x^*\|^2}{a_{1}\Gamma_2(1+\gamma_2)}+\tfrac{2\beta\eta_2\langle G_{2}-\nabla f(x_{1}),  {x^*}-z_{1}\rangle}{a_{1}\Gamma_2(1+\gamma_2)}+\tfrac{2\beta\eta_2 \tau_{1} (\tilde{L}_{1}-L_{1})
   \|x_{1}-x_{0}\|^2}{a_{1}\Gamma_2(1+\gamma_2)}.
    \end{aligned}
\end{equation}
By the definitions of $\zeta_k$ in \eqref{eqn:beta-k} and $\Gamma_k$ in \eqref{ean:Gamma}, it holds that
\begin{align}\label{eqn:main-exp-2}
   &\tfrac{\beta\zeta_{k+1}}{\Gamma_{k+1}(1+\gamma_{k+1})}=\tfrac{\beta}{\Gamma_k(1+\gamma_k)}\quad\text{and}\quad \tfrac{1-{\zeta_k}+2\gamma_k(1-\beta)}{2a_{k-1}\Gamma_k(1+\gamma_k)}={{\tfrac{\gamma_{k-1}+\gamma_{k}(1-\beta)+2\gamma_k(1-\beta)\gamma_{k-1}}{2a_{k-1}\Gamma_k(1+\gamma_k)(1+\gamma_{k-1})}}}\geq 0,\quad
   \forall\,\, k\geq 2.
\end{align}
Furthermore, by the stepsize condition
$\frac{\eta_{k+1}\tau_{k}}{\Gamma_{k+1}(1+\gamma_{k+1})}
\le
\frac{\eta_{k}(\tau_{k-1}+1)}{\Gamma_{k}(1+\gamma_k)}$
in \eqref{eqn:stepsize-3}, it follows that
\begin{align*}
   &\,\,\textstyle\sum_{k=1}^N\left[\tfrac{{\eta_{k+1}(\tau_{k}+1)}[\Psi(x_{k})-\Psi(x^*)]}{a_{k+1}(1+\gamma_{k+1})\Gamma_{k+1}}-\tfrac{\eta_{k+1}\tau_{k}[\Psi(x_{k-1})-\Psi(x^*)]}{a_k(1+\gamma_{k+1})\Gamma_{k+1}}\right]\notag\\
       &\,\,=\textstyle\sum_{k=1}^{N-1}\tfrac{\Psi(x_{k})-\Psi(x^*)}{a_{k+1}}\left[\tfrac{\eta_{k+1}(\tau_{k}+1)}{(1+\gamma_{k+1})\Gamma_{k+1}}-\tfrac{\eta_{k+2}\tau_{k+1}}{(1+\gamma_{k+2})\Gamma_{k+2}}\right]+\tfrac{\eta_{N+1}(\tau_{N}+1)[\Psi(x_{N})-\Psi(x^*)]}{a_{N+1}(1+\gamma_{N+1})\Gamma_{N+1}}-\tfrac{\eta_2\tau_1[\Psi(x_{0})-\Psi(x^*)]}{a_1(1+\gamma_2)\Gamma_2}\notag\\
       &\overset{\eqref{eqn:stepsize-3}}{\geq }\tfrac{\eta_{N+1}(\tau_{N}+1)[\Psi(x_{N})-\Psi(x^*)]}{a_{N+1}(1+\gamma_{N+1})\Gamma_{N+1}}-\tfrac{\eta_2\tau_1[\Psi(x_{0})-\Psi(x^*)]}{a_1(1+\gamma_2)\Gamma_2}.
\end{align*}
Substituting \eqref{eqn:main-exp-2} into \eqref{eqn:main-exp-0} and \eqref{eqn:main-exp-0-0}, and summing \eqref{eqn:main-exp-0} from $3$ to $N+1$ together with \eqref{eqn:main-exp-0-0}, we obtain
\begin{equation}\label{prop-1-intermediate}
    \begin{aligned}
         &\tfrac{\beta\eta_{N+1}(\tau_{N}+1)[\Psi(x_{N})-\Psi(x^*)]}{a_{N+1}(1+\gamma_{N+1})\Gamma_{N+1}}+\tfrac{\|y_{N+1}-x^*\|^2}{2a_N\Gamma_{N+1}}+\textstyle\sum_{k=3}^{N+1}\tfrac{\beta^2\zeta_k\|z_{k-1}-y_{k-2}\|^2}{4a_{k-1}\Gamma_k(1+\gamma_k)}\\
              &{\leq}\tfrac{\beta\eta_2\tau_1[\Psi(x_{0})-\Psi(x^*)]}{a_1(1+\gamma_2)\Gamma_2}+ \tfrac{\|y_{1}-x^*\|^2}{2a_0\Gamma_{1}}+\tfrac{{\beta}\|z_{1}-y_{0}\|^2}{4a_1\Gamma_1(1+\gamma_1)}-\tfrac{\beta\|z_{N+1}-y_{N}\|^2}{4a_{N+1}\Gamma_{N+1}(1+\gamma_{N+1})}+\textstyle\sum_{k=2}^{N+1}\tfrac{16\beta\eta_k^2\|\Delta_k\|}{a_{k-1}\zeta_k\Gamma_k(1+\gamma_k)}{{}{}}\\
             &\quad+\textstyle\sum_{k=2}^{N+1}\left[\tfrac{\beta\gamma_k\|y_0-x^*\|^2}{2a_{k-1}\Gamma_k(1+\gamma_k)}+\tfrac{\beta\eta_k\langle G_{k}-\nabla f(x_{k-1}), x^*-z_{k-1}\rangle}{a_{k-1}\Gamma_k(1+\gamma_k)}+\tfrac{\beta\eta_k \tau_{k-1} (\tilde{L}_{k-1}-L_{k-1})
   \|x_{k-1}-x_{k-2}\|^2}{a_{k-1}\Gamma_k(1+\gamma_k)}\right].
    \end{aligned}
\end{equation}
Furthermore, observe that for all $k\geq 2$, it holds that
\begin{equation}\label{eqn:noise-initial}
             \begin{aligned}
                 \tfrac{\beta\eta_k^2}{\zeta_k\Gamma_k(1+\gamma_k)}&\overset{\eqref{ean:Gamma}}{=}\tfrac{\beta\eta_k^2}{\zeta_k\Gamma_{k-1}(1+\gamma_k-\beta\gamma_k)}\overset{\eqref{eqn:beta-k}}{=}\tfrac{\beta^2(1+\gamma_{k-1})\eta_k^2}{\Gamma_{k-1}(1+\gamma_{k}-\beta\gamma_k)^2\beta}\overset{\eqref{eqn:stepsize-3}}{\leq}\tfrac{4\beta^2(1+\gamma_{k-1})(1-\beta)^4\eta_{k-1}^2}{\Gamma_{k-1}(1+\gamma_k-\beta\gamma_k)^2\beta}\overset{\text{(iii)}}{\leq}\tfrac{\eta_{k-1}^2}{2\beta\Gamma_{k-1}},
             \end{aligned}
         \end{equation}
where in (iii), we used $\gamma_{k-1}\in[0,1]$, $\beta\in[0,1)$, and $\beta(1-\beta)\leq \tfrac{1}{4}$.
Notice that by the definition of $\tilde{L}_{k} (\hat{\xi}_{k})$ and $L_k$ in \eqref{eqn:L-k}, we have
\begin{equation}\label{eqn:prop-intermediate-3}
   \begin{aligned}
  &\,\,\tfrac{\eta_k\tau_{k-1}\beta(\tilde{L}_{k-1}-L_{k-1})}{2a_{k-1} }\|x_{k-1}-x_{k-2}\|^2\\
   &\overset{\eqref{eqn:output-stochastic}}{=}\tfrac{\eta_k\tau_{k-1}\beta(\tilde{L}_{k-1}-L_{k-1})}{2a_{k-1} (1+\tau_{k-1})^2} \|z_{k-1}-x_{k-2}\|^2\\
   &\,\,\leq \tfrac{\eta_k\beta(\tilde{L}_{k-1}-L_{k-1})}{2a_{k-1} \tau_{k-1}} \|z_{k-1}-x_{k-2}\|^2\\
  &\overset{\text{(iv)}}\leq \tfrac{\|z_{k-1}-x_{k-2}\|^2}{2a_{k-1} }\left(\tfrac{9\lambda n_{k-1}\beta^2(\tilde{L}_{k-1}-L_{k-1})^2}{a_{k-2}\tau_{k-1}^2}+\tfrac{\eta_k^2 a_{k-2}}{36\lambda n_{k-1}}\right)\\
  &\overset{\text{(v)}}\leq\left[\tfrac{9 n_{k-1}\beta^2\lambda(\tilde{L}_{k-1}-L_{k-1})^2}{a_{k-1} \tau_{k-1}^2a_{k-2}}+\tfrac{\eta_k^2a_{k-2}}{36\lambda a_{k-1} n_{k-1}}\right](\|z_{k-1}-x^*\|^2+\|x_{k-2}-x^*\|^2)\\
  &\overset{\text{(vi)}}{\leq}{ }\left(\tfrac{{9 n_{k-1}}\beta^2\lambda{(\tilde{L}_{k-1}-L_{k-1})^2}}{{a_{k-1}}\tau_{k-1}^2} +\tfrac{\eta_k^2 a_{k-2}}{{36\lambda}n_{k-1}}\right)\left(\tfrac{\|z_{k-1}-x^*\|^2}{a_{k-2}}+\tfrac{\|x^*-x_{k-2}\|^2}{a_{k-3}}\right),
   \end{aligned}
\end{equation}
where we recall that $\tilde{L}_{k-1}=\tfrac{1}{n_{k-1}}\textstyle\sum_{i=1}^{n_{k-1}} \tilde{L}_{k}(\hat{\xi}_{k-1,i})$; in (iv), we used Young's inequality with an arbitrary $\lambda>0$; in (v), we inserted $x^*$ and used the basic inequality $\|a+b\|^2\leq 2\|a\|^2+2\|b\|^2$; in (vi), we used the monotonicity of $a_k$. Furthermore, by $a_0\leq a_1$ and $\beta\in(0,1]$, it holds that
\begin{align}\label{prop-1-intermediate-2}
   \tfrac{\beta\eta_2\tau_1[\Psi(x_{0})-\Psi(x^*)]}{a_1(1+\gamma_2)\Gamma_2}+\tfrac{{\beta}\|z_{1}-y_{0}\|^2}{4a_1\Gamma_1(1+\gamma_1)}\leq  \tfrac{\beta\eta_2\tau_1[\Psi(x_{0})-\Psi(x^*)]}{a_0(1+\gamma_2)\Gamma_2}+\tfrac{\|z_{1}-y_{0}\|^2}{4a_0\Gamma_1(1+\gamma_1)}.
\end{align}
It remains to bound $\|y_1-x^*\|^2$ and $\|z_1-y_0\|^2$ in \eqref{prop-1-intermediate}. Observe that $y_0=y_1$ because $\beta_1=0$.
By the optimality condition of \eqref{eqn:gradient-step-stochastic} at $z_1$ and the convexity of $h$, it holds that
\begin{align}\label{eqn:initial-gap}
            &\tfrac{2\eta_1}{1+\gamma_1}\langle G_{1}, z_1-x^*\rangle+\tfrac{2\eta_1}{1+\gamma_1}[h(z_1) -h(x^*)]+\|z_1-y_{0}\|^2 +\|z_1-x^*\|^2\leq\|y_{0}-x^*\|^2.
         \end{align}
Noting that $x_0=y_0=z_0$, we have
         \begin{align*}
        &\,\,\,\|z_1-x_{0}\|^2 +\|z_1-x^*\|^2\notag\\
        &\overset{\eqref{eqn:initial-gap}}{\leq}
         \tfrac{2\eta_1}{1+\gamma_1}\left[ \langle G_{1}, x_0-z_1\rangle+\langle G_{1}, x^*-x_0\rangle+h(x^*)- h(x_0)+h(x_0) -  h(z_1)\right] +\|y_{0}-x^*\|^2\notag\\
          &\,\overset{\text{(vii)}}{\leq}\tfrac{2\eta_1}{1+\gamma_1}\left[\langle  {G_{1}+s_0}, x_0-z_1\rangle+\langle G_{1}, x^*-x_0\rangle+h(x^*)- h(x_0)\right] +\|y_{0}-x^*\|^2\notag\\
        &\,\,\,\,\leq \tfrac{2\eta_1}{1+\gamma_1}\left[\tfrac{\eta_1\|  {G_{1}+s_0}\|^2 }{1+\gamma_1}+\tfrac{(1+\gamma_1)\|x_0-z_1\|^2}{4\eta_1}+\langle {G}_1, x^*-x_0\rangle+h(x^*)-    h(x_0)\right] +\|y_{0}-x^*\|^2,
         \end{align*}
where in (vii), $ {s_0}\in\partial h(x_0).$ Therefore, we have
\begin{align}\label{prop-1-intermediate-3}
\tfrac{\|z_{1}-x_0\|^2}{4a_0\Gamma_1(1+\gamma_1)}
         &\leq\tfrac{1}{2a_0(1+\gamma_1)}\left[\tfrac{2\eta_1^2\|  {G_{1}+s_0}\|^2}{(1+\gamma_1)^2}+ \tfrac{2\eta_1\langle {G}_1, x^*-x_0\rangle}{1+\gamma_1}+\tfrac{2\eta_1[h(x^*)-    h(x_0)] }{1+\gamma_1}+\|y_{0}-x^*\|^2\right].
\end{align}
Substituting \eqref{eqn:noise-initial}, \eqref{eqn:prop-intermediate-3}, \eqref{prop-1-intermediate-2}, and \eqref{prop-1-intermediate-3} into \eqref{prop-1-intermediate} concludes the proof.
\end{proof}

\vgap

We next state two lemmas characterizing lower bounds on the stepsize in two regimes: $\gamma_k=0, \gamma_k=\tfrac{1}{k}$.
\begin{lemma}\label{lem:remark-stepsize-lowerbound}
Suppose   \(\eta_1>0\) and
$\eta_k$ satisfies \eqref{eqn:eta-cor1-const}.
    Then, for all $N\ge 2$, it holds that
\begin{align}\label{eqn:conclusion-stepsize}
    \eta_N\geq \tfrac{N}{32\hat{L}_{N-1}},\quad\text{where}\quad
    \hat{L}_{N-1}\coloneqq \max\left\{\tfrac{1}{32(1-\beta)\eta_1},\, \bar{L}_1,\, \bar{L}_2,\, \dots,\, \bar{L}_{N-1}\right\}.
\end{align}
\end{lemma}
\begin{proof}
    When $k=2,$ by the definition of $\hat{L}_1,$
  there holds
  \begin{align*}
    \tfrac{1}{ 16\hat{L}_1}=\min\left\{{2(1-\beta)\eta_1}, \tfrac{1}{16\bar{L}_1}\right\}.
  \end{align*}
  Therefore, it holds that $$  \eta_2=\min\left\{\tfrac{1}{16\bar{L}_1},2(1-\beta_2)\eta_1\right\}= \tfrac{2}{32\hat{L}_1}.$$
   Suppose $\eta_k\geq\tfrac{k}{32\hat{L}_{k-1}},$
 then for the $k+1$- th iteration, it holds that
 \begin{equation*}
     \begin{aligned}
         \eta_{k+1}&\,\,=\min\left\{{\tfrac{k}{16\bar{L}_{k}}}, {\tfrac{(k+1)\eta_{k}}{k}}\right\}\geq \min\left\{{\tfrac{k}{16\bar{L}_{k}}}, {}\tfrac{k+1}{32\hat{L}_{k-1}}\right\}\geq\tfrac{k+1}{32\hat{L}_k}.
     \end{aligned}
 \end{equation*}
\end{proof}

\begin{lemma}\label{lem:remark-stepsize-lowerbound-beta}
Suppose   \(\eta_1>0\) and
$\eta_k$ satisfies \eqref{eqn:stepsize-5}.
    Then, for all $N\ge 2$, it holds that
\begin{align}
    \eta_N\geq \tfrac{N}{32\hat{L}_{N-1}}\cdot\tfrac{15}{16},\quad\text{where}\quad
    \hat{L}_{N-1}\coloneqq \max\left\{\tfrac{1}{64(1-\beta)\eta_1},\, \bar{L}_1,\, \bar{L}_2,\, \dots,\, \bar{L}_{N-1}\right\}.
\end{align}
\end{lemma}
\begin{proof}
     When $k=2,$ by the definition of $\hat{L}_{1},$
  there holds
  \begin{align*}
   \tfrac{1}{32\hat{L}_{1}}\tfrac{4-\beta}{4}\leq\tfrac{1}{32\hat{L}_{1}}=\min\left\{2(1-\beta)\eta_{1},\tfrac{1}{32\bar{L}_{1}}\right\}.
  \end{align*}
  Therefore, there holds $  \eta_2=\min\left\{2(1-\beta)\eta_{1},\tfrac{1}{16\bar{L}_{1}}\right\}=\tfrac{1}{32\hat{L}_{1}}\tfrac{4-\beta}{4}.$
  Suppose for the $k$-th iteration, there holds $\eta_k\geq \tfrac{(k-1)^2}{32\hat{L}_{k-1}}\tfrac{k+2-\beta}{k^2},$ then for the $k+1$- th iteration, there holds
\begin{align*}
      \eta_{k+1}&=\min\left\{\tfrac{k}{16\bar{L}_{k}},\tfrac{k(k+3-\beta)}{(k+1)^2}\eta_{k}\right\}\geq \min\left\{\tfrac{k}{16\bar{L}_{k}},\tfrac{k(k+3-\beta)}{(k+1)^2}\tfrac{(k-1)^2}{32\hat{L}_{k-1}}\tfrac{k+2-\beta}{k^2}\right\}\notag\\
      &\overset{\text{(i)}}{\geq}  \min\left\{\tfrac{k}{16\bar{L}_{k}},\tfrac{k^2(k+3-\beta)}{(k+1)^2}\tfrac{1}{32\hat{L}_{k-1}}\right\}\overset{\text{(ii)}}{\geq}\tfrac{k^2(k+3-\beta)}{(k+1)^2}\tfrac{1}{32\hat{L}_{k}},
  \end{align*}
  where in (i), we used $k\geq 2$; in (ii), we used $2k(k+1)^2\geq k^2(k+3-\beta)$ for all $k$. Thus, we have $\eta_{k+1}\geq \tfrac{k}{32\hat{L}_k}\cdot\tfrac{15}{16}$.
\end{proof}

         \subsection{In-expectation convergence guarantees}\label{proof-exp}
              With Proposition \ref{thm:main-expectation} in hand, we are ready to prove \autoref{cor:main-fixed-T-const} and \autoref{cor:main}.             We first establish a few results under \autoref{assump:Bounded local variance}.

  \begin{lemma}\label{prop:exp}
Suppose the Assumptions in Proposition \ref{thm:main-expectation}. Furthermore, suppose \autoref{assump:Bounded local variance}, $\beta_k\equiv\beta,$ for all $k\geq 2.$ Then, for
  \begin{align}\label{eqn:max-var-n}
    a_{k-1}\coloneqq\max_{0\le i\le k-1}\left\{\tfrac{\tilde{c} v_i^2}{\beta}\right\},
\end{align}
and $a_{-1}=a_0,$
for all $k\geq 2,$ it holds that
\begin{align}
    &\mathbb{E}\left[\tfrac{\beta\eta_k\langle G_{k}-\nabla f(x_{k-1}), x^*-z_{k-1}\rangle}{a_{k-1}(1+\gamma_k)\Gamma_{k}}\right]=0,\label{eqn:exp-innerproduct}\\
&\mathbb{E}\left[\textstyle\sum_{k=2}^{N+1}\tfrac{1}{\Gamma_k(1+\gamma_k)}\cdot\tfrac{{9\lambda n_{k-1}}\beta^2{(\tilde{L}_{k-1}-L_{k-1})^2}}{{a_{k-1}}\tau_{k-1}^2} \left(\tfrac{\|z_{k-1}-x^*\|^2}{a_{k-2}}+\tfrac{\|x^*-x_{k-2}\|^2}{a_{k-3}}\right)\right]\notag\\
&\leq\mathbb{E}\left[\textstyle\sum_{k=2}^{N+1}\tfrac{{9}\beta^3\lambda{}}{{{\tilde{c} \Gamma_k(1+\gamma_k)}\tau_{k-1}^2}} \left(\tfrac{\|z_{k-1}-x^*\|^2}{a_{k-2}}+\tfrac{\|x^*-x_{k-2}\|^2}{a_{k-3}}\right)\right].\label{eqn:normalize}
\end{align}
\end{lemma}

\begin{proof}
                Observe that $n_{k-1}\in \mathcal{F}_{k-\frac{5}{3}},$
            $z_{k-1}, x_{k-1}, \beta_k\equiv\beta\in \mathcal{F}_{0}$ for all $k\geq 2,$ and $\Gamma_{k-1}$ is a function of $\beta_1,\dots, \beta_{k-1},$ thus, $\Gamma_{k-1}\in\mathcal{F}_{0}.$
            By the choice of $a_k,$
         it is random and satisfies $a_k\in\mathcal{F}_{k},$ $a_k\geq a_{k-1}$ for all $k\geq 1.$
Furthermore,
        $G_k\in\mathcal{F}_{k-\frac{2}{3}},$
   there holds
     \begin{align*}
\mathbb{E}\left[\tfrac{\beta\eta_k\langle G_{k}-\nabla f(x_{k-1}), z_{k-1}-x^*\rangle}{a_{k-1}(1+\gamma_k-\beta\gamma_k)\Gamma_{k-1}}\right]
&=\mathbb{E}\left[\tfrac{\beta\eta_k\langle \mathbb{E}_{{\xi}_{k}}\left[G_{k}\,|\,\mathcal{F}_{k-1}\right]-\nabla f(x_{k-1}), z_{k-1}-x^*\rangle}{a_{k-1}(1+\gamma_k-\beta\gamma_k)\Gamma_{k-1}}\right]\overset{\eqref{eqn: usual-unbiasedness}}{=}0.
     \end{align*}
Furthermore, observe that
\begin{equation}\label{eqn:condtional-var-application}
\begin{aligned}
\tfrac{{n_{k-1}}{\mathbb{E}_{\hat{\xi}_{k-1}}\left[{(\tilde{L}_{k-1}-L_{k-1})^2}\,\big|\,\mathcal{F}_{k-\frac{5}{3}}\right]}}{{a_{k-1}}}
    &\,\,\,\overset{\text{(i)}}{\leq}  \tfrac{\beta{n_{k-1}}{\mathbb{E}_{\hat{\xi}_{k-1}}\left[{(\tilde{L}_{k-1}-L_{k-1})^2}\,\big|\,\mathcal{F}_{k-\frac{5}{3}}\right]}}{{{  {\tilde{c} }v_{k-1} }}}\\
     &\,\,\,\overset{\text{(ii)}}{=}  \tfrac{\beta{}{}}{{{  {\tilde{c} }}}}\mathbb{E}_{\hat{\xi}_{k-1}}\left[\tfrac{{n_{k-1}(\tilde{L}_{k-1}-L_{k-1})^2}}{v_{k-1} }\,\big|\,\mathcal{F}_{k-\frac{5}{3}}\right]\\
    &\overset{\eqref{eqn:bar-L-k}}{=}\tfrac{\beta{}{}}{{{  {\tilde{c} }}}}\mathbb{E}_{\hat{\xi}_{k-1}}\left[\tfrac{|\sum_{i=1}^{n_{k-1}}[\ell_{k-1}(\hat{\xi}_{k-1, i})-L_{k-1}]|^{2}}{n_{k-1}v_{k-1} }\,\big|\,\mathcal{F}_{k-\frac{5}{3}}\right]\overset{\text{A}.\ref{assump:Bounded local variance}}{\leq}\tfrac{\beta{}{}}{{{  {\tilde{c} }}}},
\end{aligned}
\end{equation}
where in (i), we used $\eqref{eqn:max-var-n},$ and in
(ii), we used $n_{k-1}\in\mathcal{F}_{k-\frac{5}{3}},$ $v_{k-1}\in\mathcal{F}_{k-\frac{5}{3}}.$ Furthermore, notice that $z_{k-1}, x_{k-2}\in\mathcal{F}_{k-\frac{5}{3}},$ and $a_{k-1}$
satisfies \eqref{eqn:max-var-n}, thus $a_{k-1}\in\mathcal{F}_{k-\frac{2}{3}}$ by \eqref{eqn:local-var-2},
therefore, we have
\begin{align*}
&\,\,\mathbb{E}\left[\textstyle\sum_{k=2}^{N+1}\tfrac{1}{\Gamma_k(1+\gamma_k)}\cdot\tfrac{{9n_{k-1}}\beta^2{(\tilde{L}_{k-1}-L_{k-1})^2}}{{a_{k-1}}\tau_{k-1}^2} \left(\tfrac{\|z_{k-1}-x^*\|^2}{a_{k-2}}+\tfrac{\|x^*-x_{k-2}\|^2}{a_{k-3}}\right)\right]\\
&\overset{\text{(iii)}}=\mathbb{E}\left[\textstyle\sum_{k=2}^{N+1}\tfrac{{9n_{k-1}}\beta^2{\left(\tfrac{\|z_{k-1}-x^*\|^2}{a_{k-2}}+\tfrac{\|x^*-x_{k-2}\|^2}{a_{k-3}}\right)}}{{{\Gamma_k(1+\gamma_k)a_{k-1}}\tau_{k-1}^2}} \mathbb{E}_{\hat{\xi}_{k-1}}\left[(\tilde{L}_{k-1}-L_{k-1})^2\,\big|\,\mathcal{F}_{k-\frac{5}{3}}\right]\right]\\
&\overset{\eqref{eqn:condtional-var-application}}{\leq}\mathbb{E}\left[\textstyle\sum_{k=2}^{N+1}\tfrac{{9}\beta^3{\left(\tfrac{\|z_{k-1}-x^*\|^2}{a_{k-2}}+\tfrac{\|x^*-x_{k-2}\|^2}{a_{k-3}}\right)}}{{{\tilde{c} \Gamma_k(1+\gamma_k)}\tau_{k-1}^2}} \right],
\end{align*}
 {where in (iii), we used the tower property, since $n_{k-1}\in\mathcal{F}_{k-\frac{5}{3}}$ and hence
\begin{equation*}
   a_{k-1}=\max_{0\le i\le k-1}\left\{\tfrac{\tilde{c} v_i^2}{\beta}\right\}\overset{\eqref{eqn:definition-v_k}}{\in}\mathcal{F}_{k-\frac{5}{3}},
\end{equation*}
while $z_{k-1},x_{k-2}\in\mathcal{G}_{k-1}\overset{\eqref{eqn:G-k}}{\subseteq}\mathcal{F}_{k-\frac{5}{3}}$.}
\end{proof}
\subsubsection{Proof of \autoref{cor:main-fixed-T-const}}\label{Proof of {cor:main-fixed-T-const}}

We first bound the error term associated with $\|\Delta_k\|$ (cf.~\eqref{eqn:e-k}) in Proposition \ref{thm:main-expectation} under the setting of \autoref{cor:main-fixed-T-const}.

\begin{lemma}
    Suppose the Assumptions in \autoref{cor:main-fixed-T-const}, then it holds that
    \begin{align}\label{eqn:intermedia-noise-e_k'}
\mathbb{E}\left[\textstyle\sum_{k=2}^{N+1}\tfrac{8\eta_{k-1}^2\|\Delta_k\|}{a_{k-1}\beta}\right] \leq \tfrac{8{\beta} \tilde{D}^2}{{c }a_0}.
    \end{align}
\end{lemma}
\begin{proof}
    Notice that
\begin{equation}\label{eqn:fact1''}
     \begin{aligned}
         &\,\,\textstyle\sum_{k=2}^{N+1} \mathbb{E}\left[\tfrac{\eta_{k-1}^2}{a_{k-1}n_{k-1}^2}{\| \textstyle\sum_{i=1}^{n_{k-1}}G(x_{k-1}, \bar{\xi}_{k-1, i})-\nabla f(x_{k-1})\|^2}\right]\\
         &\,\,\overset{\text{(i)}}{\leq}\tfrac{1}{a_{0}}\textstyle\sum_{k=2}^{N+1} \mathbb{E}\left[\tfrac{\eta_{k-1}^2}{n_{k-1}^2}{\| \textstyle\sum_{i=1}^{n_{k-1}}G(x_{k-1}, \bar{\xi}_{k-1, i})-\nabla f(x_{k-1})\|^2}\right]\\
          &\overset{\eqref{eqn:n-k}}{\leq}\tfrac{1}{a_{0}}\textstyle\sum_{k=2}^{N+1} \mathbb{E}\left[\tfrac{ {\beta^2} \tilde{D}^2}{(N+1)c  n_{k-1}}\cdot\tfrac{\| \textstyle\sum_{i=1}^{n_{k-1}}G(x_{k-1}, \bar{\xi}_{k-1, i})-\nabla f(x_{k-1})\|^2}{\delta_{k-1}^2}\right]\\
     &\,\,\overset{\text{(ii)}}{=}\tfrac{\beta^2 \tilde{D}^2 }{c  a_0}\textstyle\sum_{k=2}^{N+1} \mathbb{E}\left[\tfrac{1}{ (N+1)n_{k-1}}\cdot\mathbb{E}_{\bar{\xi}_{k-1}}\left[\tfrac{\|\sum_{i=1}^{n_{k-1}}[G(x_{k-1},\bar{\xi}_{k-1,i})-\nabla f(x_{k-1})]\|^{2}}{  {\delta_{k-1}^{2}}}\,\big|\,  {\mathcal{F}_{k-\frac{5}{3}}}\right]\right]\\
      &\,\,\overset{\text{(iii)}}{=}\tfrac{\beta^2 \tilde{D}^2 }{ c  a_0}\textstyle\sum_{k=2}^{N+1} \mathbb{E}\left[\tfrac{1}{ (N+1)}\cdot\mathbb{E}_{\bar{\xi}_{k-1}}\left[\tfrac{\|[G(x_{k-1},\bar{\xi}_{k-1})-\nabla f(x_{k-1})]\|^{2}}{  {\delta_{k-1}^{2}}}\,\big|\,  {\mathcal{F}_{k-\frac{5}{3}}}\right]\right]\\
           &\,\,\overset{\text{(iv)}}{=}\tfrac{\beta^2 \tilde{D}^2 }{(N+1)c  a_0}\textstyle\sum_{k=2}^{N+1} \mathbb{E}\left[\tfrac{\|[G(x_{k-1},\bar{\xi}_{k-1})-\nabla f(x_{k-1})]\|^{2}}{  {\delta_{k-1}^{2}}}\right]\\
    &\,\,\overset{\text{(v)}}{=}\tfrac{\beta^2 \tilde{D}^2 }{(N+1) c  a_0}\textstyle\sum_{k=2}^{N+1} \mathbb{E}\left[{ }\mathbb{E}_{\bar{\xi}_{k-1}}\left[\tfrac{\|G(x_{k-1},\bar{\xi}_{k-1})-\nabla f(x_{k-1})\|^{2}}{  {\delta_{k-1}^{2}}}\,\big|\,  {\mathcal{G}_{k-1}}\right]\right]\\
    &\overset{\eqref{eqn:local-var-2}}{\leq}\tfrac{\beta^2 \tilde{D}^2 }{c  a_0},
     \end{aligned}
\end{equation}
where in (i), we used the monotonicity of $a_k$; in (ii) and (iv), we used the tower property together with $\delta_{k-1}\in\mathcal{G}_{k-1}\subseteq \mathcal{F}_{k-\frac{5}{3}}$ due to the construction of $\mathcal{G}_k$ in \eqref{eqn:G-k}; in (iii), we used the conditional i.i.d. property of $\bar{\xi}_{k-1,i}$ for all $i\in[n_{k-1}]$, together with $n_{k-1}\in \mathcal{F}_{k-\frac{5}{3}}$ and the conditional unbiasedness \autoref{a:unbiasness}, namely,
\begin{equation*}
\mathbb{E}_{\bar{\xi}_{k-1}}\!\left[G(x_{k-1},\bar{\xi}_{k-1})-\nabla f(x_{k-1})\mid \mathcal{F}_{k-\frac{5}{3}}\right]=0;
\end{equation*}
and in (v), we used the tower property through $\mathcal{G}_{k-1}$.

Similarly,
we have
\begin{align*}
      &\textstyle\sum_{k=1}^{N+1} \mathbb{E}\left[\tfrac{\eta_{k}^2{}}{a_{k-1}m_k^2{}}\| \textstyle\sum_{i=1}^{m_k}G(x_{k-1}, \xi_{k, i})-\nabla f(x_{k-1})\|^2\right]\overset{\eqref{eqn:local-var}, \eqref{eqn:batch-size-cor1-const}}{\leq} \tfrac{ {\beta^2}\tilde{D}^2}{{c }a_0}.
 \end{align*}
Moreover, by a similar argument as \eqref{eqn:fact1''},
it holds that
\begin{equation*}
    \begin{aligned}
          &\,\,\,\textstyle\sum_{k=2}^{N+1} \mathbb{E}\left[\tfrac{\eta_{k-1}^2}{a_{k-1}n_{k-1}^2}{}{}\| \textstyle\sum_{i=1}^{n_{k-1}}G(x_{k-2}, \bar{\xi}_{k-1, i})-\nabla f(x_{k-2})\|^2\right]\\
     &\,\,\leq \tfrac{\beta^2 \tilde{D}^2 }{ c  a_0(N+1)}\textstyle\sum_{k=2}^{N+1} \mathbb{E}\left[{ }\mathbb{E}_{\bar{\xi}_{k-1}}\left[\tfrac{\|G(x_{k-2},\bar{\xi}_{k-1})-\nabla f(x_{k-2})\|^{2}}{  {\sigma_{k-2}^{2}}}\,\big|\,  {\mathcal{F}_{k-2}}\right]\right]\overset{\eqref{eqn:local-var-2}}{\leq}\tfrac{\beta^2 \tilde{D}^2 }{ c  a_0}.
    \end{aligned}
\end{equation*}
\end{proof}

\begin{lemma}\citep[Lemma 4]{ghadimi2016accelerated} \label{lem:stochastic-true-gradient}
    For any $y_1, y_2\in\mathbb{R}^n, $ we have $\|\mathcal{G}(x, y_1, c)-\mathcal{G}(x, y_2, c)\|\leq\|y_1-y_2\|.$
\end{lemma}
Now we are ready to prove \autoref{cor:main-fixed-T-const}.
\begin{proof}[Proof of \autoref{cor:main-fixed-T-const}]
Given that $\gamma_k\equiv0,$ there holds $\Gamma_k\equiv 1.$
It is straightforward to see that under
the choice of $\beta_k$
as $\beta_1=0$ and $\beta_k\equiv \beta\in \left(0,\,\tfrac{1}{8}\right]$ for all $k\ge 2,$
then by the choice of $\zeta_k$ in \eqref{eqn:beta-k}, it holds that
\begin{equation*}
 \zeta_{k}= 1,\quad \forall\, k\geq 2.
\end{equation*}
Furthermore,
given that
$\tau_k=\tfrac{k}{2},$ it holds that
\begin{align*}
\tfrac{\eta_{k-1}(\tau_{k-2}+1)}{\tau_{k-1}}
= \tfrac{k\eta_{k-1}}{k-1}
\leq 2(1-\beta)^2\eta_{k-1},\quad\forall\, k\geq 3.
\end{align*}
Therefore, under \eqref{eqn:eta-cor1-const}, for all $k\geq 2,$ there holds
\begin{align*}
\eta_k\leq \tfrac{\zeta_k\tau_{k-1}}{8\bar{L}_{k-1}},\qquad
\eta_{k+1}\tau_{k}\leq \eta_{k}(\tau_{k-1}+1),\qquad
\eta_k\leq2(1-\beta_{k-1})(1-\beta_{k})^2\eta_{k-1},
\end{align*}
i.e., \eqref{eqn:stepsize-3} holds.
Therefore, \eqref{eqn:eta-cor1-const} is sufficient for \eqref{eqn:stepsize-3} to hold.
Therefore, \autoref{thm:main-expectation} with $\gamma_k\equiv 0, \Gamma_k\equiv 1$.

Taking
expectation on both sides of \eqref{eqn:final-iterate}, it holds that
  \begin{equation}\label{eqn:cor-1-1}
      \begin{aligned}
            &{\beta(\tau_{N}+1)}{}\mathbb{E}\left[\tfrac{  \eta_{N+1}}{a_{N+1}}[\Psi(x_{N})-\Psi(x^*)]\right]+\mathbb{E}\left[\tfrac{\|y_{N+1}-x^*\|^2}{2a_N}\right]+\tfrac{\beta^2}{4}\textstyle\sum_{k=3}^{N+1}\mathbb{E}\left[\tfrac{\eta_k^2 \|\mathcal{G}(y_{k-1},  G_k, \eta_k)\|^2}{a_{k-1}}\right]\\
     &\overset{\text{(i)}}{\leq}\tfrac{\beta}{2}\cdot{\mathbb{E}\left[\tfrac{\eta_2\left[\Psi(x_{0})-\Psi(x^*)\right]}{{a_0}}\right]}+\mathbb{E}\left\{\tfrac{\eta_1}{a_0}[\langle {G}_1, x^*-x_0\rangle+h(x^*)-    h(x_0)]\right\}+ {\mathbb{E}\left[\tfrac{\|x_{0}-x^*\|^2}{a_0}\right]+{\tfrac{2\eta_1^2\left\| {\nabla f(x_0)+s_0}\right\|^2}{a_0}}}\\
&\quad+\tfrac{2\eta_1^2\sigma^2_0}{a_0m_1}+\tfrac{32{\beta}\tilde{D}^2}{{c }a_0}+{\textstyle\sum_{k=2}^{N+1}{\mathbb{E}\left[\left(\tfrac{\eta_k^2a_{k-2}}{36  {\lambda} n_{k-1}}+\tfrac{9\beta^3  {\lambda} {}{}}{{{\tau_{k-1}^2  {\tilde{c} }}}}\right)\left(\tfrac{\|x^*-z_{k-1}\|^2}{a_{k-2}}+\tfrac{\|x^*-x_{k-2}\|^2}{a_{k-3}}\right)\right]}{{}}},
      \end{aligned}
  \end{equation}
where in (i), we substituted \eqref{eqn:intermedia-noise-e_k'},  \eqref{eqn:exp-innerproduct},  \eqref{eqn:normalize} into \eqref{eqn:final-iterate} and used $\tau_1=\tfrac{1}{2}.$ Utilizing \eqref{eqn:cor-1-1}, we are now ready to prove the iterates are bounded in expectation as follows.
\begin{align}\label{eqn:induction-goal}
  \min\limits_{x^*\in X^*} \mathbb{E}\left[\tfrac{\|y_{N+1}-x^*\|^2}{a_N}\right] \leq  \tfrac{ D^2_0}{a_0},\,  \min\limits_{x^*\in X^*}\mathbb{E}\left[\tfrac{\|z_{N}-x^*\|^2}{a_{N-1}}\right]\leq  \tfrac{4D^2_0}{\beta^2a_0},\,\min\limits_{x^*\in X^*}\mathbb{E}\left[\tfrac{\|x_{N}-x^*\|^2}{a_{N-1}}\right]\leq  \tfrac{4D^2_0}{\beta^2a_0}.
\end{align}
We prove  by induction.
It is immediate to see that
\begin{align*}
    &\min\limits_{x^*\in X^*}\mathbb{E}\left[\tfrac{\|y_{1}-x^*\|^2}{a_{0}}\right]\overset{\text{(ii)}}{=}\min\limits_{x^*\in X^*}\mathbb{E}\left[\tfrac{\|y_{0}-x^*\|^2}{a_{0}}\right]\leq \tfrac{D_0^2}{a_0},\,\,\min\limits_{x^*\in X^*}\mathbb{E}\left[\tfrac{\|z_{0}-x^*\|^2}{a_{-1}}\right]\leq \tfrac{ 4D_0^2}{\beta^2a_0},\quad\min\limits_{x^*\in X^*}\mathbb{E}\left[\tfrac{\|x_{0}-x^*\|^2}{a_{-1}}\right]\leq \tfrac{ 4D_0^2}{\beta^2a_0},
\end{align*}
due to the choice $a_{-1}=a_0,$
where in (ii), we used $\beta_1=0.$
Suppose this holds for iteration $N,$ i.e.,
\begin{align}\label{eqn:bound-N-y}
    \min\limits_{x^*\in X^*}\mathbb{E}\left[\tfrac{\|y_{N}-x^*\|^2}{a_{N-1}}\right]\leq \tfrac{D^2_0}{a_0},\,\, \min\limits_{x^*\in X^*}\mathbb{E}\left[\tfrac{\|z_{N-1}-x^*\|^2}{a_{N-2}}\right]\leq \tfrac{4D^2_0}{\beta^2a_0}, \,\, \min\limits_{x^*\in X^*}\mathbb{E}\left[\tfrac{\|x_{N-1}-x^*\|^2}{a_{N-2}}\right]\leq \tfrac{4D^2_0}{\beta^2a_0}.
\end{align}
Then for  {iteration} $N,$ it holds that
\begin{equation}\label{eqn:bound-N}
    \begin{aligned}
    \min\limits_{x^*\in X^*} \mathbb{E}\left[\tfrac{\|z_{N}-x^*\|^2}{a_{N-1}}\right]&\overset{\eqref{output-center}}{=}\min\limits_{x^*\in X^*}\mathbb{E}\left[\left\|\tfrac{y_{N}-x^*-(1-\beta)(y_{N-1}-x^*)}{a_{N-1}\beta}\right\|^2\right]\\
     &\,\,\leq 2\min\limits_{x^*\in X^*}\mathbb{E}\left[\left\|\tfrac{y_{N}-x^*}{a_{N-1}\beta}\right\|^2+\left\|\tfrac{(1-\beta)(y_{N-1}-x^*)}{a_{N-2}\beta}\right\|^2\right]\leq \tfrac{4 D_0^2}{\beta^2a_0},\\
\min\limits_{x^*\in X^*}\mathbb{E}\left[\tfrac{\|x_{N}-x^*\|^2}{a_{N-1}}\right]&\overset{\eqref{eqn:output-stochastic}}{\leq} \tfrac{1}{1+\tau_{N}}\min\limits_{x^*\in X^*}\mathbb{E}\left[\tfrac{\|z_{N}-x^*\|^2}{a_{N-1}}\right]+\tfrac{\tau_{N}}{1+\tau_{N}}\min\limits_{x^*\in X^*}\mathbb{E}\left[\tfrac{\|x_{N-1}-x^*\|^2}{a_{N-1}}\right]\\
    &\,\,\,\leq\tfrac{1}{1+\tau_{N}}\min\limits_{x^*\in X^*}\mathbb{E}\left[\tfrac{\|z_{N}-x^*\|^2}{a_{N-1}}\right]+\tfrac{\tau_{N}}{1+\tau_{N}}\min\limits_{x^*\in X^*}\mathbb{E}\left[\tfrac{\|x_{N-1}-x^*\|^2}{a_{N-2}}\right]\leq \tfrac{4 D^2_0}{\beta^2a_0},
    \end{aligned}
\end{equation}
where the inequalities follow from Jensen's inequality and the monotonicity of $a_k$.
Therefore, we just need to prove
$\min\limits_{x^*\in X^*}\mathbb{E}\left[\tfrac{\|y_{N+1}-x^*\|^2}{2a_N}\right]\leq \tfrac{D_0^2}{a_0}.$

For the first two terms in \eqref{eqn:cor-1-1},
observe that by \eqref{eqn:eta-cor1-const}, there holds $\eta_2\leq \tfrac{2\eta_1}{\beta},$ thus
\begin{equation*}
    \begin{aligned}
     \min\limits_{x^*\in X^*} \mathbb{E}\left[ \tfrac{\beta\eta_2}{2}\cdot\tfrac{\Psi(x_{0})-\Psi(x^*)}{a_0}\right]&\leq     \mathbb{E}\left[\tfrac{\eta_1}{a_0}(\Psi(x_{0})-\Psi(x^*))\right].
    \end{aligned}
\end{equation*}
Furthermore, notice that
\begin{equation}\label{eqn:function-initial-gap-exp-n}
    \begin{aligned}
       & \min\limits_{x^*\in X^*}\mathbb{E}\left[\tfrac{\eta_1}{a_0}[\langle {G}_1, x^*-x_0\rangle+h(x^*)-    h(x_0)]\right\}\\
        &\overset{\text{(iii)}}{\leq} \min\limits_{x^*\in X^*}\mathbb{E}\left[\tfrac{\eta_1\langle {G}_1-\nabla f(x_0), x^*-x_0\rangle}{a_0}+\tfrac{\eta_1[\Psi(x^*)-    \Psi(x_0)]}{a_0}\right]\overset{\text{(iv)}}{=} \mathbb{E}\left[\tfrac{\eta_1[\Psi(x^*)-    \Psi(x_0)]}{a_0}\right],
    \end{aligned}
\end{equation}
where in (iii), we used the convexity of $f;$ in (iv), we used the unbiasedness of $G_1$ and the fact that $\eta_1$ and $a_0$ are deterministic. Therefore, the first two terms cancel.
Furthermore, it holds that
\begin{align}\label{eqn:fact-0-n}
    \tfrac{2\eta_1^2\sigma^2_0}{a_0m_1}\overset{\eqref{eqn:batch-size-cor1-const}}{\leq} \tfrac{2\beta^2 \tilde{D}^2}{c  a_0(N+2)
 }.
\end{align}
We now bound the remaining two terms.
Under the choices for $a_{k-1}$ in \eqref{eqn:max-var-n},
recalled here for convenience,
\begin{align*}
    a_{k-1}\coloneqq\max_{0\le i\le k-1}\left\{\tfrac{\tilde{c} v_i^2}{\beta}\right\},
\end{align*}
 we then can rewrite the batch condition \eqref{eqn:n-k} as
\begin{equation}\label{eqn:alter-n}
    n_k
=
\max\left\{
1,\,
\tfrac{\tilde{c}(N+2)\eta_k^2 v_{k-1}^{\max}}{\beta^3},\,
\tfrac{(N+2)\eta_k^2}{\beta^2}\cdot \tfrac{c(\sigma_{k-1}^2+\delta_k^2)}{\tilde{D}^2}
\right\}=\max\left\{
1,\,
\tfrac{(N+2)\eta_k^2 a_{k-1}}{\beta^2},\,
\tfrac{(N+2)\eta_k^2}{\beta^2}\cdot \tfrac{c(\sigma_{k-1}^2+\delta_k^2)}{\tilde{D}^2}
\right\},
\end{equation}
and hence
\begin{equation}\label{eqn:in-exp-intermedia-5-n}
    \begin{aligned}
&\,\,\,\,\min\limits_{x^*\in X^*}\textstyle\sum_{k=2}^{N+1}\mathbb{E}\left[\tfrac{\eta_k^2 a_{k-2}}{36\lambda n_{k-1}}\left(\tfrac{\|x^*-z_{k-1}\|^2}{a_{k-2}}+\tfrac{\|x^*-x_{k-2}\|^2}{a_{k-3}}\right)\right]\\
&\overset{\eqref{eqn:eta-cor1-const}}{\leq} \min\limits_{x^*\in X^*}\mathbb{E}\left[\tfrac{4(1-\beta)^2 \eta_1^2a_{0}}{36\lambda n_{1}}\left(\tfrac{\|x^*-z_{1}\|^2}{a_{0}}+\tfrac{\|x^*-x_{0}\|^2}{a_{-1}}\right)\right]\\
&\quad\quad\,\,+\min\limits_{x^*\in X^*}\textstyle\sum_{k=3}^{N+1}\mathbb{E}\left[\tfrac{{\eta_{k-1}^2} a_{k-2}}{16 {\lambda}n_{k-1}}\cdot\left(\tfrac{\|x^*-z_{k-1}\|^2}{a_{k-2}}+\tfrac{\|x^*-x_{k-2}\|^2}{a_{k-3}}\right)\right]\\
&\overset{\eqref{eqn:alter-n}}{\leq}{} \min\limits_{x^*\in X^*}\mathbb{E}\left[\tfrac{4(1-\beta)^2\beta^2 }{36\lambda (N+2)}\left(\tfrac{\|x^*-z_{1}\|^2}{a_{0}}+\tfrac{\|x^*-x_{0}\|^2}{a_{-1}}\right)\right]\\
&\quad\quad\,\,+\min\limits_{x^*\in X^*}\textstyle\sum_{k=3}^{N+1}\mathbb{E}\left[\tfrac{{\beta^2} }{16\lambda (N+2)}\cdot\left(\tfrac{\|x^*-z_{k-1}\|^2}{a_{k-2}}+\tfrac{\|x^*-x_{k-2}\|^2}{a_{k-3}}\right)\right]\overset{\eqref{eqn:bound-N}}{\leq}\tfrac{{8D^2_0} }{9\lambda a_0}.
    \end{aligned}
\end{equation}
For the last term in \eqref{eqn:cor-1-1}, notice that
 \begin{equation}\label{eqn:exp-p-intermedia-4-exp-n}
     \begin{aligned}
     &\min\limits_{x^*\in X^*}\textstyle\sum_{k=2}^{N+1}{\mathbb{E}\left[\tfrac{9\beta^3\lambda{}{}}{{{\tau_{k-1}^2  {\tilde{c} }}}}\left(\tfrac{\|x^*-z_{k-1}\|^2}{a_{k-2}}+\tfrac{\|x^*-x_{k-2}\|^2}{a_{k-3}}\right)\right]}{{}}\overset{\text{(iv)}}{\leq}\min\limits_{x^*\in X^*}\textstyle\sum_{k=2}^{N+1}{\mathbb{E}\left[\tfrac{288\beta\lambda{}{}}{{{(k-1)^2  {\tilde{c} }}}}\tfrac{D^2_0}{a_0}\right]}{{}}\leq \tfrac{288{\beta}\lambda D^2_0}{ {\tilde{c} }a_0}  {\cdot\tfrac{1}{1/2}},
\end{aligned}
 \end{equation}
 where in (iv), we substituted $\tau_{k-1}=\tfrac{k-1}{2}$ and used the induction  {hypothesis} \eqref{eqn:bound-N}.

  {
By \autoref{lem:stochastic-true-gradient} we have
\begin{equation}\label{eqn:exp-p-intermedia-5-exp-n}
    \begin{aligned}
        \mathbb{E}[\|\mathcal{G}(y_{k},  G_{k+1}, \eta_{k+1})-\mathcal{G}(y_{k},  \nabla f(x_k), \eta_{k+1})\|^2]&\leq \mathbb{E}[\|G_{k+1}-\nabla f(x_k)\|^2]\\
    &=\mathbb{E}[\mathbb{E}[\|G_{k+1}-\nabla f(x_k)\|^2\,|\,\mathcal{F}_k]]\\
    &\overset{\text{(v)}}{\leq} \mathbb{E}\left[\tfrac{\sigma_k^2}{m_{k+1}}\right]\overset{\eqref{eqn:batch-size-cor1-const}}{\leq} \tfrac{\beta^2\tilde{D}^2}{c(N+2)\eta_{k+1}^2},
    \end{aligned}
\end{equation}
where in (v), we used
$\xi_{k+1, i}, i\in[m_{k+1}]$ are independent and identically distributed, and $m_{k+1}\in\mathcal{F}_k.$
}

Substituting \eqref{eqn:function-initial-gap-exp-n} -
\eqref{eqn:exp-p-intermedia-5-exp-n}
into \eqref{eqn:cor-1-1},  {and choosing $\lambda=4$ in \eqref{eqn:cor-1-1},}
we obtain
\begin{equation}\label{eqn:final-upper}
    \begin{aligned}
     &{\beta(\tau_{N}+1)}{}\mathbb{E}\left[\tfrac{  \eta_{N+1}}{a_{N+1}}[\Psi(x_{N})-\Psi(x^*)]\right]+\min\limits_{x^*\in X^*}\mathbb{E}\left[\tfrac{\|y_{N+1}-x^*\|^2}{2a_N}\right]+\tfrac{\beta^2}{8}\textstyle\sum_{k=3}^{N+1}\mathbb{E}\left[\tfrac{\eta_k^2 \|\mathcal{G}(y_{k-1},  \nabla f(x_{k-1}), \eta_k)\|^2}{a_{k-1}}\right]\\
    &\leq{\beta(\tau_{N}+1)}{}\mathbb{E}\left[\tfrac{  \eta_{N+1}}{a_{N+1}}[\Psi(x_{N})-\Psi(x^*)]\right]+\min\limits_{x^*\in X^*}\mathbb{E}\left[\tfrac{\|y_{N+1}-x^*\|^2}{2a_N}\right]\\
    &\quad+\tfrac{\beta^2}{4}\textstyle\sum_{k=3}^{N+1}\mathbb{E}\left[\tfrac{\eta_k^2 \|\mathcal{G}(y_{k-1},  G_k, \eta_k)\|^2}{a_{k-1}}\right]+\tfrac{\beta^2}{4}\textstyle\sum_{k=3}^{N+1}\mathbb{E}\left[\tfrac{\eta_k^2 \|\mathcal{G}(y_{k-1},  G_k, \eta_k)-\mathcal{G}(y_{k-1},  \nabla f(x_{k-1}), \eta_k)\|^2}{a_{k-1}}\right]\\
              &\leq \min\limits_{x^*\in X^*}{\mathbb{E}\left[\tfrac{\|x_{0}-x^*\|^2}{a_0}\right]+{\tfrac{2\eta_1^2\left\| {\nabla f(x_0)+s_0}\right\|^2}{a_0}}}+\tfrac{2\beta^2 \tilde{D}^2}{{c }a_0 (N+2)
 }+\tfrac{32{\beta}\tilde{D}^2}{{c }a_0}+\tfrac{2304\beta D_0^2}{ {\tilde{c} }a_0}+\tfrac{{8D_0^2} }{36}+\tfrac{\beta^4\tilde{D}^2}{4ca_0}\\
 &\overset{\text{(vi)}}{\leq} \tfrac{D_0^2}{2a_0}=\tfrac{D_0^2\beta}{2v_0\tilde{c}},
    \end{aligned}
\end{equation}
where in (vi), we used  {\eqref{eqn:chocie D} with $c \coloneqq 73$, $\tilde{c} \coloneqq 1728$, and $\beta\leq 1/8$, together with the fact that $D_0$ satisfies \eqref{eqn:chocie D}, which implies}
\begin{align*}
  \min\limits_{x^*\in X^*} \tfrac{\|x_{0}-x^*\|^2+\tilde{D}^2}{a_0}
  +\tfrac{2\eta_1^2\| {\nabla f(x_0)+s_0}\|^2}{a_0}
  \leq \tfrac{D_0^2}{18a_0}.
\end{align*}
 On the other hand,
by the lower bound on the stepsize from \autoref{lem:remark-stepsize-lowerbound},
  we have
  \begin{equation}\label{eqn:lowerbound-type-1}
      \begin{aligned}
          &\mathbb{E}\left[\tfrac{  \eta_{N+1}\beta_{N+1}(\tau_{N}+1)[\Psi(x_{N})-\Psi(x^*)]}{a_{N+1}}\right]+\min\limits_{x^*\in X^*}\mathbb{E}\left[\tfrac{\|y_{N+1}-x^*\|^2}{2a_N}\right]+\tfrac{\beta^2}{8}\textstyle\sum_{k=3}^{N+1}\mathbb{E}\left[\tfrac{\eta_k^2 \|\mathcal{G}(y_{k-1},  \nabla f(x_{k-1}), \eta_k)\|^2}{a_{k-1}}\right]\\
            &\overset{\eqref{eqn:conclusion-stepsize}}{\geq} \tfrac{\beta^2 N(N+2)}{64\mathcal{L}v_{\max}\tilde{c}}\mathbb{E}[\Psi({x}_N)-\Psi(x^*)]+\min\limits_{x^*\in X^*}\mathbb{E}\left[\tfrac{\beta\|y_{N+1}-x^*\|^2}{2v_{\max}\tilde{c}}\right]+\tfrac{\beta^3N^3\min\limits_{2\leq k\leq N}\mathbb{E}[\|\mathcal{G}(y_{k},  \nabla f(x_k), \eta_{k+1})\|^2]}{8192\mathcal{L}^2v_{\max}\tilde{c}}.
      \end{aligned}
  \end{equation}
The deterministic case follows directly from the definition of $a_{k-1}$ in \eqref{eqn:max-var-n}. Specifically, in the deterministic setting, we have
\begin{align*}
    a_{k-1}\coloneqq\max_{0\le i\le k-1}\left\{\tfrac{\tilde{c} v_i^2}{\beta}\right\}=\tfrac{\tilde{c} v_0^2}{\beta}.
\end{align*}
Due to the choice of $v_0^2>0$, and it is  deterministic. Therefore, it holds that
 \begin{equation}\label{eqn:lowerbound-type-2}
      \begin{aligned}
          &\mathbb{E}\left[\tfrac{  \eta_{N+1}\beta_{N+1}(\tau_{N}+1)[\Psi(x_{N})-\Psi(x^*)]}{a_{N+1}}\right]+\min\limits_{x^*\in X^*}\mathbb{E}\left[\tfrac{\|y_{N+1}-x^*\|^2}{2a_N}\right]+\tfrac{\beta^2}{8}\textstyle\sum_{k=3}^{N+1}\mathbb{E}\left[\tfrac{\eta_k^2 \|\mathcal{G}(y_{k-1},  \nabla f(x_{k-1}), \eta_k)\|^2}{a_{k-1}}\right]\\
            &\overset{\eqref{eqn:conclusion-stepsize}}{\geq} \tfrac{\beta^2 N(N+2)}{64\mathcal{L}v_{0}\tilde{c}}\mathbb{E}[\Psi({x}_N)-\Psi(x^*)]+\min\limits_{x^*\in X^*}\tfrac{\beta\mathbb{E}\left[\|y_{N+1}-x^*\|^2\right]}{2v_{0}\tilde{c}}+\tfrac{\beta^3N^3\min\limits_{2\leq k\leq N}\mathbb{E}[\|\mathcal{G}(y_{k},  \nabla f(x_k), \eta_{k+1})\|^2]}{8192\mathcal{L}^2v_{0}\tilde{c}}.
      \end{aligned}
  \end{equation}
  Combining \eqref{eqn:lowerbound-type-1} and \eqref{eqn:lowerbound-type-2}, and noting that $\max\left\{\tfrac{v_{\max}}{v_0},1\right\}=1$, concludes the proof.
\end{proof}

\begin{proof}[Proof of \autoref{remark:weak-guarantee}]
Instead of \eqref{eqn:lowerbound-type-1}, we consider
\begin{equation}
    \begin{aligned}
    \mathbb{E}^2\left[\sqrt{\Psi(x_{N})-\Psi(x^*)}\right]
    &\,\,\,\overset{\text{(i)}}{\leq}\mathbb{E}\left[\tfrac{ (N+1)\beta(\tau_{N}+1)[\Psi(x_{N})-\Psi(x^*)]}{32\hat{L}_Na_{N+1}}\right]\cdot \mathbb{E}\left[\tfrac{32\hat{L}_Na_{N+1}}{  (N+1)\beta (\tau_{N}+1)}\right]\\
    &\overset{\eqref{eqn:conclusion-stepsize}}{\leq}  \mathbb{E}\left[\tfrac{  \eta_{N+1}\beta(\tau_{N}+1)[\Psi(x_{N})-\Psi(x^*)]}{a_{N+1}}\right]\cdot \mathbb{E}\left[\tfrac{32\hat{L}_Na_{N+1}}{  (N+1)\beta (\tau_{N}+1)}\right]\\
    &\overset{\eqref{eqn:final-upper}}{\leq}\tfrac{32D_0^2\mathbb{E}\left[\hat{L}_Nv_{N+1}^{\max}\right]}{v_0\beta N^2},
    \end{aligned}
\end{equation}
where in  (i), we used the Cauchy--Schwarz inequality. Similar to \eqref{eqn:lowerbound-type-2}, we can derive the deterministic-case bound. Combining the two, we obtain
\begin{align*}
\mathbb{E}^2\left[\sqrt{\Psi({x}_N)-\Psi(x^*)}\right]
&\le
\tfrac{32D_0^2}{\beta N^2} \max\left\{\tfrac{\mathbb{E}\left[\hat{L}_Nv_{N+1}^{\max}\right]}{v_0}, \mathbb{E}\left[\hat{L}_N\right]\right\}.
\end{align*}
Similarly, we have
\begin{equation*}
    \begin{aligned}
        \min_{x^*\in X^*}\mathbb{E}^2\bigl[\|y_{N+1}-x^*\|\bigr]
        &\leq D_0^2
        \max\left\{\tfrac{\mathbb{E}\left[v^{\max}_{N}\right]}{v_0},1\right\},\\
        \min\limits_{2\leq k\leq N}\mathbb{E}^2\bigl[\|\mathcal{G}(y_k,G_{k+1},\eta_{k+1})\|\bigr]
        &\le
        \tfrac{4096D_0^2}{\beta^2N^3}
        \max\left\{\tfrac{\mathbb{E}\left[\hat{L}^2_Nv_{N+1}^{\max}\right]}{v_0}, \mathbb{E}\left[\hat{L}^2_N\right]\right\}.
    \end{aligned}
\end{equation*}

\end{proof}

\subsubsection{Proof of \autoref{cor:main}}\label{Proof of {cor:main}}
 We first bound the error term associated with $\|\Delta_k\|$ (cf.~\eqref{eqn:e-k}) in Proposition \ref{thm:main-expectation} under the setting of \autoref{cor:main}.
\begin{lemma}
    Suppose the Assumptions in \autoref{cor:main}, then it holds that
    \begin{align}\label{eqn:intermedia-noise-e_k}
\mathbb{E}\left[\textstyle\sum_{k=3}^{N+1}\tfrac{8\eta_{k-1}^2\|\Delta_k\|}{a_{k-1}\beta\Gamma_{k-1}}\right] \leq \tfrac{8{\beta}(N+1)^\beta \tilde{D}^2}{2^\beta{c }a_0}.
    \end{align}
\end{lemma}
\begin{proof}
By the choice of $\gamma_k=\tfrac{1}{k}$ for all $k\geq 1$ and $\beta_k\equiv\beta$ for all $k\geq 2$, we have
\begin{align}\label{eqn:Gamma-lower}
\Gamma_k\overset{\eqref{ean:Gamma}}{=}\Gamma_{k-1}\left(1-\tfrac{\beta\gamma_k}{1+\gamma_k}\right)=\Gamma_1\prod\limits_{t=2}^{k}\tfrac{t+1-\beta}{t+1}=\prod\limits_{s=3}^{k+1}\tfrac{s-\beta}{s}\geq \prod\limits_{s=3}^{k+1}\left(1-\tfrac{1}{s}\right)^{\beta}=\left(\tfrac{2}{k+1}\right)^{\beta}.
\end{align}
    Notice that
\begin{equation}\label{eqn:fact1'}
     \begin{aligned}
         &\,\,\textstyle\sum_{k=2}^{K+1} \mathbb{E}\left[\tfrac{\eta_{k-1}^2}{a_{k-1}n_{k-1}^2\Gamma_{k-1}}{\| \textstyle\sum_{i=1}^{n_{k-1}}G(x_{k-1}, \bar{\xi}_{k-1, i})-\nabla f(x_{k-1})\|^2}\right]\\
         &\,\,\overset{\text{(i)}}{\leq}\tfrac{1}{a_{0}}\textstyle\sum_{k=2}^{K+1} \mathbb{E}\left[\tfrac{\eta_{k-1}^2}{n_{k-1}^2\Gamma_{k-1}}{\| \textstyle\sum_{i=1}^{n_{k-1}}G(x_{k-1}, \bar{\xi}_{k-1, i})-\nabla f(x_{k-1})\|^2}\right]\\
          &\overset{\eqref{eqn:mini-batch-n-N-free}}{\leq}\tfrac{1}{a_{0}}\textstyle\sum_{k=2}^{K+1} \mathbb{E}\left[\tfrac{\beta^2 \tilde{D}^2}{(k+1)c  n_{k-1}\Gamma_{k-1}}\cdot\tfrac{\| \textstyle\sum_{i=1}^{n_{k-1}}G(x_{k-1}, \bar{\xi}_{k-1, i})-\nabla f(x_{k-1})\|^2}{\delta_{k-1}^2}\right]\\
    &\overset{\eqref{eqn:Gamma-lower}}{\leq}\tfrac{\beta^2 \tilde{D}^2 }{2^\beta c  a_0}\textstyle\sum_{k=2}^{K+1} \mathbb{E}\left[\tfrac{k^{\beta -1}}{ n_{k-1}}\cdot\tfrac{\| \textstyle\sum_{i=1}^{n_{k-1}}G(x_{k-1}, \bar{\xi}_{k-1, i})-\nabla f(x_{k-1})\|^2}{\delta_{k-1}^2}\right]\\
     &\,\,\overset{\text{(ii)}}{=}\tfrac{\beta^2 \tilde{D}^2 }{2^\beta c  a_0}\textstyle\sum_{k=2}^{K+1} \mathbb{E}\left[\tfrac{k^{\beta -1}}{ n_{k-1}}\cdot\mathbb{E}_{\bar{\xi}_{k-1}}\left[\tfrac{\|\sum_{i=1}^{n_{k-1}}[G(x_{k-1},\bar{\xi}_{k-1,i})-\nabla f(x_{k-1})]\|^{2}}{  {\delta_{k-1}^{2}}}\,\big|\,  {\mathcal{F}_{k-\frac{5}{3}}}\right]\right]\\
      &\,\,\overset{\text{(iii)}}{=}\tfrac{\beta^2 \tilde{D}^2 }{2^\beta c  a_0}\textstyle\sum_{k=2}^{K+1} \mathbb{E}\left[\tfrac{k^{\beta -1}}{ n_{k-1}}\cdot\mathbb{E}_{\bar{\xi}_{k-1}}\left[\tfrac{n_{k-1}\|[G(x_{k-1},\bar{\xi}_{k-1})-\nabla f(x_{k-1})]\|^{2}}{  {\delta_{k-1}^{2}}}\,\big|\,  {\mathcal{F}_{k-\frac{5}{3}}}\right]\right]\\
           &\,\,\overset{\text{({iv})}}{=}\tfrac{\beta^2 \tilde{D}^2 }{2^\beta c  a_0}\textstyle\sum_{k=2}^{K+1} \mathbb{E}\left[{k^{\beta -1}}{ }\tfrac{\|[G(x_{k-1},\bar{\xi}_{k-1})-\nabla f(x_{k-1})]\|^{2}}{  {\delta_{k-1}^{2}}}\right]\\
    &\,\,\overset{\text{({v})}}{=}\tfrac{\beta^2 \tilde{D}^2 }{2^\beta c  a_0}\textstyle\sum_{k=2}^{K+1} \mathbb{E}\left[{k^{\beta -1}}{ }\cdot\mathbb{E}_{\bar{\xi}_{k-1}}\left[\tfrac{\|G(x_{k-1},\bar{\xi}_{k-1})-\nabla f(x_{k-1})\|^{2}}{  {\delta_{k-1}^{2}}}\,\big|\,  {\mathcal{G}_{k-1}}\right]\right]\\
    &\overset{\eqref{eqn:local-var-2}}{\leq}\tfrac{\beta^2 \tilde{D}^2 }{2^\beta c  a_0}\textstyle\sum_{k=2}^{K+1} \mathbb{E}\left[{k^{\beta -1}}{ }\right]{\leq}\tfrac{ {\beta^2}(K+1)^\beta \tilde{D}^2}{2^\beta{c }a_0},
     \end{aligned}
\end{equation}
where in (i), we used the monotonicity of $a_k$; in (ii)  {and (iv)}, we used the tower property; in (iii), we used the conditional i.i.d. property of $\bar{\xi}_{k-1,i}$ for all $i\in[n_{k-1}]$, together with $\delta_{k-1}\in\mathcal{G}_{k-1}\subseteq \mathcal{F}_{k-\frac{5}{3}}$ due to the construction of $\mathcal{G}_k$ in \eqref{eqn:G-k}, $n_{k-1} \in \mathcal{F}_{k-\frac{5}{3}}$, and the conditional unbiasedness \autoref{a:unbiasness}, namely,
\begin{equation*}
\mathbb{E}_{\bar{\xi}_{k-1}}\!\left[G(x_{k-1},\bar{\xi}_{k-1})-\nabla f(x_{k-1})\mid \mathcal{F}_{k-\frac{5}{3}}\right]=0;
\end{equation*}
and in (v), we used the tower property through $\mathcal{G}_{k-1}$.

Similarly,
we have
\begin{align*}
      &\textstyle\sum_{k=1}^{K+1} \mathbb{E}\left[\tfrac{\eta_{k}^2{}}{a_{k-1}m_k^2{\Gamma_{k-1}}}\| \textstyle\sum_{i=1}^{m_k}G(x_{k-1}, \xi_{k, i})-\nabla f(x_{k-1})\|^2\right]\overset{\eqref{eqn:mini-batch-N-free}}{\leq} \tfrac{ {\beta}(K+2)^\beta \tilde{D}^2}{2^\beta{c }a_0}.
 \end{align*}
Moreover, by a similar argument as \eqref{eqn:fact1'},
it holds that
\begin{equation*}
    \begin{aligned}
          &\,\,\,\,\textstyle\sum_{k=2}^{K+1} \mathbb{E}\left[\tfrac{\eta_{k-1}^2}{a_{k-1}n_{k-1}^2\Gamma_{k-1}}{}{}\| \textstyle\sum_{i=1}^{n_{k-1}}G(x_{k-2}, \bar{\xi}_{k-1, i})-\nabla f(x_{k-2})\|^2\right]\\
     &\,\,\,\,\leq \tfrac{\beta^2 \tilde{D}^2 }{2^\beta c  a_0}\textstyle\sum_{k=2}^{K+1} \mathbb{E}\left[{k^{\beta -1}}{ }\cdot\mathbb{E}_{\bar{\xi}_{k-1}}\left[\tfrac{\|G(x_{k-2},\bar{\xi}_{k-1})-\nabla f(x_{k-2})\|^{2}}{  {\sigma_{k-2}^{2}}}\,\big|\,  {\mathcal{F}_{k-2}}\right]\right]\\
    &\overset{\eqref{eqn:local-var-2}}{\leq}\tfrac{\beta^2 \tilde{D}^2 }{2^\beta c  a_0}\textstyle\sum_{k=2}^{K+1} \mathbb{E}\left[{k^{\beta -1}}{ }\right]{\leq}\tfrac{ {\beta}(K+1)^\beta \tilde{D}^2}{2^\beta{c }a_0}\leq \tfrac{ {\beta}(K+1)^\beta \tilde{D}^2}{2^\beta{c }a_0}.
    \end{aligned}
\end{equation*}
\end{proof}

\begin{proof}[Proof of \autoref{cor:main}]
By the choice of $\gamma_k=\tfrac{1}{k}$ for all $k\geq 1$ and $\beta_k\equiv\beta$ for all $k\geq 2$, we have \eqref{eqn:Gamma-lower},
and therefore
\begin{align}\label{eqn:Gamma-lower-2}
\textstyle\sum_{k=2}^{N+1}\tfrac{\beta\gamma_k}{\Gamma_k (1+\gamma_k)}\leq \textstyle\sum_{k=2}^{N+1}\left(\tfrac{k+1}{2}\right)^{\beta}\tfrac{\beta}{k+1}\leq {\left(\tfrac{N+2}{2}\right)^\beta}.
\end{align}
Next, by the choice of $\gamma_k$ and $\beta_k$, it holds that
\begin{equation}\label{eqn:verifiable-3}
\begin{aligned}
\zeta_{k}\overset{\eqref{eqn:beta-k}}{=} \tfrac{(k-1)(k+1-\beta)}{k^2},\quad\forall \quad k\geq 2.
\end{aligned}
\end{equation}
Since
$\tau_k=\tfrac{k+2-\beta}{2}\geq \tfrac{k+1}{2}$ for all $k\geq 1$, and $0<\beta \leq\tfrac{1}{4}$, for all $k\geq 3$ we have
\begin{equation}\label{eqn:verifiable-2}
\begin{aligned}
    \tfrac{\Gamma_{k}(1+\gamma_k)}{\Gamma_{k-1}(1+\gamma_{k-1})}\cdot\tfrac{\tau_{k-2}+1}{\tau_{k-1}}&=\tfrac{(k-1)(k+2-\beta)}{k^2}\leq \tfrac{(k-1)(k+2)}{k^2}\leq \tfrac{2(k-1)}{k}=\tfrac{2\gamma_k}{\gamma_{k-1}},\quad \\
    \tfrac{(k-1)(k+2)}{k^2}&\leq 2(1-\beta)^2,\quad
 \tfrac{k-1}{16} \leq \tfrac{(k-1)(k+1-\beta)^2}{16k^2}\overset{\eqref{eqn:verifiable-3}}{=}\tfrac{\zeta_k\tau_{k-1}}{8},
\end{aligned}
\end{equation}
and when $k=2$,
\begin{align*}
    \tfrac{2(1-\beta)}{ {(3-\beta)}}\eta_{1}\leq 2(1-\beta)\eta_1, \quad \eta_2\leq\eta_1=\tfrac{2\gamma_2\eta_1}{\gamma_1}, \quad \eta_2\leq \tfrac{1}{16\bar{L}_1}\leq\tfrac{1}{16\bar{L}_1}\cdot\tfrac{(3-\beta)^2}{4}=\tfrac{\zeta_2\tau_1}{8\bar{L}_1}.
\end{align*}
Therefore, \eqref{eqn:stepsize-5} is sufficient for \eqref{eqn:stepsize-3} to hold.
Therefore, Proposition \ref{thm:main-expectation}  {holds} with $\gamma_k=1/k$. Taking expectation on both sides of \eqref{eqn:final-iterate}, it holds that
\begin{equation}\label{eqn:lem6-1/k}
     \begin{aligned}
&\tfrac{ \beta(\tau_{N}+1)}{(1+\gamma_{N+1})\Gamma_{N+1}}\mathbb{E}\left[\tfrac{ \eta_{N+1}}{a_{N+1}}(\Psi(x_{N})-\Psi(x^*))\right]+\tfrac{1}{2\Gamma_{N+1}}\min\limits_{x^*\in X^*}\mathbb{E}\left[\tfrac{\|y_{N+1}-x^*\|^2}{a_N}\right]\\
              &\overset{\text{(i)}}{\leq}\tfrac{(3-\beta)\beta}{3\Gamma_2}\cdot{\mathbb{E}\left[\tfrac{\eta_2\left[\Psi(x_{0})-\Psi(x^*)\right]}{{a_0}}\right]}+\min\limits_{x^*\in X^*}\mathbb{E}\left\{\tfrac{\eta_1}{4a_0}[\langle {G}_1, x^*-x_0\rangle+h(x^*)-    h(x_0)]\right\}\\
                & \quad+\tfrac{\eta_1^2\sigma^2_0}{4a_0m_1}+ \min\limits_{x^*\in X^*}{\mathbb{E}\left[\tfrac{\|x_{0}-x^*\|^2}{2a_0}\right]\left[2+\left(\tfrac{N+2}{2}\right)^\beta\right]+{\tfrac{\eta_1^2\left\| {\nabla f(x_0)+s_0}\right\|^2}{4a_0}}}+\tfrac{32{\beta}(N+1)^\beta \tilde{D}^2}{2^\beta{c }a_0}\\
&\quad+{\textstyle\sum_{k=2}^{N+1}\tfrac{(k+1)^\beta }{2^{\beta}}{\min\limits_{x^*\in X^*}\mathbb{E}\left[\left(\tfrac{9\beta^3  {\lambda}{}{}}{{{\tau_{k-1}^2  {\tilde{c} }}}}+\tfrac{\eta_k^2a_{k-2}}{36  {\lambda}n_{k-1}}\right)\left(\tfrac{\|x^*-z_{k-1}\|^2}{a_{k-2}}+\tfrac{\|x^*-x_{k-2}\|^2}{a_{k-3}}\right)\right]}{{}}},
      \end{aligned}
\end{equation}
where in (i), we substituted \eqref{eqn:intermedia-noise-e_k},  \eqref{eqn:exp-innerproduct},  \eqref{eqn:normalize}, \eqref{eqn:Gamma-lower}, \eqref{eqn:Gamma-lower-2} into \eqref{eqn:final-iterate} and used $\tau_1=\tfrac{3-\beta}{2},$ $\gamma_1=\tfrac{3}{2}.$

Utilizing \eqref{eqn:lem6-1/k}, we are now ready to prove the iterates are bounded in expectation. Similar to the \eqref{eqn:induction-goal}, \eqref{eqn:bound-N-y} and \eqref{eqn:bound-N}, we just need to prove
 $\min\limits_{x^*\in X^*}\mathbb{E}\left[\tfrac{\|y_{N+1}-x^*\|^2}{2a_N}\right]\leq \tfrac{D_0^2}{a_0}.$

For the first two terms in \eqref{eqn:lem6-1/k},
observe that
\begin{equation}\label{eqn:function-initial-gap-exp}
    \begin{aligned}
      \mathbb{E}\left[ \tfrac{(3-\beta)\beta\eta_2}{3\Gamma_2}\cdot\tfrac{\Psi(x_{0})-\Psi(x^*)}{a_0}\right]&\overset{\eqref{eqn:stepsize-5},\eqref{eqn:Gamma-lower}}{\leq} \tfrac{3^{\beta-1}\beta(1-\beta)\eta_{1}}{2^{\beta-1}}{}\cdot   \mathbb{E}\left[\tfrac{\Psi(x_{0})-\Psi(x^*)}{a_0}\right]\leq    \mathbb{E}\left[\tfrac{\eta_1}{4a_0}\right][\Psi(x_{0})-\Psi(x^*)].
    \end{aligned}
\end{equation}
Furthermore, notice that
\begin{equation}\label{eqn:function-initial-gap-exp'}
    \begin{aligned}
       & \min\limits_{x^*\in X^*}\mathbb{E}\left[\tfrac{\eta_1}{4a_1}[\langle {G}_1, x^*-x_0\rangle+h(x^*)-    h(x_0)]\right\}\\
        &\overset{\text{(ii)}}{\leq} {\min\limits_{x^*\in X^*}\mathbb{E}\left[\tfrac{\eta_1\langle {G}_1-\nabla f(x_0), x^*-x_0\rangle}{4a_0}+\tfrac{\eta_1[\Psi(x^*)-    \Psi(x_0)]}{4a_0}\right]\overset{\text{(iii)}}{=} \mathbb{E}\left[\tfrac{\eta_1}{4a_0}\right][\Psi(x^*)-    \Psi(x_0)]},
    \end{aligned}
\end{equation}
where in (ii), we used the convexity of $f,$ and in (iii), we used the
unbiased estimator property \autoref{a:unbiasness}. Therefore, the first two terms cancel each other out.
Furthermore, it holds that
\begin{align}\label{eqn:fact-0}
\tfrac{\eta_1^2\sigma^2_0}{4a_0m_1}\overset{\eqref{eqn:mini-batch-N-free}}{\leq} \tfrac{\beta^2 \tilde{D}^2}{12c  a_0
 }.
\end{align}
We now bound the remaining two terms.
Under the choices for $a_{k-1}$ in \eqref{eqn:max-var-n},
recall here for convenience,
\begin{align*}
    a_{k-1}\coloneqq\max_{0\le i\le k-1}\left\{\tfrac{\tilde{c} v_i^2}{\beta}\right\},
\end{align*}
then, we can rewrite the batch condition \eqref{eqn:mini-batch-n-N-free} as
\begin{equation}\label{eqn:alter-n-1}
    n_k
=
\max\left\{
1,\,
\tfrac{\tilde{c}(k+2)\eta_k^2 v_{k-1}^{\max}}{\beta^4},\,
\tfrac{(k+2)\eta_k^2}{\beta^2}\cdot \tfrac{c(\sigma_{k-1}^2+\delta_k^2)}{\tilde{D}^2}
\right\}=\max\left\{
1,\,
\tfrac{(k+2)\eta_k^2 a_{k-1}}{\beta^3},\,
\tfrac{(k+2)\eta_k^2}{\beta^2}\cdot \tfrac{c(\sigma_{k-1}^2+\delta_k^2)}{\tilde{D}^2}
\right\},
\end{equation}
and hence,
\begin{equation}\label{eqn:in-exp-intermedia-5}
    \begin{aligned}
&\,\,\,\min\limits_{x^*\in X^*}\textstyle\sum_{k=2}^{N+1}\tfrac{(k+1)^{\beta}}{2^{\beta}}\mathbb{E}\left[\tfrac{\eta_k^2 a_{k-2}}{36n_{k-1}}\left(\tfrac{\|x^*-z_{k-1}\|^2}{a_{k-2}}+\tfrac{\|x^*-x_{k-2}\|^2}{a_{k-3}}\right)\right]\\
&\overset{\eqref{eqn:stepsize-5}}{\leq} \textstyle\min\limits_{x^*\in X^*}\tfrac{3^{\beta}\cdot 4(1-\beta)^2}{2^{\beta}}{}\mathbb{E}\left[\tfrac{{\eta_{1}^2} a_{0}}{36n_{1}}\cdot\left(\tfrac{\|x^*-z_{1}\|^2}{a_{0}}+\tfrac{\|x^*-x_{0}\|^2}{a_{-1}}\right)\right]\\
&\quad+\textstyle\min\limits_{x^*\in X^*}\sum_{k=3}^{N+1}\tfrac{(k+1)^{\beta}}{2^{\beta}}\cdot\tfrac{16}{9} \mathbb{E}\left[\tfrac{{\eta_{k-1}^2} a_{k-2}}{36n_{k-1}}\left(\tfrac{\|x^*-z_{k-1}\|^2}{a_{k-2}}+\tfrac{\|x^*-x_{k-2}\|^2}{a_{k-3}}\right)\right]\\
&\overset{\eqref{eqn:alter-n-1}}{\leq } \textstyle\min\limits_{x^*\in X^*}\tfrac{3^{\beta}\cdot 4(1-\beta)^2}{2^{\beta}}{}\mathbb{E}\left[\tfrac{{\eta_{1}^2} a_{0}\beta^3}{36\times 3\eta_1^2}\cdot\left(\tfrac{\|x^*-z_{1}\|^2}{a_{0}}+\tfrac{\|x^*-x_{0}\|^2}{a_{-1}}\right)\right]\\
&\quad\quad\,\,+\textstyle\min\limits_{x^*\in X^*}\sum_{k=3}^{N+1}\tfrac{(k+1)^{\beta}}{2^{\beta}}\cdot\tfrac{16}{9} \mathbb{E}\left[\tfrac{{\eta_{k-1}^2} a_{k-2}}{36n_{k-1}}\left(\tfrac{\|x^*-z_{k-1}\|^2}{a_{k-2}}+\tfrac{\|x^*-x_{k-2}\|^2}{a_{k-2}}\right)\right]\\
&\overset{\eqref{eqn:bound-N}}{\leq}\beta \textstyle\sum_{k=2}^{N+1}\tfrac{(k+1)^{\beta-1}D_0^2}{2^{\beta}}\cdot\tfrac{32}{36}\leq \tfrac{(N+2)^\beta}{2^{\beta}a_0}\cdot\tfrac{8D_0^2}{9
a_0}.
    \end{aligned}
\end{equation}
For the last term, notice that
 \begin{equation}\label{eqn:exp-p-intermedia-4-exp}
     \begin{aligned}
     &\min\limits_{x^*\in X^*}\textstyle\sum_{k=2}^{N+1}\tfrac{(k+1)^\beta \lambda}{2^{\beta}}{\mathbb{E}\left[\tfrac{9\beta^3{}{}}{{{\tau_{k-1}^2  {\tilde{c} }}}}\left(\tfrac{\|x^*-z_{k-1}\|^2}{a_{k-2}}+\tfrac{\|x^*-x_{k-2}\|^2}{a_{k-3}}\right)\right]}{{}}\\
     &\overset{\text{(iv)}}{\leq}\textstyle\sum_{k=2}^{N+1}\tfrac{(k+1)^\beta \lambda}{2^{\beta}}{\mathbb{E}\left[\tfrac{288\beta^3{}{}}{{{k^2  {\tilde{c} }}}}\tfrac{D_0^2}{a_0}\right]}{{}}\leq \tfrac{288 {\beta}\lambda D_0^2}{2^{\beta} {\tilde{c} }a_0}\tfrac{(3/2)^\beta}{(1-\beta)}\leq \tfrac{288 {\beta}\lambda D_0^2}{ {\tilde{c} }(1-\beta)a_0},
\end{aligned}
 \end{equation}
 where in (iv), we substituted $\tau_{k-1}=\tfrac{k+1-\beta}{2}\geq \tfrac{k}{2}$ and used the induction  {hypothesis}  \eqref{eqn:bound-N-y}.

Substituting \eqref{eqn:function-initial-gap-exp} -
\eqref{eqn:exp-p-intermedia-4-exp}
into \eqref{eqn:lem6-1/k},  {and choosing $\lambda=4$ in \eqref{eqn:lem6-1/k},}
we obtain
\begin{equation}\label{eqn:cor2-inter}
    \begin{aligned}
     &\tfrac{ \beta(\tau_{N}+1)}{(1+\gamma_{N+1})\Gamma_{N+1}}\mathbb{E}\left[\tfrac{ \eta_{N+1}}{a_{N+1}}(\Psi(x_{N})-\Psi(x^*))\right]+\tfrac{1}{2\Gamma_{N+1}}\min\limits_{x^*\in X^*}\mathbb{E}\left[\tfrac{\|y_{N+1}-x^*\|^2}{a_N}\right]\\
              &\leq {\min\limits_{x^*\in X^*}\mathbb{E}\left[\tfrac{\|x_{0}-x^*\|^2}{2a_0}\right]}\left[2+\left(\tfrac{N+2}{2}\right)^\beta+\tfrac{1}{4}\right]+\tfrac{\eta_1^2\| {\nabla f(x_0)+s_0}\|^2}{4a_0}\\
               &\quad+\tfrac{\beta^2 \tilde{D}^2}{12 {c }a_0
 }+\tfrac{32{\beta}(N+1)^\beta \tilde{D}^2}{2^\beta {c }a_0}+\tfrac{1152\beta D_0^2}{ {\tilde{c} }(1-\beta)a_0}+\tfrac{(N+2)^\beta}{2^{\beta}}\cdot\tfrac{8D_0^2}{36a_0}\\
 &\overset{\text{(v)}}{\leq} \tfrac{(N+2)^{\beta}}{2^\beta}\cdot\tfrac{D_0^2}{2a_0}=\tfrac{(N+2)^{\beta}}{2^\beta}\cdot\tfrac{\beta D_0^2}{2v_0\tilde{c} },
    \end{aligned}
\end{equation}
where in (v), we used  {\eqref{eqn:constant-exp-p} with $c \coloneqq 8$, $\tilde{c} \coloneqq 745$, and $\beta\leq 1/8$, together with the fact that $D_0$ satisfies \eqref{eqn:chocie D-case2}, which implies}
\begin{align*}
   \tfrac{\|x_{0}-x^*\|^2+\tilde{D}^2}{2a_0}\left[\left(\tfrac{N+2}{2}\right)^\beta+\tfrac{9}{4}\right]+\tfrac{\eta_1^2\| {\nabla f(x_0)+s_0}\|^2}{4a_0}\leq \tfrac{1}{18}\cdot\tfrac{(N+2)^{\beta}}{2^\beta}\cdot\tfrac{D_0^2}{a_0}.
\end{align*}
On the other hand, by \autoref{lem:remark-stepsize-lowerbound-beta}, we have the following lower bound on the optimality gap
\begin{align*}
    &\mathbb{E}\left[\tfrac{  \eta_{N+1}\beta_{N+1}(\tau_{N}+1)}{a_{N+1}(1+\gamma_{N+1})\Gamma_{N+1}}[\Psi(x_{N})-\Psi(x^*)]\right]+\min\limits_{x^*\in X^*}\mathbb{E}\left[\tfrac{\|y_{N+1}-x^*\|^2}{2a_N\Gamma_{N+1}}\right]\notag\\
             &\geq \tfrac{\beta^2 N(N+2)^{1+\beta}\mathbb{E}[\Psi(x_{N})-\Psi(x^*)]}{2^{\beta}32\mathcal{L}v_{\max}\tilde{c} }\tfrac{15}{16}+ \tfrac{(N+2)^{\beta}}{2^\beta}\cdot\tfrac{\beta}{2v_{\max}\tilde{c} }\min\limits_{x^*\in X^*}\mathbb{E}[\|y_{N+1}-x^*\|^2].
          \end{align*}
        Combining this with \eqref{eqn:cor2-inter} and simplifying yields the desired result.  {The deterministic part follows similarly from the proof of \autoref{cor:main-fixed-T-const}, \eqref{eqn:lowerbound-type-2}, so we omit it for simplicity. This concludes the proof.}

\end{proof}

  \subsubsection{Proof of \autoref{cor:remove variance}}

\begin{lemma}
    Suppose the Assumptions in \autoref{cor:remove variance}, then on $A_N,$
    it holds that
    \begin{align}\label{eqn:intermedia-noise-e_k-n}
\mathbb{E}\left[\textstyle\sum_{k=3}^{N+1}\tfrac{8\eta_{k-1}^2\|\Delta_k\|}{a_{k-1}\beta\Gamma_{k-1}}\right] \leq \tfrac{32{\beta}(N+1)^\beta }{2^\beta{c }}\cdot\tfrac{\tilde{D}^2}{a_0}.
    \end{align}
\end{lemma}
\begin{proof}
    Notice that
\begin{equation*}
     \begin{aligned}
         &\,\,\textstyle\sum_{k=2}^{K+1} \mathbb{E}\left[\tfrac{\eta_{k-1}^2}{a_{k-1}n_{k-1}^2\Gamma_{k-1}}{\| \textstyle\sum_{i=1}^{n_{k-1}}G(x_{k-1}, \bar{\xi}_{k-1, i})-\nabla f(x_{k-1})\|^2}\right]\\
         &\,\,\overset{\text{(i)}}{\leq}\tfrac{1}{a_{0}}\textstyle\sum_{k=2}^{K+1} \mathbb{E}\left[\tfrac{\eta_{k-1}^2}{n_{k-1}^2\Gamma_{k-1}}{\| \textstyle\sum_{i=1}^{n_{k-1}}G(x_{k-1}, \bar{\xi}_{k-1, i})-\nabla f(x_{k-1})\|^2}\right]\\
          &\overset{\eqref{eqn:m2'}}{\leq}\tfrac{1}{a_{0}}\textstyle\sum_{k=2}^{K+1} \mathbb{E}\left[\tfrac{\beta^2 \tilde{D}^2}{(k+1)c  n_{k-1}\Gamma_{k-1}}\cdot\tfrac{\| \textstyle\sum_{i=1}^{n_{k-1}}G(x_{k-1}, \bar{\xi}_{k-1, i})-\nabla f(x_{k-1})\|^2}{\hat{\delta}_{k-1}^2}\right]\\
    &\overset{\eqref{eqn:Gamma-lower}}{\leq}\tfrac{\beta^2 \tilde{D}^2 }{2^\beta c  a_0}\textstyle\sum_{k=2}^{K+1} \mathbb{E}\left[\tfrac{k^{\beta -1}}{ n_{k-1}}\cdot\tfrac{\| \textstyle\sum_{i=1}^{n_{k-1}}G(x_{k-1}, \bar{\xi}_{k-1, i})-\nabla f(x_{k-1})\|^2}{\hat{\delta}_{k-1}^2}\right]\\
     &\,\,\overset{\text{(ii)}}{=}\tfrac{\beta^2 \tilde{D}^2 }{2^\beta c  a_0}\textstyle\sum_{k=2}^{K+1} \mathbb{E}\left[\tfrac{k^{\beta -1}}{ n_{k-1}}\cdot\mathbb{E}_{\bar{\xi}_{k-1}}\left[\tfrac{\|\sum_{i=1}^{n_{k-1}}[G(x_{k-1},\bar{\xi}_{k-1,i})-\nabla f(x_{k-1})]\|^{2}}{  {\hat{\delta}_{k-1}^{2}}}\,\big|\,  {\mathcal{F}_{k-\frac{5}{3}}}\right]\right]\\
      &\,\,\overset{\text{(iii)}}{=}\tfrac{\beta^2 \tilde{D}^2 }{2^\beta c  a_0}\textstyle\sum_{k=2}^{K+1} \mathbb{E}\left[\tfrac{k^{\beta -1}}{ n_{k-1}}\cdot\mathbb{E}_{\bar{\xi}_{k-1}}\left[\tfrac{n_{k-1}\|[G(x_{k-1},\bar{\xi}_{k-1})-\nabla f(x_{k-1})]\|^{2}}{  {\hat{\delta}_{k-1}^{2}}}\,\big|\,  {\mathcal{F}_{k-\frac{5}{3}}}\right]\right]\\
           &\,\,\overset{\text{(iv)}}{=}\tfrac{\beta^2 \tilde{D}^2 }{2^\beta c  a_0}\textstyle\sum_{k=2}^{K+1} \mathbb{E}\left[{k^{\beta -1}}{ }\tfrac{\|[G(x_{k-1},\bar{\xi}_{k-1})-\nabla f(x_{k-1})]\|^{2}}{  {\hat{\delta}_{k-1}^{2}}}\right]\\
           &\,\,\overset{\text{(v)}}{\leq}\tfrac{\beta^2 \tilde{D}^2 }{2^\beta c  a_0}\textstyle\sum_{k=2}^{K+1} \mathbb{E}\left[{k^{\beta -1}}{ }\tfrac{\|[G(x_{k-1},\bar{\xi}_{k-1})-\nabla f(x_{k-1})]\|^{2}}{  {{\delta}_{k-1}^{2}}}\right]\\
    &\,\,\overset{\text{(vi)}}{=}\tfrac{\beta^2 \tilde{D}^2 }{2^\beta c  a_0}\textstyle\sum_{k=2}^{K+1} \mathbb{E}\left[{k^{\beta -1}}{ }\cdot\mathbb{E}_{\bar{\xi}_{k-1}}\left[\tfrac{\|G(x_{k-1},\bar{\xi}_{k-1})-\nabla f(x_{k-1})\|^{2}}{  {\delta_{k-1}^{2}}}\,\big|\,  {\mathcal{G}_{k-1}}\right]\right]\\
    &\overset{\eqref{eqn:local-var-2}}{\leq}\tfrac{\beta^2 \tilde{D}^2 }{2^\beta c  a_0}\textstyle\sum_{k=2}^{K+1} \mathbb{E}\left[{k^{\beta -1}}{ }\right]{\leq}\tfrac{ {\beta^2}(K+1)^\beta \tilde{D}^2}{2^\beta{c }a_0},
     \end{aligned}
\end{equation*}
where in (i), we used the monotonicity of $a_k$; in (ii) and (iv), we used the tower property together with $\hat{\delta}_{k-1}\in \mathcal{F}_{k-\frac{5}{3}}$ due to the construction of $\hat{\delta}_{k-1}$ in \eqref{eqn:pair-wise-sample} and the definition of the filtration in \eqref{def:filtration-new}; in (iii), we used the conditional i.i.d. property of $\bar{\xi}_{k-1,i}$ for all $i\in[n_{k-1}]$, together with $n_{k-1}\in \mathcal{F}_{k-\frac{5}{3}}$ and the conditional unbiasedness \autoref{a:unbiasness}, namely,
\begin{equation*}
\mathbb{E}_{\bar{\xi}_{k-1}}\!\left[G(x_{k-1},\bar{\xi}_{k-1})-\nabla f(x_{k-1})\mid \mathcal{F}_{k-\frac{5}{3}}\right]=0;
\end{equation*}
in (v), we used \eqref{eqn:high-p-event}; in (vi), we used the tower property through $\mathcal{G}_{k-1}$.

Similarly,
we have
\begin{align*}
      &\textstyle\sum_{k=1}^{K+1} \mathbb{E}\left[\tfrac{\eta_{k}^2{}}{a_{k-1}m_k^2{\Gamma_{k-1}}}\| \textstyle\sum_{i=1}^{m_k}G(x_{k-1}, \xi_{k, i})-\nabla f(x_{k-1})\|^2\right]\leq\tfrac{ {\beta}(K+1)^\beta \tilde{D}^2}{2^\beta{c }a_0},\\
      &\textstyle\sum_{k=2}^{K+1} \mathbb{E}\left[\tfrac{\eta_{k-1}^2}{a_{k-1}n_{k-1}^2\Gamma_{k-1}}{}{}\| \textstyle\sum_{i=1}^{n_{k-1}}G(x_{k-2}, \bar{\xi}_{k-1, i})-\nabla f(x_{k-2})\|^2\right]\leq\tfrac{ {\beta}(K+1)^\beta \tilde{D}^2}{2^\beta{c }a_0}.
 \end{align*}
\end{proof}

\begin{proof}[Proof of \autoref{cor:remove variance}]
Due to the exactly same choices of $\eta_k, \gamma_k, \tau_k, \beta_k$ as in \autoref{cor:main}, we can then apply the same arguments to show \eqref{eqn:Gamma-lower}-\eqref{eqn:verifiable-2} hold, which implies \eqref{eqn:beta-k} and \eqref{eqn:stepsize-3} hold.
Thus, Proposition \ref{thm:main-expectation} holds with $\gamma_k=1/k.$
 Taking expectation on both sides of \eqref{eqn:final-iterate}, it holds that
\begin{equation*}
     \begin{aligned}
&\tfrac{ \beta(\tau_{N}+1)}{(1+\gamma_{N+1})\Gamma_{N+1}}\mathbb{E}\left[\tfrac{ \eta_{N+1}}{a_{N+1}}(\Psi(x_{N})-\Psi(x^*))\right]+\tfrac{1}{2\Gamma_{N+1}}\min\limits_{x^*\in X^*}\mathbb{E}\left[\tfrac{\|y_{N+1}-x^*\|^2}{a_N}\right]\\
              &\overset{\text{(i)}}{\leq}\tfrac{(3-\beta)\beta}{3\Gamma_2}\cdot{\mathbb{E}\left[\tfrac{\eta_2\left[\Psi(x_{0})-\Psi(x^*)\right]}{{a_0}}\right]}+\min\limits_{x^*\in X^*}\mathbb{E}\left\{\tfrac{\eta_1}{4a_0}[\langle {G}_1, x^*-x_0\rangle+h(x^*)-    h(x_0)]\right\}\\
                & \quad+\tfrac{\eta_1^2\sigma^2_0}{4a_0m_1}+ \min\limits_{x^*\in X^*}{\mathbb{E}\left[\tfrac{\|x_{0}-x^*\|^2}{2a_0}\right]\left[2+\left(\tfrac{N+2}{2}\right)^\beta\right]+{\tfrac{\eta_1^2\left\| {\nabla f(x_0)+s_0}\right\|^2}{4a_0}}}+\tfrac{32{\beta}(N+1)^\beta \tilde{D}^2}{2^\beta{c }a_0}\\
&\quad+{\textstyle\sum_{k=2}^{N+1}\tfrac{(k+1)^\beta }{2^{\beta}}{\mathbb{E}\left[\left(\tfrac{9\beta^3  {\lambda}{}{}}{{{\tau_{k-1}^2  {\tilde{c} }}}}+\tfrac{\eta_k^2a_{k-2}}{36  {\lambda}n_{k-1}}\right)\left(\tfrac{\|x^*-z_{k-1}\|^2}{a_{k-2}}+\tfrac{\|x^*-x_{k-2}\|^2}{a_{k-3}}\right)\right]}{{}}},
      \end{aligned}
\end{equation*}
where in (i), we substituted \eqref{eqn:intermedia-noise-e_k-n},  \eqref{eqn:exp-innerproduct},  \eqref{eqn:normalize}, \eqref{eqn:Gamma-lower}, \eqref{eqn:Gamma-lower-2} into \eqref{eqn:final-iterate} and used $\tau_1=\tfrac{3-\beta}{2},$ $\gamma_1=\tfrac{3}{2}.$
Under the choice of $a_{k-1}$ in \eqref{eqn:max-var-n}, recall for convenience that
\begin{align*}
    a_{k-1}\coloneqq\max_{0\le i\le k-1}\left\{\tfrac{\tilde{c} v_i^2}{\beta}\right\}.
\end{align*}
We then define its sample version as
\begin{align*}
    \hat{a}_{k-1}\coloneqq\max_{0\le i\le k-1}\left\{\tfrac{\tilde{c} \hat{v}_i^2}{\beta}\right\}.
\end{align*}
Then we can rewrite the batch condition \eqref{eqn:m2'} as
\begin{equation}\label{eqn:alter-n-2}
    n_k
=
\max\left\{
1,\,
\tfrac{\tilde{c}(k+2)\eta_k^2 \hat{v}_{k-1}^{\max}}{\beta^4},\,
\tfrac{(k+2)\eta_k^2}{\beta^2}\cdot \tfrac{c(\hat{\sigma}_{k-1}^2+\hat{\delta}_k^2)}{\tilde{D}^2}
\right\}
=
\max\left\{
1,\,
\tfrac{(k+2)\eta_k^2 \hat{a}_{k-1}}{\beta^3},\,
\tfrac{(k+2)\eta_k^2}{\beta^2}\cdot \tfrac{c(\hat{\sigma}_{k-1}^2+\hat{\delta}_k^2)}{\tilde{D}^2}
\right\}.
\end{equation}
Furthermore, notice that on $A_N,$ it holds that $a_{k-2}\leq \hat{a}_{k-2}.$
Hence,
\begin{equation*}
    \begin{aligned}
          &\,\,\min\limits_{x^*\in X^*}\textstyle\sum_{k=2}^{N+1}\tfrac{(k+1)^{\beta}}{2^{\beta}}\mathbb{E}\left[\tfrac{\eta_k^2 a_{k-2}}{36n_{k-1}}\left(\tfrac{\|x^*-z_{k-1}\|^2}{a_{k-2}}+\tfrac{\|x^*-x_{k-2}\|^2}{a_{k-3}}\right)\,\middle|\,A_N\right]\\
    &\,\,\overset{\text{(ii)}}{\leq}\quad\min\limits_{x^*\in X^*}\textstyle\sum_{k=2}^{N+1}\tfrac{(k+1)^{\beta}}{2^{\beta}}\mathbb{E}\left[\tfrac{\eta_k^2 \hat{a}_{k-2}}{36n_{k-1}}\left(\tfrac{\|x^*-z_{k-1}\|^2}{a_{k-2}}+\tfrac{\|x^*-x_{k-2}\|^2}{a_{k-3}}\right)\,\middle|\,A_N\right]\\
    &\overset{\eqref{eqn:stepsize-5}}{\leq} \textstyle\min\limits_{x^*\in X^*}\tfrac{3^{\beta}\cdot 4(1-\beta)^2}{2^{\beta}}{}\mathbb{E}\left[\tfrac{{\eta_{1}^2} \hat{a}_{0}}{36n_{1}}\cdot\left(\tfrac{\|x^*-z_{1}\|^2}{a_{0}}+\tfrac{\|x^*-x_{0}\|^2}{a_{-1}}\right)\,\middle|\,A_N\right]\\
&\quad+\textstyle\min\limits_{x^*\in X^*}\sum_{k=3}^{N+1}\tfrac{(k+1)^{\beta}}{2^{\beta}}\cdot\tfrac{16}{9} \mathbb{E}\left[\tfrac{{\eta_{k-1}^2} \hat{a}_{k-2}}{36n_{k-1}}\left(\tfrac{\|x^*-z_{k-1}\|^2}{a_{k-2}}+\tfrac{\|x^*-x_{k-2}\|^2}{a_{k-3}}\right)\,\middle|\,A_N\right]\\
&\overset{\eqref{eqn:alter-n-2}}{\leq} \textstyle\min\limits_{x^*\in X^*}\tfrac{3^{\beta}\cdot 4(1-\beta)^2}{2^{\beta}}{}\mathbb{E}\left[\tfrac{{}\beta^3}{36\times 3}\cdot\left(\tfrac{\|x^*-z_{1}\|^2}{a_{0}}+\tfrac{\|x^*-x_{0}\|^2}{a_{-1}}\right)\,\middle|\,A_N\right]\\
&\quad+\textstyle\min\limits_{x^*\in X^*}\sum_{k=3}^{N+1}\tfrac{(k+1)^{\beta}}{2^{\beta}}\cdot\tfrac{16}{9} \mathbb{E}\left[\tfrac{{\beta^3} }{36\times 3}\left(\tfrac{\|x^*-z_{k-1}\|^2}{a_{k-2}}+\tfrac{\|x^*-x_{k-2}\|^2}{a_{k-3}}\right)\,\middle|\,A_N\right]\\
&\,\,\overset{\text{(iii)}}{\leq}\beta \textstyle\sum_{k=2}^{N+1}\tfrac{(k+1)^{\beta-1}D_0^2}{2^{\beta}}\cdot\tfrac{32}{36a_0}\leq \tfrac{(N+2)^\beta}{2^{\beta}}\cdot\tfrac{8D_0^2}{9
a_0},
    \end{aligned}
\end{equation*}
where in (ii), we used $a_{k-2}\leq \hat{a}_{k-2}$ on $A_N;$ and in (iii), we used the induction step.
The remaining proofs follow similarly as
\autoref{cor:main},  thus  omitted for simplicity.

\end{proof}

 \subsection{High-Probability convergence guarantees}\label{proof-high-probability}
        With Proposition \ref{thm:main-expectation} in hand, we are ready to prove \autoref{eqn:high-main-2}.       We first establish a few results under the light tail
        \autoref{ass:subgaussian}.
\paragraph{Martingale concentration bound.} 
Recall the following well-known result regarding the concentration of the martingale.
The proof of this result can be found in \citep[Lemma 2]{lan2012validation}.
\begin{lemma}\label{Lem:concentration-martingale}
    Let $\{\xi_{k,i}\}_{k\geq1, i\in m_1}$ be a sequence of i.i.d. random variables, and $\nu_k$  be deterministic Borel functions of $\{\xi_{k,i}\}_{k\geq1, i\in m_1}$ such that
  \begin{align*}
      \mathbb{E}_{\xi_k}[\nu_k|\mathcal{F}_{k-1}]=0,\quad  \mathbb{E}_{\xi_k}\left[\exp\left\{\tfrac{\nu_k^2}{\sigma_k^2}\right\}|\mathcal{F}_{k-1}\right]\leq \exp\{1\},\quad \text{a.s.}
  \end{align*}
   where $\mathbb{E}_{\xi_k}[\cdot\,|\,\mathcal{F}_{k-1}]$ denotes the expectation with respect to $\xi_{k}$ conditional on $\mathcal{F}_{k-1}$, and $0<\sigma_k<\infty.$ Then, for all $\Lambda\geq 0,$ it holds that
  \begin{align*}
     \mathbb{P}\left\{\textstyle\sum_{k=2}^{N+1}\nu_k>\Lambda \sqrt{\textstyle\sum_{k=2}^{N+1}\sigma_k^2}\right\}\leq \exp\{-\tfrac{\Lambda^2}{3}\}.
  \end{align*}
  \end{lemma}
  Define the martingale difference sequence appearing in Proposition \ref{thm:main-expectation} as follows:
\begin{align}\label{eqn:martingale-difference}
  \nu_k\coloneqq  \tfrac{\eta_k}{a_{k-1}m_{k}}\cdot\tfrac{1}{\Gamma_k(1+\gamma_k)}\textstyle\sum_{i=1}^{m_k}\left\langle G(x_{k-1}, \xi_{k, i})-\nabla f(x_{k-1}), x^*-z_{k-1}\right\rangle,
\end{align}
Next, we define the events that bound the stochastic errors appearing on the right-hand side of Proposition \ref{thm:main-expectation}. Note that the events considered in this section are cylinder sets in the product space $\Omega_{[\infty]}\coloneqq \prod_{i=1}^\infty \Omega$.
   \begin{equation*}
    \begin{aligned}
        &E_{1,K}\coloneqq \left\{  \textstyle\sum_{k=2}^{K+1}\nu_k \leq\tfrac{3\Lambda}{2  {\sqrt{c_{\Lambda}}}}\left(\tfrac{K+2}{2}\right)^{\beta}\tfrac{D_0^2}{a_{0}}\right\},\quad\\
        &E_{2,K}\coloneqq\left\{ \textstyle\sum_{k=1}^{K+1} \tfrac{\eta_{k}^2\| \textstyle\sum_{i=1}^{m_k}G(x_{k-1}, \xi_{k, i})-\nabla f(x_{k-1})\|^2{}}{a_{k-1}m_k^2{\Gamma_{k-1}}}\leq \tfrac{  {\beta}(1+\Lambda)(K+1)^\beta  {\tilde{D}^2}}{2^\beta  {c_{\Lambda}}a_0}\right\},\\
         &E_{3,K}\coloneqq\left\{ \textstyle\sum_{k=2}^{K+1} \tfrac{\eta_{k-1}^2\| \textstyle\sum_{i=1}^{n_{k-1}}G(x_{k-1}, \bar{\xi}_{k-1, i})-\nabla f(x_{k-1})\|^2}{a_{k-1}n_{k-1}^2\Gamma_{k-1}}{}\leq \tfrac{  {\beta}(1+\Lambda)(K+1)^\beta  {\tilde{D}^2}}{2^\beta  {c_{\Lambda}}a_0}\right\},\\
        &E_{4,K}\coloneqq \left\{ \textstyle\sum_{k=2}^{K+1} \tfrac{\eta_{k-1}^2\| \textstyle\sum_{i=1}^{n_{k-1}}G(x_{k-2}, \bar{\xi}_{k-1, i})-\nabla f(x_{k-2})\|^2}{a_{k-1}n_{k-1}^2\Gamma_{k-1}}{}{}\leq \tfrac{  {\beta}(1+\Lambda)(K+1)^\beta  {\tilde{D}^2}}{2^\beta  {c_{\Lambda}}a_0}\right\}.
    \end{aligned}
\end{equation*}
We next establish a convergence guarantee in probability on the event
\[
E_K = E_{1,K}\cap E_{2,K}\cap E_{3,K}\cap E_{4,K},
\]
which will be used in the proof of \autoref{eqn:high-main-2}. Although it assumes that \(z_k\) remains bounded up to iteration \(K\), this assumption is needed only as part of an induction  {hypothesis} to prove the boundedness of the next iterate \(y_{K+1}, z_{K+1}, x_{K+1}\). Hence, it can be removed in the proof of \autoref{eqn:high-main-2}.
\begin{lemma}\label{eqn:martingale-diff}
   Suppose \autoref{a:unbiasness}, \autoref{ass:finite sample convexity}, \autoref{ass:subgaussian}.
   Furthermore, suppose $m_k$ satisfies \eqref{eqn:batch-high-gamma-neq-0'}, and $\eta_k$ satisfies \eqref{eqn:stepsize-5}, $a_k$ is defined as 
   \begin{align}\label{eqn:a_k-define-high}
        a_{k-1}\coloneqq\max_{0\le i\le k-1}\left\{\tfrac{\tilde{c}_{\Lambda}v_i^2}{\beta}\right\}.
   \end{align}
   Furthermore, suppose that
   \begin{align}\label{eqn:bounded-assm}
       \max\limits_{1\leq k\leq K+1}\tfrac{\|x^*-z_{k-1}\|^2}{a_{k-1}} \leq \tfrac{4D_0^2}{a_0\beta^2}.
   \end{align}
   Then, it holds that
      \begin{equation*}
          \mathbb{P}(E_K)\geq 1-\exp\{-\tfrac{\Lambda^2}{3}\}-3\exp\{-\Lambda\}.
      \end{equation*}
  \end{lemma}
  \begin{proof}
  {\noindent\bf (a)}  For $E_{1,K},$
      by the definition of $\nu_k,$
      it is immediate to see that
      \begin{align*}
         \mathbb{E}_{\xi_k}[\nu_k|\mathcal{F}_{k-1}]
         &\overset{\text{(i)}}{=} \tfrac{\eta_k}{a_{k-1}m_{k}\Gamma_k(1+\gamma_k)}\textstyle\sum_{i=1}^{m_k}\left\langle \mathbb{E}_{\xi_k}[G(x_{k-1}, \xi_{k, i})-\nabla f(x_{k-1})|\mathcal{F}_{k-1}], x^*-z_{k-1}\right\rangle\overset{\text{A}\ref{a:unbiasness}}{=}0,
      \end{align*}
      where in (i), we substituted \eqref{eqn:martingale-difference} and
      used $\eta_k, m_k\in\mathcal{F}_{k-1}$ and $z_{k-1}\in\mathcal{F}_{k-1}.$
      Observe that
\begin{equation}\label{eqn:nu-stopped-square-bound}
\nu_k^2
\overset{\eqref{eqn:martingale-difference}}{\le}\tfrac{1}{\Gamma_k^2(1+\gamma_k)^2}\cdot\tfrac{\eta_k^2\|x^*-z_{k-1}\|^2}{a_{k-1}^2m_{k}^2}\| \textstyle\sum_{i=1}^{m_k}[G(x_{k-1}, \xi_{k, i})-\nabla f(x_{k-1})]\|^2.
\end{equation}
Define $\gamma_k^2 \coloneqq \tfrac{ \eta_k^2\sigma_{k-1}^2\|x^*-z_{k-1}\|^2}{a_{k-1}^2m_k\Gamma_k^2(1+\gamma_k)^2},$
then,  we obtain
\begin{align*}
  \mathbb E_{\xi_k}\!\left[\exp\!\left\{\tfrac{\nu_k^2}{\gamma_k^2}\right\}\Bigm|\mathcal F_{k-1}\right]
     &\overset{\eqref{eqn:nu-stopped-square-bound}}{\leq}\mathbb  E_{\xi_k}\left[\exp\left\{\tfrac{\| \textstyle\sum_{i=1}^{m_k}G(x_{k-1}, \xi_{k, i})-\nabla f(x_{k-1})\|^2}{m_k\sigma^2_{k-1}}\right\}\,\bigg|\,\mathcal{F}_{k-1}\right]\overset{\text{A\ref{ass:subgaussian}}}{\leq}\exp\{1\}.
 \end{align*}
Observe that \eqref{eqn:bounded-assm} ensures $\gamma_k^2<\infty$ for all $k=1,\ldots,K+1$.
Hence, we may apply \autoref{Lem:concentration-martingale} to conclude that, with probability at least $1-\exp\{-\tfrac{\Lambda^2}{3}\}$, there holds
\begin{equation}\label{eqn:martingale-bound-1}
     \begin{aligned}
           \textstyle\sum_{k=2}^{K+1}{\nu}_k &\,\leq \Lambda{}\sqrt{\textstyle\sum_{k=2}^{K+1}\gamma_k^2}=\Lambda{}\sqrt{\textstyle\sum_{k=2}^{K+1}\tfrac{ \eta_k^2\sigma_{k-1}^2\|x^*-z_{k-1}\|^2}{a_{k-1}^2m_k\Gamma_k^2(1+\gamma_k)^2}}\\&\overset{\text{(ii)}}{\leq} \tfrac{\Lambda \beta^2}{2  {\sqrt{c_{\Lambda}}}}\max\limits_{2\leq k\leq K+1}\tfrac{\|x^*-z_{k-1}\|^2}{\Gamma_k(1+\gamma_k)a_{k-1}}+\tfrac{\Lambda  {\sqrt{c_{\Lambda}}}}{2\beta^2}\textstyle\sum_{k=2}^{K+1}\tfrac{\eta_k^2\sigma^2_{k-1}}{\Gamma_k(1+\gamma_k)a_{k-1}m_k},
     \end{aligned}
 \end{equation}
     where (ii) follows from Young's inequality with $c_{\Lambda} >0$.
By the choice of $\gamma_k$ and $\eta_k$, it holds 
\begin{align*}
    \Gamma_k\overset{\eqref{eqn:Gamma-lower} }{=}\left(\tfrac{2}{k+1}\right)^{\beta},
\,\,
\eta_k^2&\overset{\eqref{eqn:stepsize-5}}{\le} \left[\tfrac{(k-1)(k+2-\beta)}{k^2}\right]^2\eta_{k-1}^2\leq \tfrac{81\eta_{k-1}^2}{64}.
\end{align*}
Combining it with the choice of $m_k,$
the non-decreasing property of $a_k,$ it holds that
\begin{align*}
\tfrac{\Lambda  {\sqrt{c_{\Lambda}}}}{2\beta^2}\textstyle\sum_{k=2}^{K+1}\tfrac{\eta_k^2\sigma^2_{k-1}}{\Gamma_{k-1}a_{k-1}m_k}&\overset{\eqref{eqn:batch-high-gamma-neq-0'}}{\leq}\tfrac{\Lambda  {\sqrt{c_{\Lambda}}}}{2\beta^2}\textstyle\sum_{k=2}^{K+1} \tfrac{k^{\beta-1}}{2^{\beta}} \tfrac{\beta^3 D_0^2}{  {c_{\Lambda}}a_{k-1}} \cdot\tfrac{81}{64}\leq\tfrac{81{}\Lambda D_0^2}{  {128\sqrt{c_{\Lambda}}}a_1}\left(\tfrac{K+2}{2}\right)^{\beta}.
\end{align*}
Substituting it into \eqref{eqn:martingale-bound-1}, we have
\begin{align*}
      \textstyle\sum_{k=2}^{K+1}{\nu}_k&\,\,\,\leq  \tfrac{\Lambda\beta^2}{2  {\sqrt{c_{\Lambda}}}}\max\limits_{2\leq k\leq K+1}\tfrac{\|x^*-z_{k-1}\|^2}{\Gamma_k(1+\gamma_k)a_{k-1}}+\tfrac{\Lambda  {\sqrt{c_{\Lambda}}}}{2\beta^2}\textstyle\sum_{k=2}^{K+1}\tfrac{\eta_k^2\sigma^2_{k-1}}{\Gamma_k(1+\gamma_k)a_{k-1}m_k}\\
      &\overset{\text{(iii)}}{\leq}\tfrac{\Lambda\beta^2}{2  {\sqrt{c_{\Lambda}}}}\left(\tfrac{K+2}{2}\right)^{\beta}\tfrac{4D_0^2}{a_{0}\beta^2}+\tfrac{81\Lambda D_0^2 }{128  {\sqrt{c_{\Lambda}}}a_1}\left(\tfrac{K+2}{2}\right)^{\beta}\leq\tfrac{3\Lambda}{2  {\sqrt{c_{\Lambda}}}}\left(\tfrac{K+2}{2}\right)^{\beta}\tfrac{D_0^2}{a_{0}},
\end{align*}
where in (iii), we used \eqref{eqn:bounded-assm} and $\Gamma_k$ decreasing and $\Gamma_{k}\geq \left(\tfrac{2}{K+2}\right)^{\beta}$, see \eqref{eqn:Gamma-lower}.
Hence, $\mathbb{P}(E_{1,K})\geq 1-\exp\{-{\Lambda^2}/{3}\}.$

{\noindent\bf (b)}
 For $E_{2,K},$ it holds that
 \begin{equation}\label{eqn:for-markov}
     \begin{aligned}
         &\,\,\,\,\textstyle\sum_{k=1}^{K+1} \tfrac{\eta_{k}^2{}}{a_{k-1}m_k^2{\Gamma_{k-1}}}\| \textstyle\sum_{i=1}^{m_k}G(x_{k-1}, \xi_{k, i})-\nabla f(x_{k-1})\|^2\\
 &\,\,\,\overset{\text{(iv)}}{\leq}\tfrac{1}{a_0}\textstyle\sum_{k=1}^{K+1} \tfrac{\eta_{k}^2k^{\beta}}{m_k^22^\beta}{\| \textstyle\sum_{i=1}^{m_k}G(x_{k-1}, \xi_{k, i})-\nabla f(x_{k-1})\|^2}\\
  &\overset{\eqref{eqn:batch-high-gamma-neq-0'}}{\leq}\tfrac{  {\beta^3}  {\tilde{D}^2}}{  {c_{\Lambda}}a_0}\textstyle\sum_{k=1}^{K+1} \tfrac{k^{\beta-1}}{m_k{\sigma}_{k-1}^22^\beta}{\| \textstyle\sum_{i=1}^{m_k}G(x_{k-1}, \xi_{k, i})-\nabla f(x_{k-1})\|^2},
     \end{aligned}
 \end{equation}
 where in (iv), we again used the expression for $\Gamma_k$ from \eqref{eqn:Gamma-lower}, together with the nondecreasing property of $a_k$.
 Define non-negative weights as 
 $$\theta_k=\tfrac{k^{\beta-1}}{\textstyle\sum_{\tau=1}^{K+1}\tau^{\beta-1}},\quad \forall\,\, 1\leq k\leq K+1.$$
 Then,
 by Jensen's inequality, we have
\begin{align*}
 &\mathbb{E}\left[ \exp\left\{\textstyle\sum_{k=1}^{K+1} \tfrac{\theta_k}{m_k\sigma_{k-1}^2}{{\| \textstyle\sum_{i=1}^{m_k}G(x_{k-1}, \xi_{k, i})-\nabla f(x_{k-1})\|^2}}{}\right\}\right]\notag\\
&\leq\textstyle\sum_{k=1}^{K+1}{\theta_k}\mathbb{E}\left[ \mathbb{E}_{\xi_k}\left[ \exp\left\{\tfrac{1}{m_k\sigma^2_{k-1}}{\| \textstyle\sum_{i=1}^{m_k}G(x_{k-1}, \xi_{k, i})-\nabla f(x_{k-1})\|^2}{}\right\}\,\bigg|\,\mathcal{F}_{k-1}\right]\right]\overset{\text{A\ref{ass:subgaussian}}}{\leq} \exp\{1\}.
\end{align*}
 It then follows from Markov’s inequality that for all $\Lambda>0,$  with probability at least $1-\exp(-\Lambda),$ there holds 
 \begin{equation}\label{markov-1}
     \begin{aligned}
           \textstyle\sum_{k=1}^{K+1} {}\tfrac{k^{\beta-1}}{m_k\sigma_{k-1}^2}\| \textstyle\sum_{i=1}^{m_k}G(x_{k-1}, \xi_{k, i})-\nabla f(x_{k-1})\|^2&\leq (1+\Lambda){\textstyle\sum_{k=1}^{K+1}k^{\beta-1}}\leq \tfrac{(1+\Lambda)(K+1)^\beta}{\beta}.
     \end{aligned}
 \end{equation}
 Combining \eqref{markov-1} with \eqref{eqn:for-markov}, we have
 \begin{align*}
      &\textstyle\sum_{k=1}^{K+1} \tfrac{\eta_{k}^2{}}{a_{k-1}m_k^2{\Gamma_{k-1}}}\| \textstyle\sum_{i=1}^{m_k}G(x_{k-1}, \xi_{k, i})-\nabla f(x_{k-1})\|^2\leq \tfrac{  {\beta}(1+\Lambda)(K+2)^\beta  {\tilde{D}^2}}{2^\beta  {c_{\Lambda}}a_0}.
 \end{align*}
 Hence, $\mathbb{P}(E_{2,K})\geq 1- \exp\{-\Lambda\}.$

{\noindent\bf (c)} 
 Similarly to \eqref{markov-1}, for $E_{3,K},$ by the conditional sub-Gaussian property \eqref{eqn:local-var-sub-2} from \autoref{ass:subgaussian},
  {
\begin{align*}
 &\mathbb{E}\left[ \exp\left\{\textstyle\sum_{k=2}^{K+1} \tfrac{\theta_k}{n_{k-1}\delta_{k-1}^2}{{\| \textstyle\sum_{i=1}^{n_{k-1}}G(x_{k-1}, \bar{\xi}_{k-1, i})-\nabla f(x_{k-1})\|^2}}{}\right\}\right]\notag\\
&\overset{\text{(v)}}{\leq}\textstyle\sum_{k=2}^{K+1}{\theta_k}\mathbb{E}\left[\exp\left\{\tfrac{1}{n_{k-1}\delta^2_{k-1}}{\| \textstyle\sum_{i=1}^{n_{k-1}}G(x_{k-1}, \bar{\xi}_{k-1, i})-\nabla f(x_{k-1})\|^2}{}\right\}\right]\\
&\overset{\text{(vi)}}{=}\textstyle\sum_{k=2}^{K+1}{\theta_k}\mathbb{E}\left[ \mathbb{E}_{\bar{\xi}_{k-1}}\left[\exp\left\{\tfrac{n_{k-1}\left\|\textstyle\sum_{i=1}^{n_{k-1}}[G(x_{k-1}, \bar{\xi}_{k-1, i})-\nabla f(x_{k-1})]\right\|^2}{n_{k-1}\delta^2_{k-1}}{}{}\right\}\,\bigg|\,\mathcal{F}_{k-\frac{5}{3}}\right]\right]\\
&\overset{\text{A\ref{ass:subgaussian}}}{\leq} \exp\{1\},
\end{align*}}
where in (v), we used Jensen's inequality; and in (vi), we used tower property.

 It then follows from Markov’s inequality that for all $\Lambda>0,$  with probability at least $1-\exp(-\Lambda),$ 
 it holds that
\begin{align}\label{markov-2}
    \textstyle\sum_{k=2}^{K+1} \tfrac{k^{\beta-1}}{  n_{k-1}  {\delta_{k-1}^2}}{\| \textstyle\sum_{i=1}^{n_{k-1}}G(x_{k-1}, \bar{\xi}_{k-1, i})-\nabla f(x_{k-1})\|^2}&\leq \tfrac{(1+\Lambda)(K+1)^\beta}{\beta}.
 \end{align}
 Furthermore, notice that
 \begin{equation*}
     \begin{aligned}
         &\,\,\textstyle\sum_{k=2}^{K+1} \tfrac{\eta_{k-1}^2}{a_{k-1}n_{k-1}^2\Gamma_{k-1}}{\| \textstyle\sum_{i=1}^{n_{k-1}}G(x_{k-1}, \bar{\xi}_{k-1, i})-\nabla f(x_{k-1})\|^2}\\
         &\,\,\leq\tfrac{1}{a_{0}}\textstyle\sum_{k=2}^{K+1} \tfrac{\eta_{k-1}^2}{n_{k-1}^2\Gamma_{k-1}}{\| \textstyle\sum_{i=1}^{n_{k-1}}G(x_{k-1}, \bar{\xi}_{k-1, i})-\nabla f(x_{k-1})\|^2}\\
          &\overset{\eqref{eqn:batch-high-gamma-neq-0'}}{\leq}\tfrac{1}{a_{0}}\textstyle\sum_{k=2}^{K+1} \tfrac{\beta^2  {\tilde{D}^2}}{(k+1)c_{\beta,\Lambda} n_{k-1}\Gamma_{k-1}}\cdot\tfrac{\| \textstyle\sum_{i=1}^{n_{k-1}}G(x_{k-1}, \bar{\xi}_{k-1, i})-\nabla f(x_{k-1})\|^2}{\delta_{k-1}^2}\\
    &\overset{\eqref{eqn:Gamma-lower}}{\leq}\tfrac{\beta^2  {\tilde{D}^2} }{2^\beta c  a_0}\textstyle\sum_{k=2}^{K+1} \tfrac{k^{\beta -1}}{ n_{k-1}}\cdot\tfrac{\| \textstyle\sum_{i=1}^{n_{k-1}}G(x_{k-1}, \bar{\xi}_{k-1, i})-\nabla f(x_{k-1})\|^2}{\delta_{k-1}^2}\\
    &\overset{\eqref{markov-2}}{\leq}\tfrac{  {\beta}(1+\Lambda)(K+1)^\beta  {\tilde{D}^2}}{2^\beta  {c_{\Lambda}}a_0}.
     \end{aligned}
 \end{equation*}
 Hence, $\mathbb{P}(E_{3, K})\geq 1- \exp\{-\Lambda\}.$

   {\noindent\bf (d)} 
For $E_{4, K},$
it is similar to {\noindent\bf (c)}, by Markov inequality, with probability at least $1-\exp\{-\Lambda\},$ it holds that
  \begin{align*}
     &\textstyle\sum_{k=2}^{K+1} \tfrac{\eta_{k-1}^2\| \textstyle\sum_{i=1}^{n_{k-1}}G(x_{k-2}, \bar{\xi}_{k-1, i})-\nabla f(x_{k-2})\|^2}{a_{k-1}n_{k-1}^2\Gamma_{k-1}}{}{}\leq \tfrac{  {\beta}(1+\Lambda)(K+1)^\beta  {\tilde{D}^2}}{2^\beta  {c_{\Lambda}}a_0}.
 \end{align*}
 This concludes the proof.
  \end{proof}
\begin{proof}[Proof of \autoref{eqn:high-main-2}]\label{Proof of eqn:high-main-2}
Due to the exact same choices of $\eta_k, \gamma_k, \tau_k$, and $\beta_k$ as in \autoref{cor:main}, the same arguments show that \eqref{eqn:Gamma-lower}--\eqref{eqn:verifiable-2} hold. Therefore, the conditions of Proposition \ref{thm:main-expectation} are satisfied, and thus \eqref{eqn:final-iterate} holds with $\gamma_k=1/k.$

Utilizing Proposition \ref{thm:main-expectation},
we next show that with probability at least
\begin{equation*}
1-N\exp\!\left(-\tfrac{\Lambda^2}{3}\right)-4N\exp(-\Lambda),
\end{equation*}
the following holds 
\begin{equation*}
\tfrac{\|y_{N+1}-x^*\|^2}{a_N}\leq  \tfrac{ D_0^2}{a_0},\qquad
\tfrac{\|z_{N}-x^*\|^2}{a_{N-1}}\leq  \tfrac{ 4D_0^2}{\beta^2a_0},\qquad
\tfrac{\|x_{N}-x^*\|^2}{a_{N-1}}\leq  \tfrac{ 4D_0^2}{\beta^2a_0}.
\end{equation*}
We prove  by induction.
It is immediate to see that with probability $1,$ it holds that
\begin{align}\label{eqn:induction-high-p}
   \tfrac{\|y_{1}-x^*\|^2}{a_{0}}\overset{\text{(i)}}{=}\tfrac{\|y_{0}-x^*\|^2}{a_{0}}\leq \tfrac{D_0^2}{a_0},\quad \tfrac{\|z_{0}-x^*\|^2}{a_{-1}}\leq \tfrac{ 4D_0^2}{\beta^2a_0},\quad  \tfrac{\|x_{0}-x^*\|^2}{a_{-1}}\leq \tfrac{ 4D_0^2}{\beta^2a_0},
\end{align}
due to the choice $a_{-1}=a_0,$
where in (i), we used $\beta_1=0.$
Suppose this holds for the iteration $N$, that is, on the set $S_N$, where $\mathbb{P}(S_N)\geq 1-(N-1)\exp\{-\tfrac{\Lambda^2}{3}\}-4(N-1)\exp(-\Lambda),$
it holds that
\begin{align*}
\tfrac{\|y_{N}-x^*\|^2}{a_{N-1}}\leq \tfrac{ D_0^2}{a_0},\quad \tfrac{\|z_{N-1}-x^*\|^2}{a_{N-2}}\leq \tfrac{ 4D_0^2}{\beta^2a_0}, \quad \tfrac{\|x_{N-1}-x^*\|^2}{a_{N-2}}\leq \tfrac{ 4D_0^2}{\beta^2a_0}.
\end{align*}
Then, by  the definitions of $z_N, x_N,$ it holds that
\begin{equation}\label{eqn:bound-N-high}
    \begin{aligned}
  \tfrac{\|z_{N}-x^*\|^2}{a_{N-1}}&\overset{\eqref{output-center}}{=}\left\|\tfrac{y_{N}-x^*-(1-\beta)(y_{N-1}-x^*)}{a_{N-1}\beta}\right\|^2
\leq \tfrac{2\left\|y_{N}-x^*\right\|^2}{a^2_{N-1}\beta^2}+\tfrac{2(1-\beta)^2\left\|(y_{N-1}-x^*)\right\|^2}{a^2_{N-2}\beta^2}\overset{\eqref{eqn:induction-high-p}}{\leq} \tfrac{4 D_0^2}{\beta^2a_0},\\
\tfrac{\|x_{N}-x^*\|^2}{a_{N-1}}&\overset{\eqref{eqn:output-stochastic}}{\leq}\tfrac{\|z_{N}-x^*\|^2}{a_{N-1}(1+\tau_{N})}+\tfrac{\tau_{N}\|x_{N-1}-x^*\|^2}{(1+\tau_{N})a_{N-1}}\leq\tfrac{\|z_{N}-x^*\|^2}{a_{N-1}(1+\tau_{N})}+\tfrac{\tau_{N}\|x_{N-1}-x^*\|^2}{(1+\tau_{N})a_{N-2}}\overset{\eqref{eqn:induction-high-p}}{\leq} \tfrac{4 D_0^2}{\beta^2a_0},
    \end{aligned}
\end{equation}
where the inequalities follow from  
Jensen's inequality and the non-decreasing property of $a_k$. Therefore, it remains to prove $  \tfrac{\|y_{N+1}-x^*\|^2}{a_N}\leq  \tfrac{ D_0^2}{a_0}.$

Observe that \eqref{eqn:bound-N-high} implies that the boundedness condition \eqref{eqn:bounded-assm} in \autoref{eqn:martingale-diff} holds with $K=N$. We then have
\begin{equation}
          \mathbb{P}(E_{N}^c)\leq \exp\{-\tfrac{\Lambda^2}{3}\}+3\exp\{-\Lambda\}.
      \end{equation}
Therefore, on the set $S_N\cap E_{N},$ 
it holds that
\begin{equation}\label{eqn:high-p-intermedia-1}
       \begin{aligned}
& \tfrac{  \eta_{N+1}\beta(\tau_{N}+1)[\Psi(x_{N})-\Psi(x^*)]}{a_{N+1}(1+\gamma_{N+1})\Gamma_{N+1}}+\min\limits_{x^*\in X^*}\tfrac{\|y_{N+1}-x^*\|^2}{2a_N\Gamma_{N+1}}\\
      &\overset{\text{(ii)}}{\leq }  \tfrac{(3-\beta)\beta\eta_2}{3\Gamma_2}\cdot\tfrac{\Psi(x_{0})-\Psi(x^*)}{a_0}+\tfrac{\|x_{0}-x^*\|^2}{2a_0}\left[2+\left(\tfrac{N+2}{2}\right)^\beta\right]+\tfrac{\eta_1^2\| {G_1+s_0}\|^2}{8a_0}\\
      & \quad+\min\limits_{x^*\in X^*}\tfrac{\eta_1[\langle {G}_1, x^*-x_0\rangle+h(x^*)-    h(x_0)]}{4a_0}+\tfrac{3\Lambda}{2  {\sqrt{c_{\Lambda}}}}\left(\tfrac{N+2}{2}\right)^{\beta}\tfrac{D_0^2}{a_{0}}+\tfrac{32  {\beta}(1+\Lambda)(N+1)^\beta  {\tilde{D}^2}}{2^\beta  {c_{\Lambda}}a_0}\\
              &\quad+\min\limits_{x^*\in X^*}\textstyle\sum_{k=2}^{N+1}\tfrac{(k+1)^{\beta}}{2^{\beta}}\left(\tfrac{{9n_{k-1}}\beta^2\lambda{(\tilde{L}_{k-1}-L_{k-1})^2}}{{a_{k-1}}\tau_{k-1}^2} +\tfrac{\eta_k^2 a_{k-2}}{36\lambda n_{k-1}}\right)\left(\tfrac{\|z_{k-1}-x^*\|^2}{a_{k-2}}+\tfrac{\|x_{k-2}-x^*\|^2}{a_{k-3}}\right),
      \end{aligned}
\end{equation}
where in (ii), we substituted  Proposition \ref{thm:main-expectation}, \autoref{eqn:martingale-diff}, and used $\tau_1=\tfrac{3-\beta}{2},$ $\gamma_1=\tfrac{3}{2}.$
Observe that
\begin{equation}\label{eqn:function-initial-gap}
    \begin{aligned}
       \tfrac{(3-\beta)\beta\eta_2}{3\Gamma_2}\cdot\tfrac{\Psi(x_{0})-\Psi(x^*)}{a_0}\overset{\eqref{eqn:stepsize-5},\eqref{eqn:Gamma-lower}}{\leq} \tfrac{3^{\beta-1}\beta(1-\beta)\eta_{1}}{2^{\beta-1}}{}\cdot\tfrac{\Psi(x_{0})-\Psi(x^*)}{a_0}\leq \tfrac{\eta_1[\Psi(x_{0})-\Psi(x^*)]}{4a_0}.
    \end{aligned}
\end{equation}
By the basic inequality, it holds that
\begin{equation}\label{eqn:high-p-intermedia-2}
    \begin{aligned}
         \tfrac{\eta_1^2\| {G_1+s_0}\|^2}{8a_0}&\leq \tfrac{\eta_1^2\| {\nabla f(x_0)+s_0}\|^2}{4a_0}+\tfrac{\eta_1^2\|\nabla f(x_0)-G_1\|^2}{4a_0}\\
         &\overset{\eqref{eqn:stochastic-gradient}}{=}\tfrac{\eta_1^2\| {\nabla f(x_0)+s_0}\|^2}{4a_0}+\tfrac{\eta_1^2}{4m_1^2a_0}\|\textstyle\sum_{i=1}^{m_1}[\nabla f(x_0)-G(x_{0}, \xi_{1, i})]\|^2\\
          &\overset{\text{(iii)}}{\leq}\tfrac{\eta_1^2\| {\nabla f(x_0)+s_0}\|^2}{4a_0}+\tfrac{(1+\Lambda)\beta^2  {\tilde{D}^2}}{  {12c_{\Lambda}}a_0},
    \end{aligned}
\end{equation}
where in (iii), we used $\Gamma_1=1$ and \autoref{eqn:martingale-diff} with $K=0.$ Furthermore, it holds that
\begin{equation}\label{eqn:high-p-intermedia-3}
    \begin{aligned}
       \min\limits_{x^*\in X^*}\tfrac{\eta_1}{4a_0}[\langle {G}_1, x^*-x_0\rangle+h(x^*)-    h(x_0)]
        &\overset{\text{(iv)}}{\leq} \min\limits_{x^*\in X^*} \tfrac{\eta_1\langle {G}_1-\nabla f(x_0), x_0-x^*\rangle}{4a_0}+\tfrac{\eta_1[\Psi(x^*)-    \Psi(x_0)]}{4a_0}\\
         &\overset{\text{(v)}}{\leq} \tfrac{{\eta_1^2\|\nabla f(x_0)-{G}_1\|^2}}{8a_0}+\min\limits_{x^*\in X^*}\tfrac{{\| x_0-x^*\|^2}}{8a_0}+\tfrac{\eta_1[\Psi(x^*)-    \Psi(x_0)]}{4a_0},
    \end{aligned}
\end{equation}
where in (iv), we used the convexity of $\Psi;$ and in (v), we used Young inequality.

We proceed with bounding the last two terms in \eqref{eqn:high-p-intermedia-1}.
Notice that
\begin{align}\label{eqn:eqn:event-5}
    \tfrac{{n_{k-1}}{(\tilde{L}_{k-1}-L_{k-1})^2}}{{a_{k-1}}} \overset{\eqref{eqn:a_k-define-high}}{\leq}  \tfrac{\beta{n_{k-1}}{(\tilde{L}_{k-1}-L_{k-1})^2}}{{{  {\tilde{c}_{\Lambda}}v_{k-1} }}}\overset{\eqref{eqn:bar-L-k}}{=}\tfrac{  {\beta}|\sum_{i=1}^{n_{k-1}}[\ell_{k-1}(\hat{\xi}_{k-1, i})-L_{k-1}]|^{2}}{  {\tilde{c}_{\Lambda}}n_{k-1}v_{k-1} }.
\end{align}
By \autoref{ass:subgaussian}, \eqref{eqn:local-var-sub2} and the 
Markov inequality, there exists a set $F_{N+1},$ such that
 $\mathbb{P}(F_{N+1}^c)\leq (N+1)\exp(-\Lambda),$ on $F_{N+1}$
it holds that
\begin{align}\label{eqn:event-5}
    \tfrac{|\sum_{i=1}^{n_{k-1}}[\ell_{k-1}(\hat{\xi}_{k-1, i})-L_{k-1}]|^{2}}{n_{k-1}v_{k-1} }\leq (1+\Lambda), \quad\forall\, k=1,\dots, N+1.
\end{align}
Therefore, on $F_{N+1},$ it holds that
 \begin{equation}\label{eqn:high-p-intermedia-4}
     \begin{aligned}
&\quad\quad\min\limits_{x^*\in X^*}\textstyle\sum_{k=2}^{N+1}\tfrac{(k+1)^{\beta}}{2^{  {\beta}}}\tfrac{{9n_{k-1}}\beta^2{(\tilde{L}_{k-1}-L_{k-1})^2}}{{a_{k-1}}\tau_{k-1}^2} \left(\tfrac{\|z_{k-1}-x^*\|^2}{a_{k-2}}+\tfrac{\|x^*-x_{k-2}\|^2}{a_{k-3}}\right)\\
&\quad\quad\overset{\text{(vi)}}{\leq}\textstyle\sum_{k=2}^{N+1}\tfrac{(k+1)^{\beta}}{2^{\beta}}\tfrac{{36n_{k-1}}\beta^2{(\tilde{L}_{k-1}-L_{k-1})^2}}{{a_{k-1}}k^2} \tfrac{8D_0^2}{\beta^2a_0}\\
&\,\,\,\overset{\eqref{eqn:eqn:event-5},\eqref{eqn:event-5}}{\leq}\tfrac{(1+\Lambda)  {\beta}D_0^2}{2^{\beta}  {\tilde{c}_{\Lambda}}a_0}\textstyle\sum_{k=2}^{N+1}\tfrac{{288(k+1)^{\beta}}{}}{{}k^2} \leq \tfrac{288(1+\Lambda)  {\beta}D_0^2}{2^{\beta}  {\tilde{c}_{\Lambda}}a_0}\tfrac{(3/2)^\beta}{(1-\beta)}\leq \tfrac{288(1+\Lambda)  {\beta}D_0^2}{  {\tilde{c}_{\Lambda}}(1-\beta)a_0},
\end{aligned}
 \end{equation}
 where in (vi), we substituted $\tau_{k-1}=\tfrac{k+1-\beta}{2}\geq \tfrac{k}{2}$ and used the induction  {hypothesis} \eqref{eqn:bound-N-high}.
We proceed with bounding the last term in \eqref{eqn:high-p-intermedia-1}. By  the choice of $a_k$ in \eqref{eqn:a_k-define-high}, recall here for convenience, 
   \begin{align}
        a_{k-1}\coloneqq\max_{0\le i\le k-1}\left\{\tfrac{\tilde{c}_{\Lambda}v_i^2}{\beta}\right\}.
   \end{align}
 we then can rewrite the batch condition \eqref{eqn:batch-high-gamma-neq-0'} as
\begin{equation}\label{eqn:alter-n-3}
    n_k
=
\max\left\{
1,\,
\tfrac{\tilde{c}_{\Lambda}(k+2)\eta_k^2 v_{k-1}^{\max}}{\beta^4},\,
\tfrac{(k+2)\eta_k^2}{\beta^2}\cdot \tfrac{c_{\Lambda}(\sigma_{k-1}^2+\delta_k^2)}{\tilde{D}^2}
\right\}=\max\left\{
1,\,
\tfrac{(k+2)\eta_k^2 a_{k-1}}{\beta^3},\,
\tfrac{(k+2)\eta_k^2}{\beta^2}\cdot \tfrac{c_{\Lambda}(\sigma_{k-1}^2+\delta_k^2)}{\tilde{D}^2}
\right\}.
\end{equation}
Hence, similar to \eqref{eqn:in-exp-intermedia-5}, using the induction  {hypothesis} \eqref{eqn:bound-N-high}, the stepsize condition \eqref{eqn:stepsize-5} and the batch size condition \eqref{eqn:alter-n-3},
it holds that
\begin{equation}\label{eqn:high-p-intermedia-5}
    \begin{aligned}
&\quad\,\,\,\,\min\limits_{x^*\in X^*}\textstyle\sum_{k=2}^{N+1}\tfrac{(k+1)^{\beta}}{2^{\beta}}\cdot\tfrac{\eta_k^2 a_{k-2}}{36n_{k-1}}\left(\tfrac{\|z_{k-1}-x^*\|^2}{a_{k-2}}+\tfrac{\|x^*-x_{k-2}\|^2}{a_{k-3}}\right)\leq
\tfrac{(N+2)^\beta}{2^{\beta}}\cdot{\tfrac{8D_0^2}{9a_0}}.
    \end{aligned}
\end{equation}
Define $S_{N+1}\coloneqq S_N\cap E_N\cap F_{N+1}$, then $\mathbb{P}(S_{N+1})\geq1-(N+1)\exp\{-\tfrac{\Lambda^2}{3}\}-4(N+1)\exp\{-\Lambda\}.$ On \(S_{N+1}\), by substituting
\eqref{eqn:function-initial-gap},
\eqref{eqn:high-p-intermedia-2},
\eqref{eqn:high-p-intermedia-3},
\eqref{eqn:high-p-intermedia-4}, and
\eqref{eqn:high-p-intermedia-5}
into \eqref{eqn:high-p-intermedia-1},  {and choosing $\lambda=4$ in \eqref{eqn:high-p-intermedia-1},}
we obtain
\begin{equation}\label{eqn:upper-final}
    \begin{aligned}
        &\tfrac{  \eta_{N+1}\beta_{N+1}(\tau_{N}+1)}{a_{N+1}(1+\gamma_{N+1})\Gamma_{N+1}}[\Psi(x_{N})-\Psi(x^*)]+\min\limits_{x^*\in X^*}\tfrac{\|y_{N+1}-x^*\|^2}{2a_N\Gamma_{N+1}}\\
              &\leq\min\limits_{x^*\in X^*}\tfrac{\|x_{0}-x^*\|^2}{2a_0}\left[2+\left(\tfrac{N+2}{2}\right)^\beta+\tfrac{1}{4}\right]+\tfrac{\eta_1^2\| {\nabla f(x_0)+s_0}\|^2}{4a_0}\\
               &\quad+\tfrac{\beta^2(1+\Lambda) \tilde{D}^2}{  {8c_{\Lambda}}a_0}+\tfrac{32  {\beta}(1+\Lambda)(N+1)^\beta \tilde{D}^2}{2^\beta  {c_{\Lambda}}a_0}+\tfrac{3\Lambda}{2  {\sqrt{c_{\Lambda}}}}\left(\tfrac{N+2}{2}\right)^{\beta}\tfrac{D_0^2}{a_{0}}\\
              &\quad+\tfrac{1152(1+\Lambda)D_0^2 }{  {\tilde{c}_{\Lambda}}(1-\beta)a_0}+\tfrac{(N+2)^\beta}{2^{\beta}}\cdot\tfrac{8D_0^2}{36a_0}\overset{\text{(vii)}}{\leq} \tfrac{(N+2)^{\beta}}{2^\beta}\cdot\tfrac{D_0^2}{2a_0}\overset{\eqref{eqn:a_k-define-high}}{=}\tfrac{(N+2)^{\beta}}{2^\beta}\cdot\tfrac{\beta D_0^2}{2v_0\tilde{c}_{\Lambda}},
    \end{aligned}
\end{equation}
where in (vii), we used the choices of $c_{\Lambda}$ and $\tilde{c}_{\Lambda}$ in \eqref{eqn:constant-high-p}, and the fact that $D_0$ satisfies \eqref{eqn:chocie D-case2}, 
which implies
\begin{align*}
  \tfrac{\|x_{0}-x^*\|^2+\tilde{D}^2}{2a_0}\left[\left(\tfrac{N+2}{2}\right)^\beta+\tfrac{9}{4}\right]+\tfrac{\eta_1^2\| {\nabla f(x_0)+s_0}\|^2}{4a_0}\leq \tfrac{1}{18}\cdot\tfrac{(N+2)^{\beta}}{2^\beta}\cdot\tfrac{D_0^2}{a_0}.
\end{align*}
On the other hand, by \autoref{lem:remark-stepsize-lowerbound-beta}, we have the following lower bound on the optimality gap
      \begin{align*}
&\tfrac{  \eta_{N+1}\beta_{N+1}(\tau_{N}+1)}{a_{N+1}(1+\gamma_{N+1})\Gamma_{N+1}}[\Psi(x_{N})-\Psi(x^*)]+\min\limits_{x^*\in X^*}\tfrac{\|y_{N+1}-x^*\|^2}{2a_N\Gamma_{N+1}}\notag\\
             &\geq \tfrac{\beta^2 N(N+2)^{1+\beta}[\Psi(x_{N})-\Psi(x^*)]}{2^{\beta}\cdot32\cdot \tilde{c}_{\Lambda}\hat{L}_{N}v^{\max}_{N+1}}\tfrac{15}{16}+\tfrac{(N+2)^{\beta}}{2^\beta}\cdot\tfrac{\beta }{2v^{\max}_{N}\tilde{c}_{\Lambda}}\min\limits_{x^*\in X^*}\|y_{N+1}-x^*\|^2.
          \end{align*}
Combining it with \eqref{eqn:upper-final} yields the desired result for the stochastic case.  {The deterministic part follows similarly from the proof of \autoref{cor:main-fixed-T-const} and \eqref{eqn:lowerbound-type-2}, so we omit it for simplicity. This concludes the proof.}

   \end{proof}
\section{Concluding remarks}\label{sec:Concluding remarks}
In this paper, we develop stochastic AC-FGM, an optimal parameter-free method that is adaptive to the Lipschitz smoothness constant, the iteration limit, and the underlying variance. The method permits both adaptive stepsize selection and adaptive mini-batch sizing, while achieving optimal iteration and sample complexity for \eqref{eqn:main-p} without assuming either a bounded domain or bounded gradients, or resorting to stochastic line-search procedures.

Moreover, the filtration framework and the adaptive stepsize and mini-batch rules underlying our analysis are sufficiently general to accommodate a broader class of accelerated adaptive stochastic methods, thereby laying a foundation for accelerated parameter-free stochastic optimization. Our framework also opens several interesting directions for future work. In particular, it would be interesting to generalize stochastic AC-FGM beyond \autoref{ass:finite sample convexity} and extend it to the nonconvex setting. Removing the dependence on the initial optimality gap is another interesting problem.
Moreover, the local cocoercivity parameter $\bar{L}_{k-1}$ is not an unbiased estimator of its deterministic counterpart. However, as shown in this work, the resulting error can still be controlled through the fluctuation of the sample local smoothness estimator $\tilde{L}_{k-1}$ around its mean ${L}_{k-1}$. This also indicates that the variance of the local smoothness estimator plays an important role in the analysis. It remains an open problem to remove the constant dependence on the largest sample Lipschitz smoothness variance $v_{\max}$ in the in-expectation bound (resp. on $v^{\max}_{N+1}$ in the high-probability bound) for the final iteration and sample complexity guarantees. Lastly, as in the deterministic case, the method works only in the Euclidean geometry, and extending it to the non-Euclidean setting appears to be challenging.

\bibliographystyle{sn-mathphys-num}
\bibliography{references}

\end{document}